\documentclass{aptpub}
\usepackage[colorlinks = true,linkcolor = blue,urlcolor  = blue,citecolor = blue,anchorcolor = blue]{hyperref}
\authornames{F.~BAERISWYL AND O.~WINTENBERGER} 
\shorttitle{Sample‑Path Large Deviations for Functionals of PCP} 

\usepackage[round,authoryear]{natbib}
\usepackage{tikz}
\usetikzlibrary{arrows.meta, positioning}
\usepackage{pgfplots}
\usepackage{subcaption}
\usepackage{enumitem}  
\usepackage{hyperref}  

\pgfplotsset{compat=1.17} 
\providecommand{\A}{}

\providecommand{\D}{}

\providecommand{\R}{}

\providecommand{\M}{}
\providecommand{\N}{}

\providecommand{\P}{}
\providecommand{\T}{}

\renewcommand{\A}{\mathbb{A}}
\renewcommand{\D}{\mathbb{D}}
\renewcommand{\R}{\mathbb{R}}

\renewcommand{\N}{\mathbb{N}}

\renewcommand{\M}{\mathbb{M}}

\renewcommand{\T}{\mathbb{T}}

\renewcommand{\P}{\mathbb{P}}

\newcommand{\norm}[1]{\left\lVert#1\right\rVert}
\newcommand*\diff{\mathop{}\!\mathrm{d}}

\usepackage{mathtools}
\usepackage{bbm}

\newcommand{\E}[1]{{\mathbb E}\left[#1\right]}		

\newcommand{\V}[1]{{\text{Var}}\left[#1\right]}

\newcommand{\p}[1]{{\mathbb P}\left(#1\right)}

\newcommand{\I}[1]{{\mathbf 1}_{\{#1\}}}

\newcommand{\defeq}{\vcentcolon=}
\newcommand{\eqdef}{=\vcentcolon}

\newcommand\cB{\mathcal B}
\newcommand\cC{\mathcal C}

\newcommand\cF{\mathcal F}
\newcommand\cG{\mathcal G}

\newcommand\cI{\mathcal I}

\newcommand\cK{\mathcal K}

\newcommand\cM{\mathcal M}
\newcommand\cN{\mathcal N}
\newcommand\cO{\mathcal O}
\newcommand\cP{\mathcal P}

\usepackage{graphicx}
\newcommand\smallO{
  \mathchoice
    {{\scriptstyle\mathcal{O}}}
    {{\scriptstyle\mathcal{O}}}
    {{\scriptscriptstyle\mathcal{O}}}
    {\scalebox{.7}{$\scriptscriptstyle\mathcal{O}$}}
  }

 \newcommand\mfM{{\mathfrak{M}}} 
  \newcommand\mfK{{\mathfrak{K}}} 
  
\newcommand{\convv}{\xrightarrow{\text{\upshape\tiny v}}}

\DeclareUnicodeCharacter{2011}{\nobreakdash-}

\begin{document}

\title{Sample-Path Large Deviations for Functionals of Poisson Cluster Processes} 

\authorone[LPSM, Sorbonne Université \& DO, Université de Lausanne]{Fabien Baeriswyl} 
\authortwo[LPSM, Sorbonne Université]{Olivier Wintenberger} 

\addressone{Quartier UNIL-Chamberonne, Bâtiment Anthropole, Bureau 3090.1, CH-1015 Lausanne} 
\emailone{fabien.baeriswyl@sorbonne-universite.fr} 

\addresstwo{LPSM (Bureau 15-16 206), Sorbonne Université, 4 place Jussieu, 75005, Paris, FRANCE} 
\emailtwo{olivier.wintenberger@sorbonne-universite.fr} 

\begin{abstract}
We establish sample-path large deviation principles for the centered cumulative functional of marked Poisson cluster processes in the Skorokhod space equipped with the $M_1$ topology, under joint regular variation assumptions on the marks and the offspring distributions governing the propagation mechanism. These findings can also be interpreted as hidden regular variation of the cluster processes’ functionals, extending the results in \cite{dtw22} to cluster processes with heavy-tailed characteristics, including mixed Binomial Poisson cluster processes and Hawkes processes. Notably, by restricting to the adequate subspace of measures on $\D([0,1], \R_+)$, and applying the correct normalization and scaling to the paths of the centered cumulative functional, the limit measure concentrates on paths with multiple large jumps. 
\end{abstract}



\section{Introduction}\label{section:intro}

The study of regular variation, as an analytical concept describing the asymptotic behaviour of functions, originates in the early work of \cite{landau11, polya17}, and was thoroughly developed and formalised by Jovan Karamata in the 1930s (see \cite{karamata33}). Regular variation plays a crucial role in analytic number theory and modern probability theory, as emphasised in the monograph of \cite{bgt89}. In extreme value theory, it helps characterising the Fréchet maximum domain of attraction; see e.g. \cite{resnick87}, \cite{resnick07}. It also has important applications in risk theory and insurance mathematics, as underlined in \cite{ekm13}, \cite{mikosch09}.

In its simplest form, regular variation is defined as follows: let $X$ be an $\mathbb{R}^d$-valued random vector defined on a probability space $(\Omega, \mathcal{F}, \mathbb{P})$. $X$ is said to be regularly varying if there exists a sequence $\{a_n\}$, with $a_n \rightarrow \infty$, such that
$$ n \p{a_n^{-1} X \in \cdot} \convv \mu(\cdot), \quad \text{as } n \to \infty,$$
where the limiting measure $\mu(\cdot)$ is a Radon measure, i.e., a Borel measure assigning finite mass to compact sets of $\R^d$.  

The convergence is understood in the vague topology on $ \R^d \setminus \{\bf {0}\} $, excluding the origin due to the possible singularity at zero. The limit measure $\mu(\cdot)$ necessarily satisfies an $\alpha$-homogeneity property for some $\alpha > 0$: that is,
$$\mu(u \cdot) = u^{-\alpha} \mu(\cdot), \quad \text{for all } u > 0,$$ which justifies the polynomial decay rate quantifying tail sets, such as the complementary of $[0, r]^d$. This makes regular variation a natural tool for modelling the heavy-tailed behaviour of random variables and vectors.

It is natural to extend the notion of regular variation beyond simple random variables and vectors to more complex objects in infinite-dimensional spaces. For point measures, this was introduced as early as \cite{hl06}, who laid the foundational framework for $M_0$-convergence of measures on metric spaces. For stationary time series, this was studied in \cite{bs09} via the introduction of the tail process, and later in \cite{dhs18}, which clarified its relation to the tail measure and established a one-to-one correspondence. See also the comprehensive monographs \cite{ks20,mw24} for regularly varying time series. 

Regular variation of marked point processes is the main focus of \cite{dtw22}. Recently, a general formulation of regular variation on topological spaces has been proposed in \cite{bmm25}, which unifies and extends earlier approaches by identifying the key structural properties needed for regular variation in abstract settings.

One may ask whether it is possible to remove larger subsets of the state space than just the origin $\{\bf 0\}$, as originally introduced when dealing with $M_0$-convergence. The answer is positive, and this idea was first formalised in \cite{lrr14} for metric spaces, giving rise to the notion of hidden regular variation, which emerges when regular variation concentrated on the axes ``hides'', in fact, regular variation away from the axes at a finer scale. An early account of the phenomenon can be found in \cite{resnick02}, and for a modern textbook treatment with many examples, see \cite{resnick24}. Hidden regular variation of Lévy processes with regularly varying increments was developed in \cite{rbz19}, where connections with classical sample-path large deviations are made explicit. For Lévy processes with Weibull increments, see \cite{bbrz20}. In the context of dynamical systems with heavy-tailed perturbations, sample-path deviations are obtained in \cite{wr23}. Hidden regular variation of cluster sizes in multivariate Hawkes processes was obtained in \cite{blwz25}, where the authors study multitype branching representations and show how multiple large jumps can jointly contribute to extreme cluster sizes. In a very recent work, \cite{b2025}, building on \cite{blwz25}, establish a sample path large deviation principle for Lévy processes driven by multivariate heavy-tailed Hawkes processes. 

In this work, we extend the hidden regular variation framework developed in \cite{dtw22} to the setting of Poisson cluster processes. These processes are constructed from a base (immigration) Poisson point process, where each point may generate a random number of offspring points, forming a branching structure. The primary examples we consider are the mixed Binomial Poisson cluster process (with a single generation of offsprings) and the Hawkes process (with multiple generations); see Example~6.3 and the related discussions in \cite{dvj03,dvj08} for a thorough overview.

Our focus is on the regular variation properties of centered cumulative functionals associated with marked Poisson cluster processes, which can be interpreted as sample path large deviation principles. Our analysis is based on a joint regular variation assumption on the drivers of the propagation mechanism of the marked clusters. Building on the recent results of \cite{bcw23}, we establish a functional extension of the large deviations stated in Proposition~10 of that paper, specifically for partial sums of the marks of the Poisson cluster processes. This transition from scalar to functional settings introduces stronger assumptions on the distribution of offspring waiting times. However, these conditions also enable us to retrieve and extend some results in \cite{bcw23}, notably proving a non-functional version of Proposition~10 therein. 

These results contribute to the broader body of limiting results for Poisson cluster processes, an active topic for several decades. Early asymptotic results appear in \cite{westcott73}, while large deviation principles are developed in \cite{bt07}, and functional central limit theorems as well as moderate deviation principles are studied in \cite{gw20}. Central limit theorems for Poisson cluster processes were also studied in \cite{bwz19}, under a range of regular variation assumptions on their propagation mechanisms. 

Among Poisson cluster processes, the Hawkes process -- first introduced in \cite{hawkes71} -- has received particular attention due to its wide range of applications. In seismology, it models aftershock sequences (see \cite{ogata88}); in finance, it captures features like volatility clustering and extreme losses, with applications such as Value-at-Risk estimation (see \cite{cd05}). Large deviation principles have been derived for the linear Hawkes process in \cite{bt07}, extended to nonlinear variants in \cite{zhu13}, and to marked versions in \cite{kz15}. Functional limit theorems for multivariate Hawkes processes have been developed in \cite{bdh13}, providing Gaussian approximations useful for statistical inference.

The paper is structured as follows: in Section~\ref{section:not}, we introduce Poisson cluster processes, (hidden) regular variation in the space of measures, and Skorokhod’s $M_1$ topology. Section~\ref{section:aggregated} presents slight extensions of the regular variation principles established in \cite{dtw22} for Poisson point processes and their summation functionals, applying them to compound versions of the cluster processes. In Section~\ref{section:res}, we establish an $M_1$ approximation for a general centered cumulative functional of the cluster marks, leading to a sample path large deviation principle that can be interpreted as hidden regular variation. In Section~\ref{section:check}, we verify that the two main submodels of Poisson cluster processes from Subsection~\ref{section:intropp} satisfy the general assumptions and conditions identified in Section~\ref{section:res} for the regular variation properties to transfer to cluster processes. The proofs of Subsection~\ref{subsection:slutsky} and Section~\ref{section:res} are given in Sections~\ref{section:proof2} and \ref{section:proof4}, respectively.


\section{Notations, Measure Spaces, Processes of Interest, and Preliminaries}\label{section:not}

\subsection{General notations}
We write $f(x) = \smallO(g(x))$ if, for two functions $f(\cdot), g(\cdot)$, it holds that
$$
\lim_{x \rightarrow \infty} \frac{|f(x)|}{|g(x)|} = 0.
$$
We write $f(x) = \cO(g(x))$ if there exist $C > 0$ and $x_0$ such that $|f(x)| \le C |g(x)|$ for all $x \ge x_0$, and $f(x) \sim g(x)$ if
$$
\lim_{x \rightarrow \infty} \frac{f(x)}{g(x)} = 1.
$$

We write $X_n = \cO_p(a_n)$ for a sequence of random variables $\{X_n\}_{n \ge 1}$ if, for every $\varepsilon > 0$, there exist $C > 0$ and $n_0 \in \N$ such that
$$
\mathbb{P}(|X_n| > C a_n) < \varepsilon\,, \qquad \text{for all } n \ge n_0.
$$
We write $X_n = \smallO_p(a_n)$ if
$$
\frac{X_n}{a_n} \rightarrow 0 \text{ in probability}, \, \text{ as } n \rightarrow \infty.
$$

$\text{Leb}(\cdot)$ denotes the Lebesgue measure. $\delta_x(\cdot)$ denotes the Dirac measure at $x$, i.e. $\delta_x(A) = 1$ if $x \in A$ and $\delta_x(A) = 0$ otherwise. We let $(\Omega, \cF, \P)$ denote a general probability space, usually supporting all of our objects except if specified differently.

\subsection{Measure spaces and general point processes}\label{subsection:measpp}

We introduce measure spaces and point processes as in Section 2 in \cite{lrr14} and Section 2 in \cite{dtw22}, but we keep it brief and encourage readers to redirect to the aforementioned literature for more details. 

In full generality, we let $(E, \cB(E),  d)$ be a complete, separable metric space, where $\cB(E)$ is the Borel $\sigma-$algebra on $E.$ The open ball at $x$ of radius $r$ is denoted by $B_{x, r} \defeq \{y \in E: d(x, y) < r\}.$ For any closed set $F \subset E$, let $F^r \defeq \{x \in E: d(x, F) < r\},$ where $d(x, F) \defeq \inf_{y \in F} d(x, y).$ For a set $A \subset E$, $d(A, F) = \inf_{x \in A, y \in F} d(x, y). $  The class of bounded, continuous functions on $E$ is denoted by $\cC_b(E).$ Finite Borel measures on $E$ are denoted by $\M_b(E).$  $\cC(E \setminus F) \subset \cC_b(E \setminus F)$ is the subclass of bounded, continuous functions such that each element in $\cC(E \setminus F)$ vanishes on $F^r$ for some $r > 0$. $\M(E \setminus F)$ denotes the class of Borel measures on $E \setminus F$ whose restrictions to $E \setminus F^r$ are finite for each $r > 0.$ Convergence in $\M(E \setminus F)$ is defined as follows.
\begin{defn}[Section 2 in \cite{lrr14}]\label{def:convm}
    A sequence $\{\mu_n\}_{n \ge 1}$ converges to $\mu \in \M(E \setminus F)$ if, for all $f \in \cC(E \setminus F)$, $$\int_{E \setminus F} f(x)  \mu_n(\diff x) \rightarrow \int_{E \setminus F} f(x) \mu(\diff x), \text{ as } n \rightarrow \infty.$$ We write $\mu_n \rightarrow \mu$ in $\M(E \setminus F)$, as $n \rightarrow \infty.$ 
\end{defn}
\begin{rem}
      Let $\mu^{(r)}, \nu^{(r)}$ be the finite restrictions of the measures $\mu, \nu$ to $E \setminus F^{r}.$ In Theorem 2.2 in  \cite{lrr14}, the authors show that the metric $\rho_{\M(E \setminus F)}(\cdot, \cdot)$, defined by $$\rho_{\M(E \setminus F)}(\mu, \nu) \defeq \int_{0}^{\infty} e^{-r} \rho(\mu^{(r)}, \nu^{(r)}) \big(1+\rho(\mu^{(r)}, \nu^{(r)})\big)^{-1} \diff r,$$ metrizes convergence in $\M(E \setminus F)$, where $\rho(\cdot, \cdot)$ is the Prokhorov metric. 
\end{rem}

Given $E$ and a closed subset $F \subset E$, it is possible to define the closed subset $\cN(E \setminus F) \subset \M(E \setminus F)$ (abbreviated $\cN$ for brevity) of $\N$-valued measures that are point measures, i.e. if $\pi(\cdot) \in \cN(E \setminus F)$, then it admits representation $\pi(\cdot) = \sum_{i \in I} \delta_{x_i}(\cdot)$ for $I$ a countable index set, and $x_i \in E \setminus F.$ The set of point measures with at most $k$ points is denoted by $$\cN_k(E \setminus F) \defeq \Big\{\pi(\cdot) = \sum_{i=1}^p \delta_{x_i}(\cdot): 0 \le p \le k, \, x_1, \ldots, x_p \in E \setminus F \Big\}.$$

While we do not discuss this in full details, convergence in $\M(\cN(E \setminus  F) \setminus \cN_k(E \setminus F))$ -- written $\M(\cN \setminus \cN_k)$ for brevity -- which is convergence of Borel measures on the subset consisting of point measures with at least $k+1$ points, is characterised by the associated convergence of modified Laplace functionals of the point processes, see Theorem 2.5 in \cite{dtw22}. 

\subsection{Poisson cluster process and submodels}\label{section:intropp}
Poisson cluster processes are natural extensions of Poisson processes, in which events occur in clusters. For this family of processes, an immigration or background Poisson process triggers offspring processes, which can further -- in multi-generational models -- trigger new offsprings. We formally introduce the general marked Poisson cluster process, keeping the spirit of the presentation and notations from \cite{bwz19}. 

To do so, let the complete, separable metric space in Subsection~\ref{subsection:measpp} be specialised to $E \defeq \R_+ \times \A$, where $(\A, \cB(\A))$ is a measurable space for the marks endowed with its Borel $\sigma-$algebra $\cB(\A)$. We recall that a point process in $E$, defined on $(\Omega, \cF, \P)$, is a random point measure, i.e. a measurable map $N(\cdot): (\Omega, \cF, \P) \rightarrow (\cN(E), \cM_{\cN}(E))$, where $\cM_{\cN}(E)$ is the trace $\sigma-$algebra induced by $\cM(E)$, the $\sigma-$algebra generated by the integration maps $\mu \rightarrow \mu f \defeq \int f \diff \mu$, for nonnegative, measurable $f(\cdot)$ on $E$. 

We first introduce the temporal shift operator, which is an operator solely acting on the temporal coordinate of a marked point measure. 

\begin{defn}
    For any point measure $m(\cdot) = \sum_{n \in I} \delta_{t_n, x_n}(\cdot) \in \cN(E)$, for some countable index set $I$, the temporal-shift operator $\theta_t$ is defined by $$\theta_t m(\cdot) = \sum_{n \in I} \delta_{t_n + t, x_n}(\cdot).$$
\end{defn}

\begin{defn}\label{def:pcpp}
         Let $N(\cdot) \in \cN(E)$ be the marked point process, defined on $(\Omega, \cF, \P)$, as the superposition $$N(\cdot) \defeq \bigoplus_{i=1}^{\infty} (\delta_{\Gamma_i, A_{i0}}(\cdot) + \theta_{\Gamma_i}G_{A_{i0}}(\cdot)),$$ where 
    \begin{itemize}
        \item[1.] $N_0(\cdot) \defeq \sum_{i=1}^{\infty} \delta_{\Gamma_i, A_{i0}}(\cdot) \in \cN(E)$ is a marked, homogeneous Poisson point process in $E$, defined on $(\Omega, \cF, \P)$, with intensity measure $\nu_{N_0}(\cdot) \defeq (\lambda \text{Leb} \otimes F)(\cdot)$, for $\lambda > 0$, $\{A_{i0}\}_{i \ge 1}$ an i.i.d. sequence with common distribution $F(\cdot),$ independent from $\{\Gamma_i\}_{i \ge 1};$ and
        \item[2.] for each $i \ge 1,$ $G_{A_{i0}}(\cdot) \in \cN(E)$ is a marked point process in $E$, defined on $(\Omega, \cF, \P)$, with representation $$G_{A_{i0}}(\cdot) \defeq \sum_{j=1}^{K_i} \delta_{W_{ij}, A_{ij}}(\cdot),$$ and such that $K_i$ an $\N-$valued random variable with $\E{K_i} < \infty$, $\{W_{ij}\}_{i \ge 1, j \ge 1}$ is a nonnegative sequence of r.v.s., and $\{A_{ij}\}_{j \ge 0}$ is an i.i.d. sequence, with common distribution $F(\cdot)$. 
    \end{itemize}
    $N(\cdot) \in \cN(E)$ is called a Poisson cluster process in $E.$
\end{defn}

\begin{rem}
    The Poisson cluster process in Definition~\ref{def:pcpp} admits full representation $$N(\cdot) = \sum_{i=1}^{\infty} \sum_{j=0}^{K_i} \delta_{\Gamma_i + W_{ij}, A_{ij}}(\cdot),$$ where, by convention, $W_{i0} \defeq 0$ for each $i \ge 1.$ Its existence is a direct consequence of its construction, and the additional observation that $$\E{N([0, T] \times \A)} = \E{\sum_{i=1}^{\infty} \delta_{\Gamma_i, A_{i0}}([0, T] \times \A)} + \E{\sum_{i=1}^{\infty}\theta_{\Gamma_i} G_{A_{i0}}([0, T] \times \A)} < \infty$$ since the first term is simply a marked Poisson point process $N_0(\cdot)$ (hence, a locally finite point process) and the cluster part is finite by the (sufficient, see Section 6 in \cite{dvj03}) assumption on the cluster size. This implies that the process $N(\cdot)$ is a.s. finite. 
\end{rem}

\begin{rem}
    It is also possible to define the $i$th cluster point process $C_i(\cdot) \in \cN(E)$, on $(\Omega, \cF, \P)$, as $$C_i(\cdot) = \delta_{\Gamma_i, A_{i0}}(\cdot) + \theta_{\Gamma_i}G_{A_{i0}}(\cdot),$$ which helps to define, in the following subsections, the two submodels of interest in this work.  
\end{rem}

\subsubsection{Mixed Binomial Poisson Cluster Process.}

This model corresponds to Example 6.3a in \cite{dvj03}. Here, the clusters are only made up of an immigrant event potentially generating a stream of first-generation offsprings.  

\begin{defn}\label{def:mbpp}
    Let $N_{MB}(\cdot) \in \cN(E)$ be the Poisson cluster process, defined on $(\Omega, \cF, \P)$, in Definition~\ref{def:pcpp}, with its $i$th cluster point process $C_i(\cdot) \in \cN(E)$ defined by $$C_i(\cdot) = \delta_{\Gamma_i, A_{i0}}(\cdot) +  \theta_{\Gamma_i}G_{A_{i0}}(\cdot) = \delta_{\Gamma_i, A_{i0}}(\cdot) + \theta_{\Gamma_i}\sum_{j=1}^{K_{A_{i0}}} \delta_{W_{ij}, A_{ij}}(\cdot)$$ where 
    \begin{enumerate}
        \item[1.] $\left\{K_{A_{i0}}, \{W_{ij}\}_{j \ge 1}, \{A_{ij}\}_{j \ge 0}\right\}_{i \ge 1}$ is an i.i.d. sequence, such that (generically) $\E{K_A} < \infty$ and $\E{W}~<~\infty$;
        \item[2.] for each $i \ge 1$, $\{A_{ij}\}_{j \ge 1}$ is independent from both $K_{A_{i0}}$ and $\{W_{ij}\}_{j \ge 1};$
        \item[3.] for each $i  \ge 1$, the sequence $\{W_{ij}\}_{j \ge 1}$ is conditionally independent given $A_{i0}$, and conditionally independent from $K_{A_{i0}}$ given $A_{i0}$. 
    \end{enumerate}
    $N_{MB}(\cdot) \in \cN(E)$ is called a mixed Binomial Poisson cluster process in $E$. 
\end{defn}
\begin{rem}
    An important implicit consequence of the above definition is that the number of elements in the cluster $K_{A_{i0}}$ depends on the ancestral mark $A_{i0}$, as clearly emphasised in the notation of $K$. 
\end{rem}

\subsubsection{Hawkes process}
This model corresponds to Example 6.3c in \cite{dvj03}. Although the Hawkes process is classically defined through its conditional intensity function (see e.g. \cite{hawkes71, bm96}), qualifying as a point process with stochastic intensity in the sense e.g. of \cite{bremaud20}, we begin by presenting its from its branching structure perspective, originally established in \cite{ho74}. A key feature of the Hawkes process is its self-similarity: each point, whether an immigrant or an offspring, may itself act as an immigrant and generate its own offspring(s), leading to a multi-generational branching structure.

In this framework, the self-excitation mechanism is governed by a measurable fertility function denoted $h(\cdot, \cdot)$, that will be a regularly varying random variable in Section~\ref{section:check}. 

\begin{defn}\label{def:hawkesclus}
   Let $N_H(\cdot) \in \cN(E)$ be a Poisson cluster process, defined on $(\Omega, \cF, \P)$, as in Definition~\ref{def:pcpp}, with its $i$th cluster point process $C_i(\cdot) \in \cN(E)$ defined by $$C_i(\cdot) = \delta_{\Gamma_i, A_{i0}}(\cdot) + \theta_{\Gamma_i} G_{A_{i0}}(\cdot) = \delta_{\Gamma_i, A_{i0}}(\cdot) +  \theta_{\Gamma_i}\sum_{j=1}^{L_{A_{i0}}} \left(\delta_{W_{ij},A_{ij}}(\cdot) + \theta_{W_{ij}}G_{A_{ij}}(\cdot) \right)$$ where
    \begin{itemize}
        \item[1.] for each $i \ge 1,$ given $A_{i0}$, the first generation offspring process $N_{A_{i0}}(\cdot) \defeq \sum_{j=1}^{L_{A_{i0}}} \delta_{W_{ij}, A_{ij}}(\cdot) \in \cN(E)$ is an inhomogeneous Poisson process in $E$, defined on $(\Omega, \cF, \P)$, with intensity measure $$v_{N_{A_{i0}}}(\cdot) \defeq (h \text{Leb} \otimes F)(\cdot) $$ where $h(\cdot, \cdot)$ is a measurable fertility function such that $\E{\kappa_A} \defeq \int_0^{\infty} \E{h(s, A)} \diff s < 1$, and $F(\cdot)$ is a probability distribution on $\A$;
        \item[2.] $\{G_{A_{ij}}\}_{j \ge 1}$ is an i.i.d. sequence of point processes defined on $(\Omega, \cF, \P)$, conditionally independent from $N_{A_{i0}}(\cdot)$ given $\{A_{ij}\}_{j \ge 1}$, and distributed as $G_{A_{i0}}(\cdot).$
    \end{itemize}
    $N_H(\cdot) \in \cN(E)$ is called a linear marked Hawkes process in $E$. 
\end{defn}

\begin{rem}\label{rem:gw}
The condition $\E{\kappa_A} \defeq \int_0^{\infty} \E{h(s, A)} \diff s < 1$ guarantees that the clusters are a.s. finite. This is the so-called subcriticality condition in the Galton-Watson terminology -- which underpins the genealogical structure of the Hawkes process -- and is essential for the existence of the marked Hawkes process (see Section 6 in \cite{dvj03}). As mentioned above, the interpretation of the Hawkes process as a Galton–Watson process was first established in \cite{ho74}.
\end{rem}

\begin{rem}
    Note that the total cluster size of the $i$th cluster $K_i$ in this case is a sum of the form $$K_i = 1 + L_{A_{i0}} + \sum_{j=1}^{L_{A_{i0}}} L_{A_{i0j}} + \cdots$$ where $L_{A_{i0}}$ is the number of first-generation offsprings and, for each $j \ge 1$, $L_{A_{i0j}}$ is the number of first-generation offsprings of the $j$th first-generation offspring, and so on. 
\end{rem}

\subsection{(Hidden) Regular variation}
Following exposition in \cite{lrr14} and \cite{dtw22}, we introduce next regular variation and hidden regular variation as convergence in $\M(E \setminus F),$ see Subsection~\ref{subsection:measpp}. We first need the notions of scaling and cones. 
\begin{defn}[Section 2.2 in \cite{dtw22}]
    A scaling on a complete, separable metric space is a multiplication by positive scalars, acting from $(0, \infty) \times E$ to $E$ and such that 
    \begin{enumerate}
        \item $1x = x, $ for all $x \in E$;
        \item $u_1(u_2 x) = (u_1 u_2) x, $ for all $u_1, u_2 > 0$ and $x \in E$.
    \end{enumerate}
    A cone is a Borel subset $F \subset E$ such that, if $x \in F$, then $ux \in F$ for all $u >0.$ It is assumed that $F$ is a closed cone with 
    $$ d(x, F) < d(ux, F), \quad \text{for all } u > 1, \; x \in E \setminus F.$$
\end{defn}
\begin{rem}
    The scaling $T$ acts as a dilation of a set, that is, for $A \in \cB(E \setminus F)$, we denote by $TA$ the set $$TA \defeq \{Tx: x \in A \}.$$ 
\end{rem}

We recall the definition of regular variation for sequences, random elements in $E \setminus F$, and measures on $\M(E \setminus F),$ upon recalling Definition~\ref{def:convm}:

\begin{defn}[Definition 2.2 in \cite{dtw22}]\label{def:rv}
    \begin{enumerate}[label=(\alph*)]
    Let $(E, \cB(E), d)$ be a complete, separable metric space.
        \item\label{def:rv-1} A sequence $v(\cdot)$ is regularly varying with index $\alpha > 0$ if, for all $\delta > 0$,  $$\frac{v(\delta T)}{v(T)} \sim \delta^{\alpha}, \text{ as } T \rightarrow \infty.$$ 
        \item\label{def:rv-2} $X: (\Omega, \cF, \P) \rightarrow (E, \cB(E))$ is regularly varying on $E \setminus F$ with index $\alpha > 0$ if there exist a regularly varying sequence $v(\cdot)$ with index $\alpha > 0$, a scaling $T$, and a non-null measure $\mu(\cdot) \in \M(E \setminus F)$ such that $$v(T) \p{T^{-1} X \in \cdot} \rightarrow \mu(\cdot) \text{ in } \M(E \setminus F),\, \text{ as } T \rightarrow \infty. $$ We denote it by $X \in \text{RV}(E \setminus F, v(\cdot), \mu).$
        \item A measure $\nu(\cdot) \in \M(E \setminus F)$ is regularly varying with index $\alpha > 0$ if there exist a regularly varying sequence $v(\cdot)$ with index $\alpha > 0$, a scaling $T$, and a non-null measure $\mu(\cdot) \in \M(E \setminus F)$ such that $$v(T) \nu(T \cdot) \rightarrow \mu(\cdot) \text{ in } \M(E \setminus F),\, \text{ as } T \rightarrow \infty.$$ We denote it by $\nu \in \text{RV}(E \setminus F, v(\cdot), \mu).$  
    \end{enumerate}
\end{defn}
\begin{rem}
    Alternatively, for the $E-$valued random element $X$ in Definition~\ref{def:rv}\ref{def:rv-2}, if $E = \R$, and for sets of the form $(x, \infty)$, it is well known that the tail of $X$ may be written as an asymptotic equivalence of the form $\p{X > (Tx, \infty)} \sim (Tx)^{-\alpha} \ell_X(Tx), \text{ as } T \rightarrow \infty$, for $\ell_X(\cdot)$ a slowly varying function; hence, considering Definition~\ref{def:rv}\ref{def:rv-1}, it essentially means that $v(T) \p{X > (Tx, \infty)} \sim x^{-\alpha} \ell_X(Tx)/\ell_X(T) = \cO(1),$ as $T \rightarrow \infty,$ a fact we will repeatedly use in the proofs. 
\end{rem}
\begin{rem}\label{rem:rvvec}
    If the $E$-valued element considered in Definition~\ref{def:rv}\ref{def:rv-2} is a nonnegative random vector $(X,Y)$, then by Theorem 1.1 in \cite{bdm02} all linear combinations $t_1 X + t_2 Y$ are also (potentially) regularly varying, for $t_1, t_2 \in \R_+$. In fact, positively homogeneous maps (such as sums, projections on coordinates,...) are regularly varying, provided that the relevant subspace is charged by the limiting measure $\mu(\cdot)$, see e.g. Subsection 3.2.5.2 in \cite{mw24}. This will be used in the proofs of the results of Section~\ref{section:res}, where these considerations are fully detailed in the specific settings of regularly varying components for the processes defined in Section~\ref{section:intropp}. 
\end{rem}
 

\begin{defn}[Section 2.2 in \cite{dtw22}]\label{def:hrv}
    Let $X$ be an $E-$valued random element. $X$ has hidden regular variation of order $k \ge 1$ with index $\alpha_k > 0$ if there exists a regularly varying sequence $v(T)$ with index $\alpha > 0$ (and hence, $v(\cdot)^k$ is regularly varying with index $k\alpha > 0$) such that $$v(T)^k \p{T^{-1} X \in \cdot} \rightarrow \mu_{k}(\cdot) \text{ in } \M(E \setminus F_k), \, \text{ as } T \rightarrow \infty,$$ for some cone $F_k$ belonging to an increasing family $\{F_k\}_{k \ge 1}$, and with $\mu_k(\cdot)$ disjoint from $\mu_l(\cdot)$ for $k \neq l$.
\end{defn}

\subsection{Skorokhod \texorpdfstring{$M_1$}{M1} topology on \texorpdfstring{$\D([0,1], \R)$}{D([0,1], R)}}\label{m1}

Let $\D([0,1], \R)$ be the space of $\R-$valued, right-continuous with left-hand limits functions, referred to as càdlàg functions. Heuristically, the classical $J_1$ topology on $\D([0,1], \R)$ considers two càdlàg functions to be close if their discontinuities (or jumps) occur at nearly the same times and have similar magnitudes. For the cluster point processes of interest in this work, we need to go beyond this framework and consider the $M_1$ topology, first introduced in \cite{skorohod56} alongside the $J_1$ topology. A comprehensive account of the $M_1$ topology can be found in the monograph of \cite{whitt2002}; here, we only recall its construction and main features. Since the additive functionals considered in Section~\ref{section:res} are $\mathbb{R}-$valued, we do not distinguish between the weak and strong $M_1$ topologies as defined in \cite{whitt2002}, as they coincide in this specific case.

The key advantage of $M_1$ over the $J_1$ topology is its greater tolerance for how jumps are matched up: two càdlàg functions $x, y \in \D([0,1], \R)$ may be considered close not only if their jump times line up, but also if their jump sizes are slightly different. In particular, a single large jump in one process can be approximated, under $M_1$, by a cluster of smaller jumps in the other.

\begin{defn}[Equation (3.3) in \cite{whitt2002}]
    Let $x \in \D([0,1], \R)$. The graph $G_x$ of $x$ is defined by $$G_x \defeq \{(t, z) \in [0,1] \times \R : z \in [x(t-), x(t)]\}.$$
\end{defn}
\begin{rem}[Section 3.3 in \cite{whitt2002}]\label{rem:order}
    One can define a (total) order on $G_x$: for $(t_1, z_1), (t_2, z_2) \in G_x$, we say $(t_1, z_1) \le (t_2, z_2)$ if either
\begin{enumerate}
    \item $t_1 < t_2$; or
    \item $t_1 = t_2$ and $|x(t_1-) - z_1| \le |x(t_2-) - z_2|$.
\end{enumerate}
\end{rem}
\begin{defn}[Section 3.3 in \cite{whitt2002}]
    A parametric representation of $x \in \D([0,1], \R)$ is a continuous, nondecreasing function $(r, u)$ that maps $[0,1]$ onto $G_x$, where the order is to be understood in the sense of Remark~\ref{rem:order}. 
\end{defn}
In the above definition, the $r(\cdot)$ function, which is continuous from $[0,1]$ to $[0,1]$, is referred to as the temporal part of the parametric representation, while $u(\cdot)$, which is continuous from $[0,1]$ to $\R$, is referred to as its spatial part. 
\begin{defn}[Equation (3.4) in \cite{whitt2002}]
    Let $\cP(x)$ be the set of all parametric representations of $x \in \D([0,1], \R)$. The $M_1$ distance $d_{M_1}(\cdot, \cdot)$ between $x_1, x_2 \in \D([0,1], \R)$ is defined by  
\begin{equation}\label{eq:m1dist}
d_{M_1}(x_1, x_2) = \inf_{(r_j, u_j) \in \cP(x_j),\, j = 1,2} \{\|r_1 - r_2\|_{\infty} \vee \|u_1 - u_2\|_{\infty} \}.
\end{equation}
\end{defn}
\begin{rem}
    $d_{M_1}(\cdot, \cdot)$ is a proper metric, and a modification of it makes $\D([0,1], \R)$ a complete, separable metric space. Also, $d_{M_1} \le d_{J_1}$, where $d_{J_1}$ is Skorokhod $J_1$ distance. See Chapter 12 of \cite{whitt2002} for further details and properties. 
\end{rem}

Finally, we will denote by $\D_k([0,1] \times \R) \subset \D([0, 1]\times \R)$ the subset of càdlàg functions with at most $k$ discontinuities. Borel measures on the space of càdlàg processes with at least $(k+1)$ jumps ($\D([0,1], \R) \setminus \D_k([0,1] \times \R)$) will be denoted for simplicity by $\M(\D \setminus \D_k)$.

\subsection{Preliminary Lemma}\label{subsection:slutsky}

We introduce next a version of Slutsky's Theorem, for elements in $\M(E \setminus F)$, which is a counterpart of Proposition 2.4 in \cite{dtw22}. 

\begin{lem}\label{lem:slutsky}
    Let $(E,d)$ be a complete, separable metric space and let $Y_T^{(1)}$ and $Y_T^{(2)}$ be $E$-valued random variables. Let $F \subset E$ be a closed cone. Assume that $Y_T^{(1)} \in \text{RV}(E \setminus F, v(\cdot), \mu)$. 
    Suppose further that, for each $\epsilon > 0, r > 0,$ \begin{equation}\label{eq:lemc1}\limsup_{T \rightarrow \infty} v(x_T)\p{d(x_T^{-1} Y_T^{(1)}, x_T^{-1} Y_T^{(2)} ) > \epsilon, d(x_T^{-1} Y_T^{(1)}, F)>r} = 0\end{equation} and \begin{equation}\label{eq:lemc2}\limsup_{T \rightarrow \infty} v(x_T)\p{d(x_T^{-1} Y_T^{(1)}, x_T^{-1} Y_T^{(2)} ) > \epsilon, d(x_T^{-1} Y_T^{(2)}, F)>r} = 0.\end{equation} Then $Y_T^{(2)} \in \text{RV}(E \setminus F, v(\cdot), \mu)$, or, equivalently, $$v(x_T) \p{x_T^{-1}Y_T^{(2)} \in \cdot} \rightarrow \mu(\cdot) \text{ in } \M(E \setminus F), \text{ as } T \rightarrow \infty.$$
\end{lem}
\begin{proof}
    The proof of this lemma is a standard adaptation of the proof of Theorem 3.1 in \cite{billingsley99}, and is relegated to Section~\ref{secp:slutsky}.
\end{proof}

\begin{rem}
    The above lemma applies to general metrics $d(\cdot, \cdot);$ in our setting, it will be used for $d = d_{M_1}$, Skorokhod's $M_1$ distance introduced in Subsection~\ref{m1}. Specifically, it allows us to show that, if we can establish convergence for a càdlàg functional $Y_T^{(1)}(\cdot)$ in $\M(\D \setminus \D_k)$ -- where $Y_T^{(1)}$ is a measurable mapping of a point process for which we have convergence in  $\M(\cN \setminus \cN_k)$ -- then the same convergence also holds for another càdlàg functional $Y_T^{(2)}(\cdot)$ in $\M(\D \setminus \D_k)$, provided that $Y_T^{(2)}$ is sufficiently close to $Y_T^{(1)}$, in the $M_1$ sense; this holds even if we do not have convergence of the underlying point process driving $Y_T^{(2)}$ in $\M(\cN \setminus \cN_k)$, which is precisely what we need. 
\end{rem}

\section{Aggregating the cluster process}\label{section:aggregated}

In this section, we state some slight extensions of known results from \cite{dtw22}, considering aggregated versions of the Poisson cluster submodels considered in Subsection~\ref{section:intropp}. More specifically, we superpose the clusters at their immigrant starting points; formally, we let $\Pi_T(\cdot) \in \cN([0, 1] \times \R_+)$ be the compound Poisson process, defined on $(\Omega, \cF, \P)$, by \begin{equation}\label{eq:pit}\Pi_T(\cdot) \defeq \sum_{i=1}^{N_0(T)} \delta_{\Gamma_i/T,\, D_i} (\cdot),\, \text{ for } T > 0, \text{ with } D_i \defeq \sum_{j=0}^{K_i} f(A_{ij}) \eqdef \sum_{j=0}^{K_i} X_{ij},\end{equation}
with $K_i$ the total cluster size of the $i$th cluster, and $f(\cdot): \A  \rightarrow \R_+$ is a measurable function transporting the marks to $\R_+$. Because $N_0(T) \defeq N_0([0,T] \times \R_+)$ is a Poisson random variable, the following two results are standard, and found in general form in the literature (see e.g. \cite{gut09}): 
\begin{enumerate}
    \item[] \textbf{Fact 1.} The sequence of random measures on $[0,1]$ $$\frac{1}{\E{N_0(T)}} \sum_{i=1}^{N_0(T)} \delta_{\Gamma_i / T}(\cdot) \rightarrow  \text{Leb}(\cdot)\Big|_{[0,1]} \text{ in probability, as } T \rightarrow \infty,  $$ for the Prohorov metric. 
    \item[] \textbf{Fact 2.} The family $\{(N_0(T)/T)^{k+1}\}_{T \ge 0}$ is uniformly integrable for each $k \ge 0$.
\end{enumerate}
\begin{rem}\label{rem:space}
    In the terminology of Section~\ref{section:intro}, and because we now consider marked point processes with generic marks $D$ in $\R_+$, we note that the underlying space over which our random objects live is $E \setminus F \equiv (\R_+ \times \R_+) \setminus (\R_+ \times \{0\}).$ In what follows, $$\M\big(\cN(E \setminus F) \setminus \cN_k(E \setminus F)\big) \equiv \M\Big(\cN\big((\R_+ \times \R_+) \setminus (\R_+ \times \{0\})\big) \setminus \cN_k\big((\R_+ \times \R_+) \setminus (\R_+ \times \{0\})\big)\Big) \eqdef \M(\cN \setminus \cN_k)$$ and  $$\M\big(\D([0,1], E \setminus F) \setminus \D_k([0,1], E \setminus F)\big) \equiv \M\Big(\D\big([0,1], \R_+ \setminus \{0\}\big) \setminus \D_k\big([0,1], \R_+ \setminus \{0\}\big)\Big) \eqdef \M(\D \setminus \D_k).$$
\end{rem}

For our first result to hold, we need to make two assumptions. First, we require a generic sum of the marks $D$ in Equation~\eqref{eq:pit} to be regularly varying: 
\begin{assumption}\label{assumption1}
    $D \in \text{RV}(\R_+ \setminus \{0\}, v(\cdot), \mu),$ with index $\alpha > 1$. 
\end{assumption}
\begin{rem}
    Note that the limiting measure $\mu(\cdot)$ admits a density with respect to the Lebesgue measure, i.e. that $\mu(\diff x) = \alpha x^{-\alpha - 1}$ for $\alpha > 1$, since $D$ lives on $\R_+.$
\end{rem}

In addition to Assumption~\ref{assumption1} above, we suppose that the scaling $T \equiv x_T$ in Definition~\ref{def:rv} -- referred later as the large deviation scaling -- and the speed of regular variation has a specific order, that depends on the regular variation index of $D:$

\begin{assumption}\label{assumption2}
    The scaling in Definition~\ref{def:rv} satisfies $$x_T^{-1} = \smallO(T^{-\max(1/\alpha,\, 1/2)}), \text{ as } T \rightarrow \infty, \text{ for } \alpha > 1.$$  
\end{assumption}
\begin{rem}
\begin{enumerate}
    \item Assumption~\ref{assumption2} indicates that we allow for a broad class of scalings $x_T$, not only linear scalings in $T$, thus covering a wide range of asymptotic regimes.
    \item Accordingly, the speed induced by the choice of the scaling $x_T$ is denoted by $v(x_T)$. In what follows, we let $v'(x_T) \defeq v(x_T)/T$ and it follows by Assumption~\ref{assumption2} that $v'(x_T) \rightarrow \infty, $ as $T \rightarrow \infty.$
\end{enumerate}
\end{rem}

The next result is an adaptation of Theorem 3.4 in \cite{dtw22}, proving (hidden) regular variation of the aggregated process in Equation~\eqref{eq:pit}: 

\begin{prop}[Theorem 3.4 in \cite{dtw22}]\label{prop:convpp}
    Consider the sequence of marked point processes $\{\Pi_T(\cdot)\}_{T \ge 0}$ in Equation~\eqref{eq:pit}. Suppose Assumptions~\ref{assumption1} and~\ref{assumption2} hold. Then, for each $k \ge 0$, $$v'(x_T)^{k+1} \p{x_T^{-1} \Pi_T \in \cdot} \rightarrow \mu^{*}_{k+1}(\cdot) \text{ in } \M(\cN \setminus \cN_k), \text{ as } T \rightarrow \infty, $$ where, for Borel $B \in \cB(\cN \backslash \cN_k)$, $$\mu_{k+1}^{*}(B) \defeq \frac{1}{(k+1)!} \int_{([0,1]\times \R_+)^{k+1}} \I{ \sum_{i=1}^{k+1}\delta_{(t_i, y_i) \in B } } \lambda^{k+1}\diff t_1 \cdots \diff t_{k+1}\mu(\diff y_1)\cdots \mu(\diff y_{k+1}).$$
\end{prop}
\begin{proof}
    The proof is a direction adaptation of the one of Theorem 3.4 in \cite{dtw22}, since the defining quantities in Equation~\eqref{eq:pit} satisfy all required assumptions therein, and is omitted for brevity. 
\end{proof}
\begin{rem}
\begin{enumerate}
    \item  Note that the scaling $x_T$ in Proposition~\ref{prop:convpp} only acts on the space component of $\Pi_T(\cdot)$, i.e. $x_T^{-1} \Pi_T(\cdot) = \sum_{i=1}^{N_0(T)} \delta_{\Gamma_i/T,\, x_T^{-1} D_i}(\cdot).$
    \item The result for $k=0$ in the case of the Poisson process was proved in Theorem A.1 of \cite{dhs18}.
\end{enumerate}
\end{rem}

We define the centered cumulative functional of $\Pi_T(\cdot)$, for $t \in [0,1]$, as a càdlàg process $\tilde{R}_T \in \D([0,1], \R)$ -- on $(\Omega, \cF, \P)$ -- by
\begin{equation}\label{eq:rttilded}
    \tilde{R}_T(t) \defeq \sum_{i=1}^{N_0(tT)} \left( D_i - \E{D} \right) \I{\Gamma_i/T \le t}, 
    \quad \text{for } T > 0.
\end{equation}

It is possible to prove that a slight extension of Theorem 4.3 in \cite{dtw22} holds for Equation~\eqref{eq:rttilded}, therefore showing a (hidden) regular variation principle for the cumulative functional. 

\begin{prop}\label{prop:rvd}
     Suppose Assumption~\ref{assumption1} and~\ref{assumption2} above hold. Then, for $k \ge 0,$ $$v'(x_T)^{k+1} \p{x_T^{-1} \tilde{R}_T \in \cdot} \rightarrow \mu_{k+1}^{\#}(\cdot) \text{ in } \M(\D \setminus \D_k), \text{ as } T \rightarrow \infty,$$ where, for Borel $B \in \cB(\D \setminus \D_k)$, $$\mu_{k+1}^{\#}(B) \defeq \frac{1}{(k+1)!} \int_{([0,1]\times \R_+)^{k+1}} \I{ (\sum_{i=1}^{k+1} y_i\I{t_i\le t})_{0\le t\le 1} \in B }  \lambda^{k+1}\diff t_1 \cdots \diff t_{k+1} \mu(\diff y_1)\cdots \mu(\diff y_{k+1}).$$
\end{prop}
\begin{proof}
    The proof is a straightforward adaptation of the one of Theorem 4.3 in \cite{dtw22}, except that the number of terms in $\tilde{R}_T$ is random ($N_0(tT)$) in our case. The complete proof is omitted for brevity. 
\end{proof}
\begin{rem}
\begin{itemize}
    \item  Note that, from the product form of the tail measure $\mu_{k+1}^{\#}(\cdot), $ the index of regular variation of $\tilde{R}_T(\cdot)$ is given by $(k+1)\alpha,$ for $k \ge 0.$ This means that, as we remove larger and larger subparts of $\D$ as $k$ grows, the heavy-tailedness of the additive functional decreases proportionally. 
    \item From Remark~\ref{rem:space}, the underlying space on which $\mu(\cdot)$ lives is $\R_+ \setminus \{0\}$, while Equation~\eqref{eq:rttilded} is $\mathcal{D}([0,1], \mathbb{R})$-valued. This should not be surprising, since Proposition~\ref{prop:rvd} is a sample-path large deviation principle and, asymptotically, the centering term becomes negligible. See Remark~\ref{rem:samplepath}. 
\end{itemize} 
\end{rem}

\section{Unfolding the aggregated cluster process}\label{section:res}

Equipped with the results from Section~\ref{section:aggregated}, we now aim to show that an analogue of Proposition~\ref{prop:rvd} holds for Poisson cluster processes. However, to extend the proof of Proposition~\ref{prop:convpp} to more complex processes, and to apply a continuous mapping theorem in order to derive a version of Proposition~\ref{prop:rvd} in our case, it is essential that the processes under consideration satisfy the uniform integrability condition stated in Fact~2 of Section~\ref{section:aggregated}. 

Recall that $N_H(\cdot) \in \cN(E)$ denotes the Hawkes process defined in Definition~\ref{def:hawkesclus}, and set $N_H(T) \defeq N_H([0,T])$. While the family $\{ N_H(T) / \E{N_H(T)} \}_{T \ge 0}$ is known to be uniformly integrable -- as a consequence of results from Chapters~1 and~2 in \cite{gut09} -- it remains an open question whether the family $\{ (N_H(T) / \E{N_H(T)})^{k+1} \}_{T \ge 0}$ is uniformly integrable for each $k \ge 0$. 

As a result, although Proposition~\ref{prop:convpp} can be extended to Poisson cluster processes in the case $k = 0$, it does not hold in general for $k \ge 1$, and we cannot obtain a version of Proposition~\ref{prop:rvd} directly in this setting. As noted in the introduction, the recent contribution by \cite{b2025} appears to circumvent this problem in the case of multivariate Hawkes processes.

In our case, we proceed via an indirect approach. Specifically, we apply Lemma~\ref{lem:slutsky} together with Proposition~\ref{prop:rvd} to establish convergence in $\M(\D \setminus \D_k)$ for the càdlàg process $\hat{R}_T(\cdot) \in \D([0,1], \R)$ defined, for $t \in [0,1]$, and on $(\Omega, \cF, \P)$, by \begin{equation}\label{eq:rthatz} \hat{R}_T(t) \defeq \sum_{i=1}^{N_0(tT)} \sum_{j=0}^{\mfK_i} Z_{ij} \I{(\Gamma_i + W_{ij})/T \le t} - \E{\sum_{i=1}^{N_0(tT)} \sum_{j=0}^{\mfK_i} Z_{ij} \I{(\Gamma_i + W_{ij})/T \le t }}, \end{equation} where, for each $i \ge 1$ and $j \ge 0$, the precise form of the sequence of i.i.d. nonnegative random variables ${Z_{ij}}$ and the sequence of i.i.d. $\N_0$-valued random variables ${\mfK_i}$ depends on the two submodels of interest introduced in Section~\ref{section:intropp}; detailed specifications are deferred to Section~\ref{section:check}. For the remainder of the section, we work under this general formulation, subject to additional assumptions stated below. The verification of these assumptions for each submodel is again deferred to Section~\ref{section:check}.

\begin{rem}
    The definition of $\hat{R}_T(\cdot)$ in Equation~\eqref{eq:rthatz} presupposes implicitely that, for each $i \ge 1$ and $D_i$ in Equation~\eqref{eq:rttilded}, if one is willing to use $\hat{R}_T(\cdot)$ as an approximation of $\tilde{R}_T(\cdot)$, it is possible to decompose, for $t \in [0,1], \, T > 0,$ 
\begin{equation}\label{eq:decomp}
    D_i \I{\Gamma_i/T \le t} = \sum_{j=0}^{\mfK_i} Z_{ij} \I{(\Gamma_i + W_{ij})/T \le t}.
\end{equation}
\end{rem}

In this section, our objective is to show that Equation~\eqref{eq:rthatz} provides a good approximation of Equation~\eqref{eq:rttilded} in the $M_1$ topology. We proceed in several steps. The first goal is to demonstrate that the process $\tilde{R}_T(\cdot)$, defined in Equation~\eqref{eq:rttilded}, can be approximated by the trajectory formed by its $(k+1)$ largest jumps, provided that we focus on the appropriate risk scenario at the $k$-th step. This scenario is controlled by intersecting the approximation with the event $\{ d_{M_1}(x_T^{-1} \tilde{R}_T, \D_k) > r \}$ for all $r > 0$.

We maintain the hypothesis that $D$ satisfies Assumption~\ref{assumption1} from Section~\ref{section:aggregated}, with a scaling sequence also fulfilling Assumption~\ref{assumption2} from the same section. Additionally, let $N_0(\cdot) \in \cN(E)$ denote the Poisson process defined in Section~\ref{section:aggregated}.

\begin{rem}
    In what follows, for each $i \in \{1, \ldots, N_0(T)\}$, we denote by by $D_{(N-i)}$ the $(N_0(T)-i)$th largest order statistics of the exchangeable sequence $\{D_{i}\}_{1 \le i \le N_0(T)}$. For simplicity, $N \defeq N_0(T)$ in the subscripts of the statements of the propositions below, and in their proofs in Section~\ref{section:proof4}.
\end{rem}
\newpage

\begin{prop}\label{prop:t1}
    Let $\{D_{i}\}_{i \ge 1}$ be an i.i.d sequence satisfying Assumption~\ref{assumption1}, with scaling sequence $x_T$ satisfying Assumption~\ref{assumption2}, independent from $N_0$. Then it holds, for each $k \ge 0$, and all $\epsilon, r > 0$, that $$v'(x_T)^{k+1}\p{d_{M_1}\Big(x_T^{-1}\tilde{R}_T,\,  x_T^{-1} \sum_{i=0}^{k} D_{(N-i)} \I{\Gamma_{(N-i)}/T \le t}\Big) > \epsilon,\,  d_{M_1}(x_T^{-1}\tilde{R}_T, \D_k) > r}= \smallO(1), \text{ as } T \rightarrow \infty.$$
\end{prop}
\begin{proof}
    The proof follows the lines of the proof of Theorem 4.3 in \cite{dtw22}, and is relegated to Subsection~\ref{app:pt1}. 
\end{proof}



Our next result uses the decomposition in Equation~\eqref{eq:decomp} to show that, in the $M_1$ topology, the trajectory of the $(k+1)$ largest jumps of the sequence $\{D_i\}_{i \ge 1}$ is well approximated by the combined trajectories of the corresponding $(k+1)$ decompositions, restricted to events observed within the time window $[0, T]$, for $T > 0$.

To this end, we first formulate additional assumptions on the quantities involved in Equation~\eqref{eq:rthatz}.

\begin{assumption}\label{assumption3}  
For each $i \ge 1$, the pair $(\mfK_i, \{W_{ij}\}_{j \ge 1})$ is a generally dependent random vector, where:
\begin{enumerate}
    \item $\mfK_i$ is an $\N$-valued random variable, with $\p{\mfK_i > x} \sim c \p{D_i > x}, $ as $x \rightarrow \infty$, for $c \in [0, \infty)$;
    \item $\{W_{ij}\}_{j \ge 1}$ is a sequence of nonnegative random variables, conditionally independent given some $\sigma-$algebra $\cF_{FG,i}$, with $\E{W} < \infty$. It is further conditionally independent of $\mfK_i$ given $\cF_{FG,i}$. 
\end{enumerate}
\end{assumption}
\begin{rem}
\begin{enumerate}
    \item Assuming $\p{\mfK_i > x} \sim c \p{D_i > x}, $ as $x \rightarrow \infty$, for $c \in [0, \infty)$ essentially means, by Assumption~\ref{assumption1}, that for each $i \ge 1$, $\mfK_i$ is either potentially regularly varying, with the same or a greater tail index than $D_i$, or that it is negligible in front of $D_i.$
    \item By convention, we set $W_{i0} \defeq 0$ for all $i \ge 1$, and we consider the extended sequence $\{W_{ij}\}_{j \ge 0}$.
    \item The notation $\cF_{FG,i}$ for this $\sigma$-algebra emphasises that the corresponding events are independent within the first generation, conditional on the events of the previous generation.  
\end{enumerate}
\end{rem}

\begin{assumption}\label{assumption4}  
For each $i \ge 1$, $\{Z_{ij}\}_{j \ge 0}$ is an independent sequence of nonnegative random variables, with, for each $j \ge 1$, $\p{Z_{ij} > x} \sim c \p{D_i > x}, $ as $x \rightarrow \infty,$ for $c \in [0, \infty).$
\end{assumption}

We need the following definition of the ``remainder terms'' after a time truncation $T > 0.$
\begin{defn}\label{def:remaindert}
    For a fixed $T>0$, the remainder term of the $i$th cluster, consisting of events occurring after time $T$ but triggered by points in $[0,T]$, is defined by \begin{equation}\label{eq:remaindert} D_i^{>T} \defeq \sum_{j=1}^{\mfM_i} Z_{ij} \I{Q_{ij} > T - \Gamma_i}\end{equation} where, for each $i \ge 1$, $\{Z_{ij}\}_{j \ge 1}$ is an i.i.d. sequence of nonnegative random variables, $(\mfM_i, \{Q_{ij}\}_{j \ge 1})$ is a generally dependent random vector, with $\mfM_i$ an $\N-$valued random variable and $\{Q_{ij}\}_{j \ge 1}$ is a sequence of nonnegative random variables. 
\end{defn}

About the remainder term of Equation~\eqref{eq:remaindert}, we formulate the following assumption. 

\begin{assumption}\label{assumption5}  
For quantities in Equation~\ref{eq:remaindert}, it holds, for each $i \ge 1$, that
\begin{enumerate}
    \item $\mfM_i$ satisfies $\p{\mfM_i > x} \sim c \p{D_i > x},$ as $x \rightarrow \infty$, for $c \in [0, \infty);$
    \item $\mfM_i$ and $\{Q_{ij}\}_{j \ge 1}$ are conditionally independent given some $\sigma-$algebra $\cF_{CI, i}$;
    \item $\{Q_{ij}\}_{j \ge 1}$ is a conditionally independent sequence given $\sigma-$algebra $\cF_{CI, i}$. 
\end{enumerate}

Furthermore, for each $i \ge 1, j \ge 1$, it holds that 
$$\p{Q_{ij} > T - \Gamma_i \mid \cF_{CI,i}} = \smallO(1) \text{ a.s.}, \quad \text{ as } T \to \infty.$$
\end{assumption}

\begin{rem}
\begin{enumerate}\label{rem:sigci}
    \item Assumption~\ref{assumption5} controls the duration of the $i$th cluster, preventing it from growing unboundedly as $T \rightarrow \infty$. This is a crucial condition when $W$ may be arbitrarily heavy-tailed; see Remark~2 in \cite{mr05}. It is in particular essential for the $M_1$ approximation in the forthcoming Proposition~\ref{prop:t2} to hold. 
    \item $\cF_{CI,i}$ ensures that the sequence $\{Q_{ij}\}_{1 \le j \le I}$, for some index set $I \subseteq \{1,2, \ldots, K_i\}$, is conditionally i.i.d. It is particularly needed in the proof of Proposition~\ref{prop:t2}, and the exact form of $\cF_{CI,i}$ for each submodel is given in Section~\ref{section:check}. 
\end{enumerate}
\end{rem}

    

The following assumption concerns waiting times of the triggered events in the clusters.

\begin{assumption}\label{assumption6}  
For $x_T$ satisfying Assumption~\ref{assumption2}, for each cluster $i \ge 1,$ and for each $j \in \{1, \ldots, K_i\}$, it holds that 
$$ x_T \, \p{Q_{ij} > \epsilon T} = \smallO(1), \quad \text{ as } T \rightarrow \infty.$$
\end{assumption}

\begin{rem}
    Assumption~\ref{assumption6} details the relationship between the distribution of $Q$ and the scaling sequence $x_T$, which determines the order of the large deviation regime under consideration. In Remark~\ref{rem:Worder} below, we give moment conditions on $Q$ in order for Assumption~\ref{assumption6} to hold for a fixed $x_T$. This assumption represents the precise cost required to obtain a functional version of the large deviation principle for the partial sums established in Proposition~10 of \cite{bcw23}. 
\end{rem}

As previously mentioned, our aim is to use the above set of assumptions to show that, in the $M_1$ topology, the process formed by the $(k+1)$ largest jumps can be effectively approximated by considering the trajectory of their decompositions over the interval $[0, T]$, as described in Equation~\eqref{eq:decomp}.

\begin{prop}\label{prop:t2}
     Let $\{D_{i}\}_{i \ge 1}$ be an i.i.d sequence satisfying Assumption~\ref{assumption1}, with scaling sequence $x_T$ satisfying Assumption~\ref{assumption2}, independent from $N_0$. Suppose further that Assumptions~\ref{assumption3} to Assumptions~\ref{assumption6} hold. Then, for each $k \ge 0$, and for all $\epsilon, r >0$, and as $T \rightarrow \infty,$
    {\footnotesize\begin{equation*} v'(x_T)^{k+1} \p{d_{M_1}\Big(x_T^{-1} \sum_{i=0}^{k} D_{(N-i)}  \I{\Gamma_{(N-i)}/T \le t},\,  x_T^{-1} \sum_{i=0}^{k} \sum_{j=0}^{\mfK_{(N-i)}} Z_{(N-i)j} \I{(\Gamma_{(N-i)} + W_{(N-i)j})/T \le t}\Big) > \epsilon,\,  d_{M_1}(x_T^{-1}\tilde{R}_T, \D_k) > r} = \smallO(1).\end{equation*} }
\end{prop}
\vspace{0.2cm}

\begin{rem}
    Note that, in the notation of Proposition~\ref{prop:t2}, the decomposition of the $(N_0(T)-i)$th largest cluster $D_{(N-i)}$ is identified, in notation, by $\sum_{j=0}^{\mfK_{(N-i)}} Z_{(N-i)j} \I{(\Gamma_{(N-i)} + W_{(N-i)j)/T} \le \cdot}$. In particular, this means that $\mfK_{(N-i)}$, $Z_{(N-i)j}$, $\Gamma_{(N-i)}$, and $W_{(N-i)j}$ do not correspond to the $(N_0(T)-i)$th order statistics of their respective sequences.
\end{rem}

\begin{proof}
    The proof uses the specificities of the $M_1$ topology presented in \cite{whitt2002}, and is relegated to Subsection~\ref{proof:t2}. 
\end{proof}

We need to introduce a final technical condition ensuring that the centerings corresponding to different components of the decomposed processes are deemed equivalent in the $M_1$ topology. 

\begin{assumption}\label{assumption7} Let $\{D_{i}\}_{i \ge 1}$ be an i.i.d sequence satisfying Assumption~\ref{assumption1}, with scaling sequence $x_T$ satisfying Assumption~\ref{assumption2}, independent from $N_0$. It holds, for quantities satisfying Assumptions~\ref{assumption3} and~\ref{assumption4}, for each $k \ge 0$ and for all $\epsilon > 0$, that {\footnotesize$$v'(x_T)^{k+1}\p{\sup_{0 \le t \le 1}\Bigg\lvert x_T^{-1} N_0(tT) \E{D} - x_T^{-1} \E{\sum_{i=1}^{N_0(tT)} \sum_{j=1}^{\mfK_i} Z_{ij} \I{(\Gamma_i + W_{ij})/T \le t}}\Bigg\rvert > \epsilon, d_{M_1}(x_T^{-1} \tilde{R}_T, \D_k) > r} = \smallO(1), \text{ as } T \rightarrow \infty.$$} 
\end{assumption}


With Propositions~\ref{prop:t1} and~\ref{prop:t2} at hand, we can now establish our main theorem: the (hidden) regular variation propoerty of Equation~\eqref{eq:rthatz} in $\M(\D \setminus \D_k)$ for each $k \ge 0$, by using the corresponding result for $\tilde{R}_T(\cdot)$ given in Proposition~\ref{prop:rvd}.

\begin{thm}\label{thm:clustrv}
Let $\{D_{i}\}_{i \ge 1}$ be an i.i.d sequence satisfying Assumption~\ref{assumption1}, with scaling sequence $x_T$ satisfying Assumption~\ref{assumption2}, independent from $N_0$. Suppose further that Assumptions~\ref{assumption3} to~\ref{assumption7} hold. Then, for $k \ge 0$, and $\hat{R}_T(\cdot)$ in Equation~\eqref{eq:rthatz}, 
    $$v'(x_T)^{k+1} \p{x_T^{-1} \hat{R}_T \in \cdot} \rightarrow \mu_{k+1}^{\#}(\cdot) \text{ in } \M(\D \setminus \D_k),\, \text{ as } T \rightarrow \infty, $$ with respect to the $M_1$ topology, where $\mu_{k+1}^{\#}(\cdot)$ is defined as in Proposition~\ref{prop:rvd}.
\end{thm}
\begin{proof}
For the approximation to hold, combining the result of Proposition~\ref{prop:rvd} and Lemma~\ref{lem:slutsky}, we want to show, for each $k \ge 0$, and for all $\epsilon, r > 0$, that it holds 
\begin{align*}
     \limsup_{T \rightarrow \infty}v'(x_T)^{k+1} \pi_{T}^{k+1}(\tilde{R}_T, \hat{R}_T) \defeq\limsup_{T \rightarrow \infty} v'(x_T)^{k+1}\p{d_{M_1}(x_T^{-1}\tilde{R}_T, x_T^{-1}\hat{R}_T) > \epsilon, d_{M_1}(x_T^{-1} \tilde{R}_T, \D_{k}) > r} = 0.
\end{align*}
This suffices, because $d_{M_1}(x_T^{-1} \hat{R}_T, \D_k) \le d_{M_1}(x_T^{-1} \tilde{R}_T, \D_k)$ for each $k \ge 0$.

Now, the joint probability in the above equation can further be decomposed, using the fact that $d_{M_1}$ is a proper metric (see Chapter 6 of \cite{whitt2002}); by a triangular inequality, it holds that
{\footnotesize
\begin{align} 
     &v'(x_T)^{k+1}  \pi_{T}(\tilde{R}_T, \hat{R}_T) 
     \le v'(x_T)^{k+1}\p{d_{M_1}\big(x_T^{-1}\tilde{R}_T,\,  x_T^{-1} \sum_{i=0}^{k} Z_{(N-i)} \I{\Gamma_{(N-i)} \le tT} > \epsilon,\,  d_{M_1}(x_T^{-1}\tilde{R}_T, \D_k) > r} \notag \\
    &\quad + v'(x_T)^{k+1} \p{d_{M_1}\Big(x_T^{-1} \sum_{i=0}^{k} Z_{(N-i)}  \I{\Gamma_{(N-i)} \le tT},\,  x_T^{-1} \sum_{i=0}^{k} \sum_{j=0}^{\mfK_{(N-i)}} Z_{(N-i)j} \I{\Gamma_{(N-i)} + W_{(N-i)j} \le tT}\Big) > \epsilon,\,  d_{M_1}(x_T^{-1}\tilde{R}_T, \D_k) > r}\notag \\
    &\quad + v'(x_T)^{k+1} \p{d_{M_1}\Big(x_T^{-1} \sum_{i=0}^{k} \sum_{j=0}^{\mfK_{(N-i)}} Z_{(N-i)j} \I{\Gamma_{(N-i)} + W_{(N-i)j} \le tT},\, x_T^{-1}\hat{R}_T  \Big) > \epsilon,\,  d_{M_1}(x_T^{-1}\tilde{R}_T, \D_k) > r}\notag \\
    &\eqdef \T_1 + \T_2 + \T_3. \notag 
\end{align}
}
The negligibility $\T_1 = \smallO(1), $ as $T \rightarrow \infty$, follows from Proposition~\ref{prop:t1}, while $\T_2= \smallO(1),$ as $T \rightarrow \infty$, follows from Proposition~\ref{prop:t2}. 

For $\T_3$, note that, by Lemma 6.1 in \cite{dtw22}, we can bound the second event in the probability of interest by $$\{d_{M_1}(x_T^{-1}\tilde{R}_T, \D_k) > r\} \subseteq \{D_{(N-k)} > 2rx_T \}.$$
   
For the first event, we bound the $M_1$ distance by its sup-norm (see Subsection~\ref{m1}), and the expression to consider boils down to
    {\footnotesize
\begin{align*}\label{eq:t3}
 &\sup_{0 \le t \le 1} \Bigg\lvert x_T^{-1} \sum_{i=1}^{N_0(tT)}  \sum_{j=1}^{\mfK_i} Z_{ij} \I{\Gamma_i + W_{ij} \le tT} - x_T^{-1} \E{\sum_{i=1}^{N_0(tT)} \sum_{j=1}^{\mfK_i} Z_{ij} \I{\Gamma_i + W_{ij} \le tT}} - x_T^{-1}\sum_{i=0}^{k}\sum_{j=0}^{\mfK_{(N-i)}} Z_{(N-i)j} \I{\Gamma_{(N-i)} + W_{(N-i)j} \le tT}\Bigg\rvert \\
    &\quad \le \sup_{0 \le t \le 1} \Bigg\lvert x_T^{-1} \sum_{i=1}^{N_0(tT)} \Big(\sum_{j=1}^{\mfK_i} Z_{ij} \I{\Gamma_i + W_{ij} \le tT} - \E{D} \Big) - x_T^{-1} \sum_{i=0}^{k}\Big(\sum_{j=0}^{\mfK_{(N-i)}} Z_{(N-i)j} \I{\Gamma_{(N-i)} + W_{(N-i)j} \le tT} - \E{D} \Big) - x_T^{-1} \sum_{i=0}^{k} \E{D} \Bigg\rvert \\
    &\qquad + \sup_{0 \le t \le 1} \Bigg\lvert x_T^{-1} N_0(tT) \E{D} - x_T^{-1} \E{\sum_{i=1}^{N_0(tT)} \sum_{j=1}^{\mfK_i} Z_{ij} \I{\Gamma_i + W_{ij} \le tT}}\Bigg\rvert \\
    &\quad \eqdef \T_{31} + \T_{32}, 
\end{align*}
}
where the upper bound is obtained using a triangular inequality. 

By another triangular inequality on the whole expression $\T_3$, again using the fact that $d_{M_1}$ is a proper metric, this means that we need to show, for each $k \ge 0$, all $\epsilon, r > 0$, that 
\begin{align*}
    v'(x_T)^{k+1} \Big(\p{\T_{31} > \epsilon/2, D_{(N-k)} > 2rx_T} + \p{\T_{32} > \epsilon/2, D_{(N-k)} > 2rx_T}\Big) = \smallO(1), \text{ as } T \rightarrow \infty.
\end{align*}

On the one hand, to control $\T_{31}$, note that 
\begin{align*}
    x_T^{-1} \Bigg(\sum_{i=1}^{N_0(tT)} \Big(\sum_{j=1}^{\mfK_i} Z_{ij} \I{\Gamma_i + W_{ij} \le tT} - \E{D} \Big) - \Big(\sum_{i=0}^{k}& \sum_{j=0}^{\mfK_{(N-i)}} Z_{(N-i)j} \I{\Gamma_{(N-i)} + W_{(N-i)j} \le tT} - \E{D} \Big) \Bigg) \\ &= x_T^{-1}\sum_{\substack{i=1 \\ i \neq (N), (N-1), \ldots (N-k)}}^{N_0(tT)} \Big(\sum_{j=0}^{\mfK_i} Z_{ij} \I{\Gamma_i + W_{ij} \le tT} - \E{D} \Big)
\end{align*}
and because $$\sum_{j=0}^{\mfK_i} Z_{ij} \I{\Gamma_i + W_{ij} \le tT} \le D_i \I{\Gamma_i/T \le t}\, \text{ a.s.}$$
the fact that, for each $k \ge 0$, and all $\epsilon, r > 0$  $$v'(x_T)^{k+1}\p{\T_{31} > \epsilon/2, D_{(N-k)} > 2rx_T} = \smallO(1), \text{ as } T \rightarrow \infty,$$ follows by using exactly the same proof as the one of Proposition~\ref{prop:t1}, modulo the extra (asymptotically negligible) term $x_T^{-1} \sum_{i=0}^k \E{D} = \smallO(1)$, as $T  \rightarrow \infty.$ 

On the other hand, by Assumption 7, for each $k \ge 0$, and for all $\epsilon, r > 0$, it holds that
$$v'(x_T)^{k+1}\p{\T_{32} > \epsilon/2, D_{(N-k)} > 2rx_T } = \smallO(1),\, \text{ as } T \rightarrow \infty,$$ which concludes the proof of the theorem. 

\end{proof}
\begin{rem}\label{rem:samplepath}
    Theorem~\ref{thm:clustrv} shows that, for each $k \ge 0$, the process $\hat{R}_T(\cdot)$ from Equation~\eqref{eq:rthatz}, is (hiddenly) regularly varying with speed $v'(x_T)^{k+1}$ and scaling $x_T$, since its rescaled distribution, as a measure on the set of càdlàg functions with at least $(k+1)$ points, satisfy Definition~\ref{def:rv}, Definition~\ref{def:hrv}. As emphasised in Section 2.2 of \cite{dtw22}, it can be formulated as a sample paths large deviation result, as in Theorem 3.2 in \cite{rbz19}. Indeed, write $\text{cl}A$ for the closure of a set $A$ and $\text{int}A$ for its interior; for any Borel set $A \subset \D \setminus \D_k$, define $$\cK(A) = \max\{k \ge 0: A \cap \D_k = \emptyset \} \text{ and } \cI(A) = \mu_{\cK(A)}^{\#}(A).$$ Then, Theorem~\ref{thm:clustrv} can be written, in an equivalent form, as $$\cI(\text{int} A)\le \liminf_{T \rightarrow \infty} v'(x_T)^{\cK(A)} \p{x_T^{-1} \hat{R}_T \in A} \le \limsup_{T \rightarrow \infty} v'(x_T)^{\cK(A)} \p{x_T^{-1} \hat{R}_T \in A} \le \cI(\text{cl}A) $$ for $A$ such that $\cK(A)$ is finite, and $A$ is bounded away from $\D_{\cK(A)}.$ Sample paths large deviations, as emphasised in the introduction, is a very active field of research, with seminal contribution due to \cite{pinelis81}. 
\end{rem}

\begin{rem}
    Theorem~\ref{thm:clustrv} establishes a hidden regular variation principle for each $k \ge 0$, which -- by the preceding remark -- can be interpreted as a sample path large deviation principle for summation functionals of Poisson cluster processes. The form of the limiting measure $\mu^{\#}_{k+1}(\cdot)$ reveals that, for each $k \ge 1$, it is supported on the space of càdlàg functions with exactly $(k+1)$ discontinuities. This indicates that, when controlling for the relevant ``risk scenarios'' -- for all $r > 0$, by the sets $\{d_{M_1}(x_T^{-1} \tilde{R}_T, \D_k) > r\}$ -- the limiting behaviour involves multiple large jumps in the tail measure. Interestingly, under Assumptions~\ref{assumption1} to~\ref{assumption7}, no additional asymptotic cost is incurred when moving from the case $k = 0$ to the case $k \ge 1$. This phenomenon, highlighted in the proofs relegated to Section~\ref{section:proof4}, arises from the precise control of the risk scenario of interest at each level $k$, and more specifically by the way we construct the parametric representations -- in the $M_1$ topology -- of the processes in Proposition~\ref{prop:t2}. 
\end{rem}

\section{Verifying the assumptions for submodels of Section~\ref{section:intropp}}\label{section:check}

\subsection{Mixed Binomial Poisson cluster process}

In this single-generation mixed Binomial Poisson cluster process introduced in Section~\ref{section:intropp}, the centered cumulative functional corresponding to Equation~\eqref{eq:rthatz}, as a càdlàg process in $\D([0,1], \R)$, defined on $(\Omega, \cF, \P)$, is given by
\begin{equation}\label{eq:cs1} \hat{R}^{MB}_T(t) \defeq \sum_{i=1}^{N_0(tT)} \sum_{j=0}^{K_{A_{i0}}} X_{ij}  \I{(\Gamma_i + W_{ij})/T \le t} - \E{\sum_{i=1}^{N_0(tT)} \sum_{j=0}^{K_{A_{i0}}} X_{ij}  \I{(\Gamma_i + W_{ij})/T \le t}}, \text{ for } t \in [0,1]. \end{equation}

We state the sufficient assumptions on the mixed Binomial Poisson process in order to obtain a corollary of Theorem~\ref{thm:clustrv}. We denote the following set of assumptions by \textbf{(A. MB.nf)}. 
\begin{enumerate}
    \item[] \textbf{Assumption 1.} For each $i \ge 1$, $(X_{i0}, K_{A_{i0}}) \in \text{RV}((\R_+ \setminus \{0\}) \times (\R_+ \setminus \{0\}), v(\cdot), \mu')$, with index $\alpha > 1$. 
    \item[] \textbf{Assumption 2.} Assumption~\ref{assumption2} is assumed for the order of the large deviation scaling sequence $x_T$. 
\end{enumerate}
In order to state a forthcoming corollary in non-functional setting, we isolate the following assumption that we denote by \textbf{(A. MB.f)}. 
\begin{enumerate}
    \item[] \textbf{Assumption 6.} For a each $i \ge 1, j \ge 1,$ $$ x_T\p{W_{ij} > \epsilon T} = \smallO(1), \text{ as } T \rightarrow \infty.$$
\end{enumerate}

Together, the set of assumptions \textbf{(A. MB.nf)} and \textbf{(A. MB.f)} is simply denoted by \textbf{(A. MB)}. 

We can now state the following corollary to Theorem~\ref{thm:clustrv}. Define $g(\cdot, \cdot): \R^2_+ \rightarrow \R_+$ by $g(x_0, k) = x_0 + \E{X} k.$ 

\begin{cor}\label{cor:mb}
For the mixed Binomial Poisson cluster process introduced in Section~\ref{section:intropp}, assume Assumptions \textbf{(A. MB)} hold. Then, for $k \ge 0$, and $\hat{R}^{MB}_T(\cdot)$ in Equation~\eqref{eq:cs1}, 
    $$v'(x_T)^{k+1} \p{x_T^{-1} \hat{R}^{MB}_T \in \cdot} \rightarrow \mu_{k+1}^{\#}(\cdot) \text{ in } \M(\D \setminus \D_k),\, \text{ as } T \rightarrow \infty, $$ with respect to the $M_1$ topology, and where, for Borel $B \in \cB(\D \setminus \D_k)$, $$\mu_{k+1}^{\#}(B) \defeq \frac{1}{(k+1)!} \int_{([0,1]\times \R_+ )^{k+1}} \I{ (\sum_{i=1}^{k+1} y_i\I{t_i\le t})_{0\le t\le 1} \in B }  \lambda^{k+1}\diff t_1 \cdots \diff t_{k+1} \mu(\diff y_1)\cdots \mu(\diff y_{k+1})$$ with $\mu(\cdot) \defeq \mu' \circ g^{-1}(\cdot) + \E{K_{A_{i0}}}\mu_{X_{i0}}(\cdot).$
\end{cor}
\begin{proof}
    The proof of this corollary consists of verifying the set of assumptions from Section~\ref{section:res}, and is relegated to Appendix~\ref{app:pcmb}. 
\end{proof}
\begin{rem}
    In Corollary~\ref{cor:mb}, $\mu_{X_{i0}}(\cdot)$ corresponds to the projection on the first coordinate of the measure $\mu'(\cdot).$
\end{rem}

\subsection{Hawkes process}

In the multi-generational Hawkes process introduced in Section~\ref{section:intropp}, the centered cumulative functional corresponding to Equation~\eqref{eq:rthatz}, as a càdlàg process in $\D([0,1], \R)$, defined on $(\Omega, \cF, \P)$, can be written as
\begin{equation}\label{eq:cs2}
\hat{R}^H_T(t) \defeq \sum_{i=1}^{N_0(tT)} \sum_{j=0}^{L_{A_{i0}}} D_{ij}^{\le tT} \, \I{(\Gamma_i + W_{ij})/T \le t} - \E{\sum_{i=1}^{N_0(tT)} \sum_{j=0}^{L_{A_{i0}}} D_{ij} \, \I{(\Gamma_i + W_{ij})/T \le t}}, \quad \text{for } t \in [0,1]. 
\end{equation}
This representation highlights that, starting from the initial immigration sequence $\{(\Gamma_i, A_{i0})\}_{1 \le i \le N_0(T)}$ in $[0,T]$, one may view the process as being composed of $L_{A_{i0}}$ independent self-similar sub-clusters emerging from each immigration event. The total contribution of each of these is captured by the family of independent sub-cluster sums $\{D_{ij}^{\le tT}\}_{i \ge 1, j \ge 0}$. Note that, for the above notation to make sense, we let $D_{i0}^{\le tT} \defeq X_{i0}$, the immigrant mark. 

We state the sufficient assumptions on the Hawkes process in order to obtain a corollary of Theorem~\ref{thm:clustrv}. We denote the following set of assumptions by \textbf{(A. H.nf)}.

\begin{enumerate}
    \item[] \textbf{Assumption 1.} For each $i \ge 1, j \ge 1$, $(X_{ij}, \kappa_{A_{ij}}) \in \text{RV}((\R_+ \setminus \{0\}) \times (\R_+ \setminus \{0\}), v(\cdot), \mu')$, with index $\alpha > 1$. 
    \item[] \textbf{Assumption 2. } Assumption~\ref{assumption2} is assumed for the order of the large deviation scaling sequence $x_T$. 
\end{enumerate}
\begin{rem}
    Recall that the combination of Assumption 1 and Assumption 2 yields a speed $v(x_T)$ consistent with the scaling $x_T.$
\end{rem}
In order to state a forthcoming corollary in non-functional setting, we isolate the following assumption that we denote by \textbf{(A. H.f)}. 
\begin{enumerate}  
    \item[] \textbf{Assumption 6'.} For each $i \ge 1, j \ge 1,$ all $\epsilon > 0$ and some small $\delta >0$, it holds $$x_T \p{W_{ij} > \epsilon T^{1-\delta}} = \smallO(1),\, \text{ as } T \rightarrow \infty.$$
\end{enumerate}

Together, the set of assumptions \textbf{(A. H.nf)} and \textbf{(A. H.f)} is simply denoted by \textbf{(A. H)}.

We can now state the following corollary to Theorem~\ref{thm:clustrv}. Let $g(\cdot, \cdot): \R^2_+ \rightarrow \R_+$ be defined by $g(x, k) = x + (\E{X}/(1-\E{\kappa_A}))k.$

\begin{cor}\label{cor:hawkes}
For the Hawkes process introduced in Section~\ref{section:intropp}, assume Assumptions \textbf{(A. H)} hold. Then, for $k \ge 0$, and $\hat{R}^{H}_T(\cdot)$ in Equation~\eqref{eq:cs1}, 
    $$v'(x_T)^{k+1} \p{x_T^{-1} \hat{R}^{H}_T \in \cdot} \rightarrow \mu_{k+1}^{\#}(\cdot) \text{ in } \M(\D \setminus \D_k),\, \text{ as } T \rightarrow \infty, $$ with respect to the $M_1$ topology, and  where, for Borel $B \in \cB(\D \setminus \D_k)$, $$\mu_{k+1}^{\#}(B) \defeq \frac{1}{(k+1)!} \int_{([0,1]\times \R_+ )^{k+1}} \I{ (\sum_{i=1}^{k+1} y_i\I{t_i\le t})_{0\le t\le 1} \in B }  \lambda^{k+1}\diff t_1 \cdots \diff t_{k+1} \mu(\diff y_1)\cdots \mu(\diff y_{k+1})$$ with $\mu(\cdot) \defeq \mu' \circ g^{-1}(\cdot).$
\end{cor}
\begin{proof}
    The proof of this corollary consists of verifying the set of assumptions from Section~\ref{section:res}, and is relegated to Appendix~\ref{app:pchawkes}. 
\end{proof}

\begin{rem}
Considering the (constant) process $t \in (0,1] \rightarrow \hat{R}_T^{MB}(1)$ for $\hat{R}_T^{MB}(\cdot)$ in Equation~\eqref{eq:cs1} yields \begin{equation}\label{eq:stmb} S_T^{MB} - \E{S_T^{MB}} \defeq \hat{R}^{MB}_T(1) = \sum_{i=1}^{N_0(T)} \sum_{j=0}^{K_{A_{i0}}} X_{ij}  \I{\Gamma_i + W_{ij} \le T} - \E{\sum_{i=1}^{N_0(T)} \sum_{j=0}^{K_{A_{i0}}} X_{ij}  \I{\Gamma_i + W_{ij} \le T}} \end{equation} and, similarly the process $t \in (0,1] \rightarrow \hat{R}_T^{H}(1)$ for $\hat{R}_T^{H}(\cdot)$ in Equation~\eqref{eq:cs2} \begin{equation}\label{eq:sthawkes} S_T^{H} - \E{S_T^{H}} \defeq \hat{R}_T^{H}(1) = \sum_{i=1}^{N_0(T)} \sum_{j=0}^{L_{A_{i0}}} D_{ij}^{\le T} \, \I{\Gamma_i + W_{ij} \le T} - \E{\sum_{i=1}^{N_0(T)} \sum_{j=0}^{L_{A_{i0}}} D_{ij} \, \I{\Gamma_i + W_{ij} \le T}} \end{equation} where the definition of $S_T - \E{S_T}$ matches the notations of Proposition 10 in \cite{bcw23}. Using the elements in Section~\ref{section:res} designed to prove Theorem~\ref{thm:clustrv}, it is possible to state the following corollary. 

\begin{cor}\label{cor:nonfunc}
    For $S_T^{MB}$ defined in Equation~\eqref{eq:stmb}, under Assumptions \textbf{(A. MB. nf)}, it holds, for $k \ge 0$, $$v'(x_T)^{k+1} \p{x_T^{-1} (S_T^{MB} - \E{S_T^{MB}}) \in \cdot } \rightarrow \bar{\mu}_{k+1}^{\#}(\cdot)\, \text{ in } \M(\R_+ \setminus \{0\}),\, \text{ as } T \rightarrow \infty, $$ and, for $S_T^{H}$ defined in Equation~\eqref{eq:sthawkes}, under Assumptions \textbf{(A. H. nf)}, it holds, for $k \ge 0$, $$v'(x_T)^{k+1} \p{x_T^{-1} (S_T^{H} - \E{S_T^{H}}) \in \cdot } \rightarrow \bar{\mu}_{k+1}^{\#}(\cdot)\, \text{ in } \M(\R_+ \setminus \{0\}),\, \text{ as } T \rightarrow \infty, $$ where, for Borel $B \in \cB(\R_+ \backslash \{0\}),$\, $$\bar{\mu}_{k+1}^{\#}(B) = \frac{1}{(k+1)!} \int_{\R_+^{k+1}} \I{\sum_{i=1}^{k+1} y_i \in B} \lambda^{k+1}\mu(\diff y_1) \cdots \mu(\diff y_{k+1}), $$ with $\mu(\cdot)$ defined as in Corollary~\ref{cor:mb}, resp. Corollary~\ref{cor:hawkes}.  
\end{cor}
\begin{proof}
    For brevity, we give a sketch of the proof. The stated result follows from applying a continuous mapping argument: define $\pi:\D\to\R_{+}$, $\pi(f)=f(1)$, which is continuous except at paths that jump at $t=1$, and that discontinuity set has zero mass under both limit measures $\mu_{k+1}^{\#}(\cdot)$ and $\bar{\mu}_{k+1}^{\#}(\cdot) \defeq \pi(\mu_{k+1}^{\#}(\cdot))$ (jump times are a.s. in $(0,1)$). The sets of assumptions \textbf{(A. MB. nf)} and \textbf{(A. H. nf)} allow us to prove all sufficient conditions for both processes to fulfill the assumptions from Section~\ref{section:res} (see the proof of Corollary~\ref{cor:mb} and Corollary~\ref{cor:hawkes} in Appendix~\ref{app:pcmb} and~\ref{app:pchawkes}), except for Assumption 6 which is not needed in the non-functional setting. Indeed, at $t=1$, the proof of Proposition~\ref{prop:t2} relies on showing that the remainder terms of both models are negligible, which is shown by appealing to Lemma~\ref{lem:sbj} which only relies on Assumptions 1 to 5 from Section~\ref{section:res}. 
\end{proof}

By Assumption~\ref{assumption2}, the above corollary holds for scaling sequences $x_T$ corresponding to large deviation regimes $x_T^{-1} = \smallO(T^{-\max(1/\alpha,\, 1/2)})$, which is an improvement over the large deviation principle of Proposition 10 in \cite{bcw23}, which holds over $x-$regions of the form $x \ge \gamma \lambda T$, for every fixed $\gamma > 0.$ However, this latter result holds uniformly over the $x-$region considered, which is not the case of Corollary~\ref{cor:nonfunc}. Finally, note that, for each $k \ge 0$, Corollary~\ref{cor:nonfunc} also provides the correct regular variation speed $v'(x_T)^{k+1}$. 
\end{rem}

\begin{rem}\label{rem:Worder}
We give some intuition into Assumption 6, specifically in the context of the Hawkes process. This condition is the price to pay to establish a functional large deviation principle for $\tilde{R}_T(\cdot)$ with respect to the $M_1$ topology. 
It restricts the moments of $W$ by considering its interplay with $x_T$. Indeed, suppose that $x_T \equiv T^{\eta}$, with $\eta > \max(1/\alpha, 1/2)$ for $\alpha > 1$, its exact order dictating the asymptotic regime in Assumption 2. Then, Assumption 6 requires, for some small $\delta > 0$, that $\E{W^{\eta - \delta}} < \infty.$
This constrains slightly the (potential) heavy-tailedness of $W$, depending on the asymptotic regime considered. 
\end{rem}  

\section{Proof of Result in Subsection~\ref{subsection:slutsky}}\label{section:proof2}

\subsection{Proof of Lemma~\ref{lem:slutsky}}\label{secp:slutsky}
    \begin{proof}[Proof of Lemma~\ref{lem:slutsky}]
          The proof essentially follows the ones of Theorem 3.1 in \cite{billingsley99} and of Proposition 2.4 in \cite{dtw22}. 
    Let $r > 0$, $B \in \cB(E \setminus F^r)$ be a $\mu$-continuity set, and take $E \setminus F^r$ also to be $\mu$-continuous (see Theorem 2.2 in \cite{hl06}). For an upper bound, note that 
    \begin{align*} v(x_T) \p{x_T^{-1} Y_T^{(2)} \in B} &= v(x_T) \p{x_T^{-1} Y_T^{(2)} \in B, d(x_T^{-1} Y_T^{(1)}, x_T^{-1} Y_T^{(2)}) \le \epsilon} \\
    &\qquad + v(x_T) \p{x_T^{-1} Y_T^{(2)} \in B, d(x_T^{-1} Y_T^{(1)}, x_T^{-1} Y_T^{(2)}) > \epsilon} \\
    &\le v(x_T) \p{x_T^{-1}Y_T^{(1)} \in B^{\epsilon}} + v(x_T) \p{x_T^{-1} Y_T^{(2)} \in B, d(x_T^{-1} Y_T^{(1)}, x_T^{-1} Y_T^{(2)}) > \epsilon}
    \end{align*}
    with $B^{\epsilon} = \{x \in E: d(x, B) \le \epsilon \}.$ Take $\epsilon < r/2$ and note that $B^{\epsilon} \in \cB(E \setminus F^{r/2})$. Since  the distribution of the scaled $Y_T^{(1)}$ converges in $\M(E \setminus F)$, it holds that 
    $$\limsup_{T \rightarrow \infty} v(x_T) \p{x_T^{-1} Y_T^{(1)} \in B^{\epsilon}} \le \mu(\text{cl }B^{\epsilon}).$$
    Now, because $B$ is a $\mu$-continuity set, letting $\epsilon \rightarrow 0$ implies by monotone convergence that $$\limsup_{T \rightarrow \infty}v(x_T) \p{x_T^{-1} Y_T^{(1)} \in B^{\epsilon}} \le \mu(\text{cl } B) = \mu(B).$$
    On the other hand, first note that $$\{x_T^{-1}Y_T^{(2)} \in B \} \subseteq \{x_T^{-1} Y_T^{(2)} \in B, d(x_T^{-1} Y_T^{(2)}, F) > r \} \cup \{x_T^{-1} Y_T^{(2)} \in B, d(x_T^{-1} Y_T^{(2)}, F) \le r \} \subseteq \{d(x_T^{-1} Y_T^{(2)}, F) > r \}, $$ the second event being empty by the choice of $B \in \cB(E \setminus F^r).$ 
    By the assumption of the lemma, it follows that 
    \begin{align*}\limsup_{T \rightarrow \infty}\, &v(x_T) \p{x_T^{-1} Y_T^{(2)} \in B, d(x_T^{-1} Y_T^{(1)}, x_T^{-1} Y_T^{(2)}) > \epsilon} \\ &\qquad \le \limsup_{T \rightarrow \infty} v(x_T) \p{d(x_T^{-1} Y_T^{(2)}, F) > r, d(x_T^{-1} Y_T^{(1)}, x_T^{-1} Y_T^{(2)}) > \epsilon} = 0 \end{align*} and hence, that 
    $$\limsup_{T \rightarrow \infty} v(x_T) \p{x_T^{-1} Y_T^{(2)} \in B} \le  \mu(B).$$

    For the lower bound, note that 
    $$v(x_T) \p{x_T^{-1} Y_T^{(2)} \in B} = v(x_T) \p{x_T^{-1} Y_T^{(2)} \in E \setminus F^r} - v(x_T) \p{x_T^{-1} Y_T^{(2)} \in E \setminus F^r \cap B^c }$$ and using the upper bound above, it follows that 
    \begin{align*}
        \liminf_{T \rightarrow \infty} v(x_T) \p{x_T^{-1} Y_T^{(2)} \in B} &\ge \liminf_{T \rightarrow \infty} v(x_T) \p{x_T^{-1} Y_T^{(2)} \in E \setminus F^r} - \limsup_{T \rightarrow \infty} v(x_T) \p{x_T^{-1} Y_T^{(2)} \in E \setminus F^r \cap B^c } \\
        &\ge \liminf_{T \rightarrow \infty} v(x_T) \p{x_T^{-1} Y_T^{(2)} \in E \setminus F^r} - \mu(E \setminus F^r \cap B^c).  
    \end{align*}
    Now, as in the proof of Proposition 2.4 in \cite{dtw22}, one can show that $$\p{x_T^{-1} Y_T^{(2)} \in E \setminus F^r} \ge \p{x_T^{-1} Y_T^{(1)} \in E \setminus F^{r+\epsilon}} - \p{d(x_T^{-1} Y_T^{(1)}, x_T^{-1} Y_T^{(2)}) > \epsilon, x_T^{-1} Y_T^{(1)} \in E \setminus F^{r+\epsilon}}$$
    which yields, by similar arguments as for the upper bound (using the assumption on the $\limsup$) and monotone convergence as $\epsilon \rightarrow 0$, since $E \setminus F^r$ is a $\mu$-continuity set, that
    \begin{align*}
       \liminf_{T \rightarrow \infty} v(x_T) \p{x_T^{-1} Y_T^{(2)} \in E \setminus F^r} \ge \mu(\text{int } E \setminus F^r) = \mu(E \setminus F^r).
    \end{align*}
    All in all, this shows that $$\mu(B) \le \liminf_{T \rightarrow \infty} v(x_T) \p{x_T^{-1} Y_T^{(2)} \in B} \le \limsup_{T \rightarrow \infty} v(x_T) \p{x_T^{-1} Y_T^{(2)} \in B} \le \mu(B).$$
    \end{proof}

\section{Proof of Results of Section~\ref{section:res}}\label{section:proof4}

\subsection{Proof of Proposition~\ref{prop:t1}}\label{app:pt1}
   
\begin{proof}[Proof of Proposition~\ref{prop:t1}]
    This proof is an adaptation of the proof of Theorem 4.3 in \cite{dtw22}. We split the cases for $k$, as the situation where $k=0$ serves as a building block to prove the proposition in full generality. 

    \textbf{Case of $k=0$:}
    
    We start by noticing, from Lemma 6.1 in \cite{dtw22}, that 
    $$\{d_{M_1} (x_T^{-1}\tilde{R}_T, \D_0) > r\} \subseteq \{\Delta_1(x_T^{-1} \tilde{R}_T^{0}) > 2r\} = \{D_{(N)} > 2rx_T \}$$
    where $\Delta_{1}(R)$ denotes the largest jump of $R$ and $\tilde{R}^{0}_T$ is the non-centered version of $\tilde{R}_T$, and just as in \cite{dtw22}, it is possible to neglect the centering term in front of $x_T$, as $T \rightarrow \infty.$ Hence, the distance between the removed cone $\D_0$ and the process is upper-bounded by the largest jump $D_{(N)}$ (among the $N_0(T)$ possible jumps) of the process. 

     Let $c = \E{D}$ and note that it holds that  
\begin{align*}
   \T_1 &\defeq  v'(x_T) \p{d_{M_1}\big(x_T^{-1}\tilde{R}_T,\,  x_T^{-1} D_{(N)} \I{\Gamma_{(N)}/T \le t}\big) > \epsilon,\,  d_{M_1}(x_T^{-1}\tilde{R}_T, \D_0) > r} \notag \\
    &\le v'(x_T) \p{\sup_{0 \le t \le 1} \Big\lvert x_T^{-1} \Big( \sum_{i=1}^{N_0(tT)} \big(D_i - \E{D}\big)\I{\Gamma_{i}/T \le t} - D_{(N)}\I{\Gamma_{(N)}/T\le t} \Big\rvert > \epsilon, \, D_{(N)} > 2rx_T } \\
    &\le v'(x_T)\p{\sup_{0 \le t \le 1} \Big\lvert x_T^{-1} \Big( \sum_{i=1}^{N_0(tT)} \big(D_i - c\big)\I{\Gamma_{i}/T \le t} - \big(D_{(N)} - c \big)\I{\Gamma_{(N)}/T \le t}   \Big) -x_T^{-1}c\I{\Gamma_{(N)}/T<t} \Big\rvert > \epsilon, \, D_{(N)} > 2rx_T } \notag
\end{align*}
where the extra term in the last upper bound comes from the centering of $D_{(N)}$. 

We proceed similarly as in the proof of Theorem 4.3 in \cite{dtw22}, and we consider the classical probability integral transform, for any $i \ge 1$, $D = F^{\leftarrow}_D(U_i)$ for some $U_i \sim \text{Unif}[0,1]:$
{\footnotesize
\begin{align*}
     \T_{1} &\le v'(x_T) \p{\sup_{0 \le t \le 1} \Big\lvert x_T^{-1} \Big( \sum_{i=1}^{N_0(tT)} \big(D_i - c\big)\I{\Gamma_{i}/T \le t} - \big(D_{(N)} - c \big)\I{\Gamma_{(N)}/T \le t}   \Big) - x_T^{-1}c\I{\Gamma_{(N)}/T \le t}  \Big\rvert > \epsilon, \, D_{(N)} > 2r x_T} \\
    &\le v'(x_T)\p{\sup_{0 \le t \le 1} \Big\lvert x_T^{-1} \Big( \sum_{i=1}^{N_0(tT)} \big(F^{\leftarrow}_D(U_i) - c\big)\I{\Gamma_{i}/T \le t} - (F^{\leftarrow}_D(U_{(N)})-c)\I{\Gamma_{(N)}/T \le t} \Big) - x_T^{-1}c\I{\Gamma_{(N)}/T \le t}  \Big\rvert > \epsilon,\, U_{(N)} > F_D(2r x_T) }. 
\end{align*}
}

Now, if $\sigma(j)$ denotes the rank of observation $i$ over the interval $[0,T]$, i.e. $(U_i)_{1 \le i \le N_0(T)}= (U_{(\sigma(i))})_{1 \le i \le N_0(T)}$, we note that there exists a unique permutation $\sigma'$ of the indices $\{1, \ldots, N_0(T)-1\}$ such that the smallest order statistics appear in the same order in the sequences $(U_i)_{1 \le i \le N_0(T)}$ and $(U_{(\sigma'(i))})_{1 \le i \le N_0(T)-1}$, meaning that according to this permutation, the last position actually consists of the largest order statistics. This means that the above upper bound is in fact
\begin{align*}
    \T_{1} &\le v'(x_T)\E{\p{\max_{1 \le t \le N_0(T)-1} \Big\lvert x_T^{-1} \sum_{i=1}^{t} \big(F^{\leftarrow}_D(U_{(\sigma'(i))}) - c\big) - x_T^{-1}c  \Big\rvert > \epsilon/2, \, U_{(N)} > F_D(2r x_T) \mid N_0(T)}}. 
\end{align*}
If we condition on the maximal value $U_{(N)} = u$, the permutation $\sigma'$ over $\{1, \ldots, N_0(T)-1\}$ is uniform and independent of $(U_i)_{1 \le i \le N_0(T)-1}$ and $(V_i)_{1 \le i \le N_0(T)-1} = (U_{(\sigma'(i))}/u)_{1 \le i \le N_0(T)-1}$ is independent uniform on~$[0,1]$. Using Etemadi's inequality, 
\begin{align*}
    \T_{1} &\le 3 v'(x_T) \E{\E{\max_{1 \le t \le N_0(T)-1} \p{\Big\lvert \sum_{i=1}^{t} x_T^{-1}\big(F_D^{\leftarrow}(U_{(N)}V_i) - c\big) - x_T^{-1}c  \Big\rvert > \epsilon/6 } \I{U_{(N)} > F_D(2rx_T) } \mid N_0(T)}  }.
\end{align*}

Using a triangular inequality, one can neglect the additional centering term $x_T^{-1}c = x_T^{-1} \E{D}$ asymptotically, since $D$ is regularly varying with index $\alpha > 1$ by assumption and, hence, $\E{D} < \infty.$

Below, we want to show that \begin{equation}\label{eq:maxcond}
    \max_{1 \le t \le N_0(T)-1} \p{\Big\lvert \sum_{i=1}^{t} x_T^{-1}\big(F_D^{\leftarrow}(U_{(N)}V_i) - c\big) - x_T^{-1}c  \Big\rvert > \epsilon/6 } \le C N_0(T) \p{D > \delta x_T}
\end{equation} for some constant $C > 0$, because, then, it will follow that the above upper bound is further bounded from above by 
\begin{align}
    \T_1 &\le 3 v'(x_T) \E{ C N_0(T) \p{D > \delta x_T} \E{ \I{U_{(N)} > F_D(2rx_T) } \mid N_0(T)}}\label{eq:T1-bound} \\
    &\le 3 C v(x_T)\p{D > \delta x_T} \E{\frac{N_0(T)}{v(x_T)} v'(x_T)\E{ \I{U_{(N)} > F_D(2rx_T) } \mid N_0(T)}} \notag.
\end{align}
Upon using the properties (specifically the density $f_{D_{(N)}}$) of the order statistics $D_{(N)}$ (see e.g. Chapter 2 in~\cite{ans13}), we have 
\begin{align*}
    v'(x_T) \E{\I{D_{(N)}> 2rx_T} \mid N_0(T)} &= v'(x_T) \int_{2rx_T}^{\infty} f_{D_{(N)}}(z) \diff z = v'(x_T) \int_{2rx_T}^{\infty} \frac{N_0(T)!}{(N_0(T)-1)!} F_D(z)^{N_0(T)-1} f_D(z) \diff z \\
    &\le v'(x_T) N_0(T) \p{D > 2rx_T}.
\end{align*}
Now, this yields as a further upper bound, recalling $v'(x_T) = v(x_T) / T$, and using the Poisson property of $N_0$, 
\begin{align*}
    \T_1 &\le 3 C v(x_T)\p{D > \delta x_T} \E{\frac{N_0(T)}{v(x_T)} v'(x_T)\E{ \I{U_{(N)} > F_D(2rx_T) } \mid N_0(T)}} \\
    &\le 3 C \big(v(x_T)\p{D > \delta x_T}\big)\big(v(x_T)\p{D > 2r x_T}\big) \E{\frac{N_0(T)^2}{T} } \\
    &\le 3 C \big(v(x_T)\p{D > \delta x_T}\big)\big(v(x_T)\p{D > 2r x_T}\big) \Big(\frac{\lambda + \lambda^2 T}{v(x_T)}\Big)
\end{align*}
and, because $D$ is regularly varying, $v(x_T)\p{D > \delta x_T} = \cO(1)$ and  $v(x_T) \p{D > 2rx_T} = \cO(1),$ as $T  \rightarrow \infty$, while the choice of $v(x_T)$ (which is regularly varying with index $\alpha > 0$) and by Assumption~\ref{assumption2}, this ensures that $T/v(x_T) = \smallO(1),$ as $T  \rightarrow \infty$, and this allows us to conclude that $$\T_1 = \smallO(1), \text{ as } T \rightarrow \infty.$$

We hence show that the maximum over the inner probability can be controlled as $T \rightarrow \infty$. Fix $U_{(N)} = u$ and note that, for $\delta > 0$, using a union bound, 
\begin{align*}
    \p{\Big\lvert \sum_{i=1}^{t} x_T^{-1}\big(F_D^{\leftarrow}(u V_i) - c\big) \Big\rvert > \epsilon/6 } &\le \p{\Big\lvert \sum_{i=1}^{t} x_T^{-1}\big(F_D^{\leftarrow}(u V_i) \I{F_D^{\leftarrow}(u V_i)  \le \delta x_T} - c\big) \Big\rvert > \epsilon/12 } \\
    &\quad + \p{\Big\lvert \sum_{i=1}^{t} x_T^{-1}\big(F_D^{\leftarrow}(u V_i) \I{F_D^{\leftarrow}(u V_i)  > \delta x_T} - c\big) \Big\rvert > \epsilon/12 } \\
    &\eqdef \T_{11} + \T_{12}.
\end{align*}

We start by treating $\T_{12}$, which is handled in the same way for any $\alpha > 1$; in particular, using Markov inequality of order 1, note that 
\begin{align*}
     \p{\Big\lvert \sum_{i=1}^{t} x_T^{-1}\big(F_D^{\leftarrow}(u V_i) \I{F_D^{\leftarrow}(u V_i)  > \delta x_T} - c\big) \Big\rvert > \epsilon/12 } &\le 12 \epsilon^{-1} \sum_{i=1}^t \E{x_T^{-1} F^{\leftarrow}_D(uV_i) \I{F^{\leftarrow}_D(uV_i) > \delta x_T}}
\end{align*}
and, using a change of variable, one has $$\E{F^{\leftarrow}_D(uV_i) \I{F^{\leftarrow}_D(uV_i) > \delta x_T}} = \int_{F_D(\delta x_T) / u}^{1} F^{\leftarrow}_D(uv) \diff v = u^{-1} \int_{\delta x_T}^{1} x F(\diff x) = u^{-1} \E{D_i \I{D_i > \delta x_T}}.$$
Because the $V_i$s are i.i.d., it follows that 
\begin{align*}
    \p{\Big\lvert \sum_{i=1}^{t} x_T^{-1}\big(F_D^{\leftarrow}(u V_i) \I{F_D^{\leftarrow}(u V_i)  > \delta x_T} - c\big) \Big\rvert > \epsilon/12 } &\le 12 \epsilon^{-1} t u^{-1} \E{x_T^{-1} D \I{D > \delta x_T}}
\end{align*}
Using Karamata's Theorem (see e.g. Proposition 1.4.6 in \cite{ks20}), it follows that
$$\E{x_T^{-1 }D \I{D > \delta x_T}} = \delta \E{(x_T \delta)^{-1} D \I{D > \delta x_T}} \sim \delta  \frac{\alpha}{\alpha-1} \p{D > \delta x_T}, \text{ as } T \rightarrow \infty,$$
which implies that 
\begin{align*}
    \p{\Big\lvert \sum_{i=1}^{t} x_T^{-1}\big(F_D^{\leftarrow}(u V_i) \I{F_D^{\leftarrow}(u V_i)  > \delta x_T} - c\big) \Big\rvert > \epsilon/12 }  &\le 12 \epsilon^{-1} t u^{-1} \delta \frac{\alpha}{\alpha-1} \p{D > \delta x_T} \eqdef C_2(\epsilon, \delta, \alpha, u) t \p{D > \delta x_T}
\end{align*}
for some $C_2(\epsilon, \delta, \alpha, u) > 0$. 



To treat term $\T_{11}$, note it is necessary to replace $c$ by $c_{T, \delta}(u) = \E{F^{\leftarrow}_D(uV_i) \I{F^{\leftarrow}_D(uV_i) \le \delta x_T}}$ (this can always be done -- see matching expression in the proof of Equation 6.6 in \cite{dtw22}). We have to split into two cases: \\ 
\textbf{Case of $\alpha \in (1,2)$:} Using Markov inequality of order 2, we have that 
\begin{align*}
 \p{\Big\lvert \sum_{i=1}^{t} x_T^{-1}\big(F_D^{\leftarrow}(u V_i) \I{F_D^{\leftarrow}(u V_i)  \le \delta x_T} - c\big) \Big\rvert > \epsilon/12 } &\le 12^2 \epsilon^{-2} x_T^{-2}\E{\Big(\sum_{i=1}^t \big(F^{\leftarrow}_D(uV_i) \I{F^{\leftarrow}_D(uV_i) \le \delta x_T} - c_{T,\delta}(u) \big) \Big)^2} \\
 &\le 12^2 \epsilon^{-2}x_T^{-2} t \V{F^{\leftarrow}_D(uV) \I{F^{\leftarrow}_D(uV)\le \delta x_T}}.
\end{align*}
Now, note that $$\V{F^{\leftarrow}_D(uV) \I{F^{\leftarrow}_D(uV)\le \delta x_T}} \le \int_{0}^{F_D(\delta x_T)/u} F^{\leftarrow}_D(uv)^2 \diff v = u^{-1} \E{D^2 \I{D \le \delta x_T}}.$$
Using Karamata's Theorem (see e.g. Proposition 1.4.6 in \cite{ks20}), once again one has that $$x_T^{-2} \E{D^2 \I{D \le \delta x_T}} = \delta^2 \E{\Big((\delta x_T)^{-2} D^2 \I{D \le \delta x_T}} \sim \delta^2 \frac{\alpha}{2-\alpha} \p{D > \delta x_T}, \text{ as } T \rightarrow \infty,$$
which in turn implies for the above probability the upper bound
\begin{align*}
   \p{\Big\lvert \sum_{i=1}^{t} x_T^{-1}\big(F_D^{\leftarrow}(u V_i) \I{F_D^{\leftarrow}(u V_i)  \le \delta x_T} - c\big) \Big\rvert > \epsilon/12 } &\le 12^2 \epsilon^{-2} u^{-1} \delta^2 \frac{\alpha}{2-\alpha} t \p{D > \delta x_T} \eqdef C_1(\epsilon, \delta, \alpha, u) t \p{D > \delta x_T} 
\end{align*}
for some constant $C_1(\epsilon, \delta, \alpha, u) > 0$. 

Now, recollecting terms with the above development for $\T_{12}$, this shows 
\begin{align*}
    \max_{1 \le t \le N_0(T)-1} \p{\Big\lvert \sum_{i=1}^{t} x_T^{-1}\big(F_D^{\leftarrow}(U_{(N)}V_i) - c\big)  \Big\rvert > \epsilon/6 } 
    &\le C(\epsilon, \delta, \alpha, u) N_0(T) \p{D > \delta x_T} 
\end{align*}
where $C(\epsilon, \delta, \alpha, u) \defeq C_1(\epsilon, \delta, \alpha, u) + C_2(\epsilon, \delta, \alpha, u),$ which shows indeed that Equation~\eqref{eq:maxcond} holds in the case $\alpha \in (1,2)$ and we can conclude that $\T_1$ is negligible by the arguments presented above. 

\textbf{Case of $\alpha > 2$:} To treat this case, we use Fuk-Nagaev Inequality (see Equation 2.79 in \cite{petrov95}), which implies in this case that
\begin{align*}
    \p{\Big\lvert \sum_{i=1}^{t} x_T^{-1}\big(F_D^{\leftarrow}(u V_i) \I{F_D^{\leftarrow}(u V_i)  \le \delta x_T} - c\big) \Big\rvert > \epsilon/12 } &\le C_1(p) \epsilon^{-p} t \E{\Big\lvert x_T^{-1}(F^{\leftarrow}_D(uV) \I{F^{\leftarrow}_D(uV) \le \delta x_T}\Big\rvert^p} \\
    &\quad + \exp\Big(- C_2(p) \frac{\epsilon^2}{12^2 t \V{x_T^{-1} F^{\leftarrow}_D(uV) \I{F^{\leftarrow}_D(uV)}}}\Big) \\
    &\le C_1(p) \epsilon^{-p} t  x_T^{-p} \E{\Big\lvert(F^{\leftarrow}_D(uV)^p \I{F^{\leftarrow}_D(uV) \le \delta x_T}\Big\rvert^p} \\
    &\quad + \exp\Big(-C_2(p) \frac{\epsilon^2 x_T^2}{t \V{F^{\leftarrow}_D(uV)}}\Big).
\end{align*}
Note that for the first term in power $p$, one uses again Karamata's Theorem (see e.g. Proposition 1.4.6 in \cite{ks20}) to obtain that 
$$C_1(p) \epsilon^{-p} t  x_T^{-p} \E{\Big\lvert(F^{\leftarrow}_D(uV)^p \I{F^{\leftarrow}_D(uV) \le \delta x_T}\Big\rvert^p} \sim C_1(p) \epsilon^{-p} \delta^p t \p{D > \delta x_T} \eqdef C_1(p, \epsilon, \delta) t \p{D > \delta x_T}.$$ 

Collecting terms with $\T_{12}, $ this yields as an upper bound 
\begin{multline*}
    \max_{1 \le t \le N_0(T)-1}  \p{\Big\lvert \sum_{i=1}^{t} x_T^{-1}\big(F_D^{\leftarrow}(U_{(N)}V_i) - c\big)  \Big\rvert > \epsilon/6 } \\
    \le C(\epsilon, \delta, \alpha, u) N_0(T) \p{D > \delta x_T} + \exp\Big(-C_2(p) \frac{\epsilon^2 x_T^2}{N_0(T)\V{F^{\leftarrow}_D(uV)}}\Big)
\end{multline*}
where $C(\epsilon, \delta, \alpha, u) = C_1(p, \epsilon, \delta) + C_2(\epsilon, \delta, \alpha, u)$. Now, for the second term, writing 
\begin{align*}
    \exp\Big(-C_2(p) \frac{\epsilon^2 x_T^2}{N_0(T)\V{F^{\leftarrow}_D(uV)}}\Big) &= \exp\Big(-C_2(p) \frac{\epsilon^2 x_T^2}{\E{N_0(T)} \V{F^{\leftarrow}_D(uV)}} \cdot \frac{\E{N_0(T)}}{N_0(T)}\Big) \\
    &= \smallO_p(1), \text{ as } T \rightarrow \infty,
\end{align*}
the negligibility in probability following from Assumption~\ref{assumption2} which yields $x_T^2/\E{N_0(T)} \sim x_T^2/T \rightarrow \infty$ as $T  \rightarrow \infty$, and the uniform integrability of $\{N_0(T)/\E{N_0(T)}\}$. 

Hence, wee see that the conditional probability above again satisfies Equation~\eqref{eq:maxcond} and overall this shows once again, in the case $\alpha > 2$, that $$\T_{1} = \smallO(1), \text{ as } T \rightarrow \infty.$$

\textbf{Case of $k \ge 1$:}

We only provide a sketch of the proof, since term $\T_{1}$ is treated exactly as in the case of $k = 0$, by considering 
{\footnotesize
\begin{align*}
     \T_1 &\le v'(x_T)^{k+1} \p{\sup_{0 \le t \le 1} \Big\lvert x_T^{-1} \Big( \sum_{i=1}^{N_0(tT)} \big(D_i - \E{D}\big)\I{\Gamma_{i}\le tT} - \sum_{i=0}^{k} D_{(N-i)}\I{\Gamma_{(N-i)}\le tT}\Big) \Big\rvert > \epsilon, \, \Delta_{k+1}(x_T^{-1}\tilde{R}_T^{0}) > 2r } \\
    &\le v'(x_T)^{k+1}\mathbb{P}\Bigg(\sup_{0 \le t \le 1} \Big\lvert x_T^{-1} \Big( \sum_{i=1}^{N_0(tT)} \big(D_i - c\big)\I{\Gamma_{i}\le tT} - \sum_{i=0}^{k} \big(D_{(N-i)} - c \big)\I{\Gamma_{(N-i)} \le tT}   \Big) - \sum_{i=0}^k x_T^{-1}c\I{\Gamma_{(N-i)}<tT} \Big\rvert > \epsilon, \\
    &\qquad \qquad\qquad\qquad\qquad\qquad\qquad\qquad\qquad\qquad\qquad\qquad\qquad\qquad\qquad\qquad\qquad\qquad\qquad\qquad\qquad \, \Delta_{k+1}(x_T^{-1}\tilde{R}_T^{0}) > 2r \Bigg)
\end{align*}
}
where we let for simplicity $(N-i)$ corresponds to the $(N-i)$th largest order statistics of the exchangeable sequence $\{D_{i}\}_{1 \le i \le N_0(T)}$, $c \defeq \E{D},$ and $\Delta_{k+1}(R)$ denotes the $(k+1)$th largest jump of a generic $R$. 
Compared to the case $k=0$, we have to consider more terms to remove from the initial sum over $N_0(tT)$; hence, the permutation defined in the case of $k=0$ is still valid, and we consider the remaining sum over $\{1, \ldots, N_0(T)-(k+1)\}$ indices, i.e. the upper bound to control for the corresponding expression to $\T_1$ in this case is 
\begin{align*} \T_1 &\le v'(x_T)^{k+1}\E{\p{\max_{1 \le t \le N_0(T)-(k+1)} \Big\lvert x_T^{-1} \sum_{i=1}^{t} \big(F^{\leftarrow}_D(U_{(\sigma'(i))}) - c\big) - x_T^{-1}c  \Big\rvert > \epsilon, \, U_{(N-k)} > F_D(2r x_T) \mid N_0(T)}}. 
\end{align*}
In the general case $k \ge 1$, the equivalent expression to the upper bound in Equation~\eqref{eq:T1-bound}, recalling that $v'(x_T) = v(x_T)/T$, that $N_0(T)$ is Poisson, and using the distribution of the $(k+1)$ largest order statistics in Chapter 2 of \cite{ans13}, is 
\begin{align*}
    \T_1 &\le 3v'(x_T)^{k+1} \E{C N_0(T) \p{D > \delta x_T} \E{\I{U_{(N-k)}>F_D(2rx_T)}\mid N_0(T)}} \\
    &\le 3 C\big(v(x_T) \p{D > \delta x_T}\big)(v(x_T) \p{D > 2r x_T}\big) \E{\frac{N_0(T)^{k+2}}{v(x_T) T^{k+1}}} \\
    &\le  3 C\big(v(x_T) \p{D > \delta x_T}\big)(v(x_T) \p{D > 2r x_T}\big) \frac{\sum_{j=1}^{k+2} B_{j,k} (\lambda T)^{j}}{v(x_T) T^{k+1}}
\end{align*}
where $\{B_{j,k}\}_j$ are the Stirling numbers of the second kind. Because $D$ is regularly varying, $v(x_T) \p{D > \delta x_T} = \cO(1)$ and $v(x_T) \p{D > 2r x_T}\big) = \cO(1)$ as $T  \rightarrow \infty,$ and, by the choice of $v(x_T)$, just as in the case $k=0$, because the last term is dominated by $$\frac{(\lambda T)^{k+2}}{v(x_T) T^{k+1}} = \smallO(1), \text{ as } T \rightarrow \infty,$$ which implies that $\T_1 = \smallO(1), \text{ as } T \rightarrow \infty,$ and this concludes the sketch of the proof in the case $k \ge 1$. 

\end{proof}

\subsection{Proof of Proposition~\ref{prop:t2}}\label{proof:t2}

In order to prove Proposition~\ref{prop:t2}, we need the following conditional adaptation of Proposition 3.1 in \cite{olvera12} applied to the remainder term considered in Equation~\eqref{eq:remaindert}, which satisfies Assumptions~\ref{assumption3} to~\ref{assumption6} in Section~\ref{section:res}. 


\begin{lem}\label{lem:sbj}
Let $T > 0$ and $D = \sum_{j=0}^{K} Z_j$ be a generic cluster, satisfying Assumptions~\ref{assumption1}, and with $x_T$ satisfying Assumption~\ref{assumption2}. Let $D^{>T} = \sum_{j=1}^{\mfM} Z_j \I{Q_j > T - \Gamma}$ be such that $(\mfM, \{Q_j\}_{j \ge 1})$ is generically dependent, with $\{Q_j\}_{j \ge 1}$ a conditionally independent sequence given some $\sigma-$algebra $\cF_{CI}$, with $\mfM$ such that $\p{\mfM > x} \sim c_1 \p{D > x}, $ as $x \rightarrow \infty$ for $c_1 \in [0, \infty)$, and $\{Z_{j}\}_{j \ge 1}$ be a nonnegative i.i.d. sequence such that $\p{Z_j > x} \sim c_2 \p{D > x}, $ as $x \rightarrow \infty$ for $c_2 \in [0, \infty)$, independent from $(\{Q_j\}_{j \ge 1}, \mfM, \Gamma)$. Then, it holds that
    $$\p{\sum_{j=1}^{\mfM} Z_j \I{Q_j > T- \Gamma} > x_T \mid D > x_T} = \smallO(1), \text{ as } T \rightarrow \infty.$$
\end{lem}
\begin{rem}
    The setup of Lemma~\ref{lem:sbj} ensures the quantities in $D^{>T}$ further satisfy Assumptions~\ref{assumption3} to~\ref{assumption6}. 
\end{rem}
\begin{proof}
We give a sketch of the proof, since it essentially follows the lines of the lemmas used to prove Proposition 3.1 in \cite{olvera12}. Let $0 < 1/\sqrt{\log(x_T)} \le \delta < 1/2$ and note 
\begin{align*}
    \p{\sum_{j=1}^\mfM Z_j \I{Q_j > T - \Gamma} > x_T \mid D > x_T} &\le \p{\sum_{j=1}^\mfM Z_j \I{Q_j> T - \Gamma} > x_T, \big(\E{Z}+\delta \big) \sum_{j=1}^\mfM \I{Q_j> T - \Gamma} \le x_T \mid D > x_T} \\
    &\quad + \p{\big(\E{Z}+\delta \big) \sum_{j=1}^\mfM \I{Q_j> T - \Gamma} > x_T \mid D > x_T} \\
    &\eqdef \T_1 + \T_2.
\end{align*}
We first treat term $\T_1$. Let $J_\mfM(y) \defeq \text{card}\{1 \le j \le \mfM: Z_j \I{Q_j> T - \Gamma} > y\}$. \\ We further decompose the term of interest as:
\begin{align*}
    \T_1 &\le \p{\sum_{j=1}^\mfM Z_j \I{Q_j> T - \Gamma} > x_T, \big(\E{Z}+\delta \big) \sum_{j=1}^\mfM \I{Q_j> T - \Gamma} \le x_T, \, J_\mfM((1-\delta)x_T) = 0  \mid D > x_T} \\
    &\qquad + \p{J_{\mfM}((1-\delta)x_T) \ge 1 \mid D > x_T} \\
    &\eqdef \T_{11} + \T_{12}. 
\end{align*}
Start with $\T_{12}$. Consider the $\sigma-$algebra $\cF_0 \defeq \sigma(\mfM, \Gamma, Q_1, Q_2, \ldots)$, and using this filtration, note that, by a union bound, 
\begin{align*}
    \T_{12} &\le \frac{\p{J_{\mfM}((1-\delta)x_T) \ge 1}}{\p{D > x_T}} \\
    &\le \frac{\E{\E{\I{\bigcup_{j=1}^{\mfM} \{Z_j \I{Q_j> T - \Gamma} > (1-\delta)x_T\}} \mid \cF_0 }}}{\p{D > x_T}} \\
    &\le \frac{\E{ \sum_{j=1}^\mfM \E{\I{\{Z_j \I{Q_j> T - \Gamma} > (1-\delta)x_T\}} \mid \cF_0 }}}{\p{D > x_T}} \\
    &\le \frac{\E{ \sum_{j=1}^\mfM \I{Q_j> T - \Gamma} \p{Z_j > (1-\delta)x_T}}}{\p{D > x_T}} \\
    &\le \frac{\p{Z > (1-\delta)x_T}}{\p{D > x_T}} \E{\sum_{j=1}^\mfM \I{Q_j> T - \Gamma}}
\end{align*}
Combining Assumptions~\ref{assumption1} and~\ref{assumption4} yields $$\frac{\p{Z > (1-\delta)x_T}}{\p{D > x_T}} = \cO(1), \text{ as } T \rightarrow \infty,$$ while the expectation in the above product converges, because $\mfM$ has finite expectation, and we appeal to Assumption~\ref{assumption5}, which guarantees $\p{Q_j > T - \Gamma\mid \cF_{CI}} = \smallO(1) \text{ a.s.},$ as $T \rightarrow \infty.$ This shows that $$\T_{12} = \smallO(1), \text{ as } T \rightarrow \infty.$$

For $\T_{11}$, we need finer granularity.  Decompose it as 
\begin{align*}
    \T_{11} &\le \p{\sum_{j=1}^\mfM Z_j \I{Q_j> T - \Gamma} > x_T, \sum_{j=1}^\mfM \I{Q_j> T - \Gamma} \le \frac{x_T}{\log(x_T)}, \, J_\mfM((1-\delta)x_T) = 0  \mid D > x_T} \\
    &\quad + \p{\sum_{j=1}^\mfM Z_j \I{Q_j> T - \Gamma} > x_T, \frac{x_T}{\log(x_T)} \le \sum_{j=1}^\mfM \I{Q_j> T - \Gamma} \le \frac{x_T}{\E{Z}+\delta}, \, J_\mfM((1-\delta)x_T) = 0  \mid D > x_T} \\
    &\eqdef \T_{111} + \T_{112}. 
\end{align*}
For the first term, we need even more granularity, by the introduction of a lower threshold for the number of points after the temporal boundary $T$: 
\begin{align*}
  \T_{111} &\le  \p{\sum_{j=1}^\mfM Z_j \I{Q_j> T - \Gamma} > x_T,  \sum_{j=1}^\mfM \I{Q_j> T - \Gamma} \le \frac{x_T}{\log(x_T)}, \, J_\mfM((1-\delta)x_T) = 0, \, J_\mfM\Big(\frac{x_T}{\log(x_T)}\Big) = 0  \mid D > x_T} \\
  &\quad + \p{\sum_{j=1}^\mfM Z_j \I{Q_j> T - \Gamma} > x_T, \sum_{j=1}^\mfM \I{Q_j> T - \Gamma} \le \frac{x_T}{\log(x_T)}, \, J_\mfM((1-\delta)x_T) = 0, \, J_\mfM\Big(\frac{x_T}{\log(x_T)}\Big) = 1  \mid D > x_T} \\
  &\quad + \p{\sum_{j=1}^\mfM Z_j \I{Q_j> T - \Gamma} > x_T, \sum_{j=1}^\mfM \I{Q_j> T - \Gamma} \le \frac{x_T}{\log(x_T)}, \, J_\mfM((1-\delta)x_T) = 0, \, J_\mfM\Big(\frac{x_T}{\log(x_T)}\Big) \ge 2  \mid D > x_T} \\
  &\eqdef \T_{1111} + \T_{1112} + \T_{1113}.
\end{align*}
For the first two terms above, let $$\overline{Z_j \mathbbm{1}} = Z_j \I{Q_j> T - \Gamma} \I{Z_j \I{Q_j> T - \Gamma} \le (\frac{x_T}{\log x_T})} \text{ and } \underline{Z_j \mathbbm{1}} = Z_j \I{Q_j> T - \Gamma} \I{Z_j \I{Q_j> T - \Gamma} > (\frac{x_T}{\log x_T})},$$ and note that, for the first one,
\begin{align*}
    \T_{1111} &= \p{\sum_{j=1}^\mfM \overline{Z_j \mathbbm{1}} > x_T, \, \sum_{j=1}^\mfM \I{Q_j> T - \Gamma} \le \frac{x_T}{\log(x_T)}, \, J_\mfM\Big(\frac{x_T}{\log(x_T)}\Big) = 0 \mid D > x_T} \\
    &\le \frac{\p{\sum_{j=1}^\mfM \overline{Z_j \mathbbm{1}} > x_T, \, \sum_{j=1}^\mfM \I{Q_j> T - \Gamma} \le \frac{x_T}{\log(x_T)}}}{\p{D > x_T}} 
\end{align*}
and, for the second term, 
\begin{align*}
  \T_{1112} &= \p{\sum_{j=1}^\mfM (\overline{Z_j\mathbbm{1}} + \underline{Z_j \mathbbm{1}}) > x_T, \sum_{j=1}^\mfM \I{Q_j> T - \Gamma} \le \frac{x_T}{\log(x_T)}, \, J_\mfM((1-\delta)x_T) = 0, \, J_\mfM\Big(\frac{x_T}{\log(x_T)}\Big) = 1  \mid D > x_T} \\
  &\le \p{\sum_{j=1}^\mfM \overline{Z_j\mathbbm{1}}  > \delta x_T, \sum_{j=1}^\mfM \I{Q_j> T - \Gamma} \le \frac{x_T}{\log(x_T)}\mid D > x_T} \\
  &\le \frac{\p{\sum_{j=1}^\mfM \overline{Z_j\mathbbm{1}}  > \delta x_T, \, \sum_{j=1}^\mfM \I{Q_j> T - \Gamma} \le \frac{x_T}{\log(x_T)}}}{\p{D > x_T}} \eqdef \frac{\T_{1112}'}{\p{D > x_T}}.
\end{align*}
We start with the observation that 
$\T_{1111} \le \T_{1112}$.
We apply Lemma~3.4 in \cite{olvera12}: in this result, an event of the form \(\{I_K(w) = 0\}\), which tracks exceedances of level some level $w$ by the weights $C_i$;  is considered; however, in our setting, the weights considered are indicators
$$C_j \equiv \I{Q_j> T - \Gamma},$$
which implies that the event \(I_\mfM(w)=0\) is trivially bounded above by 1. Therefore, we can ignore the event \(\{I_\mfM(w) = 0\}\) and no further bounding as in Lemma~3.2 of \cite{olvera12} is needed. In essence, we are allowed to take \(u \equiv 1\) in the Lemmas cited in the final expressions obtained, providing some minor constants are adjusted in the quantities throughout.

Lemma 3.2 in \cite{olvera12} directly yields that, for any \(\eta > 0\), we have, for the numerator in $\T_{1112}$,
$$\T_{1112}' = \smallO\bigl(x_T^{-\eta}\bigr),\, \text{ as } T \rightarrow \infty.$$
Since this conclusion holds in particular for $\eta = \alpha$, it follows that
$$\T_{1111} = \T_{1112} = \smallO(1), \, \text{ as } T \rightarrow \infty.$$

Consider now term $\T_{1113}.$ To show negligibility, we proceed as in the proof of Lemma 3.3 in \cite{olvera12}, neglecting once again the event $\{I_\mfM(w)=0\}$ therein, which yields, using the above-defined filtration $\cF_0$ and the independence of $\mfM$ and $\{W_j\}$ from $\{Z_j\}$,
\begin{align*}
    \T_{1113} &\le \frac{\E{\I{\sum_{j=1}^\mfM \I{Q_j> T - \Gamma} \le \frac{x_T}{\log(x_T)}} \E{\I{\bigcup_{1 \le i < j \le \mfM} \{Z_i \I{Q_i > T - \Gamma} > \frac{x_T}{\log(x_T)}, \,Z_j \I{Q_j> T - \Gamma} > \frac{x_T}{\log(x_T)} \}} \mid \cF_0 }}}{\p{D > x_T}} \\
    &\le \frac{\E{\I{\sum_{j=1}^\mfM \I{Q_j> T - \Gamma} \le \frac{x_T}{\log(x_T)}} \sum_{1 \le i < j \le \mfM}\E{\I{ \{Z_i \I{Q_i > T - \Gamma} > \frac{x_T}{\log(x_T)}, \,Z_j \I{Q_j> T - \Gamma} > \frac{x_T}{\log(x_T)} \}} \mid \cF_0 }}}{\p{D > x_T}} \\
    &\le \frac{\E{\I{\sum_{j=1}^\mfM \I{Q_j> T - \Gamma} \le \frac{x_T}{\log(x_T)}} \Big(\sum_{j = 1}^\mfM \I{Q_j> T - \Gamma} \p{Z_j > \frac{x_T}{\log(x_T)}}\Big)^2}}{\p{D > x_T}}. 
\end{align*}
Now, split the product inside the expectation above, and consider one element of it; by Markov inequality, using the moment assumption on $Z$ which, by Assumption~\ref{assumption4}, has at least moment of order $\alpha > 1+\epsilon > 1$ for some $\epsilon >0,$ 
\begin{align*}
    &\I{\sum_{j=1}^\mfM \I{Q_j> T - \Gamma} \le \frac{x_T}{\log(x_T)}}\sum_{j = 1}^\mfM \I{Q_j> T - \Gamma} \p{Z_j > \frac{x_T}{\log(x_T)}} \\
    &\qquad \qquad \le \Big(\frac{x_T}{\log(x_T)} \Big)^{-1-\epsilon} \E{\lvert Z \rvert^{1+\epsilon}} \I{\sum_{j=1}^\mfM \I{Q_j> T - \Gamma} \le \frac{x_T}{\log(x_T)}}\sum_{j=1}^\mfM \I{Q_j> T - \Gamma} \\
    &\qquad\qquad \le \Big(\frac{x_T}{\log(x_T)} \Big)^{-\epsilon} \E{\lvert Z \rvert^{1+\epsilon}}.
\end{align*}
Overall, this yields, using Assumptions~\ref{assumption1} and~\ref{assumption4}, recalling that $\p{D > x} \sim x^{-\alpha} \ell_D(x),$ as $x \rightarrow \infty$, 
\begin{align*}
    \T_{1113} &\le \frac{\Big(\frac{x_T}{\log(x_T)} \Big)^{-\epsilon} \E{\lvert Z \rvert^{1+\epsilon}} \p{Z > \frac{x_T}{\log(x_T)}} \E{\sum_{j=1}^\mfM \I{Q_j> T - \Gamma}}}{\p{D > x_T}} \\
    &\le \frac{\Big(\frac{x_T}{\log(x_T)} \Big)^{-\epsilon} \E{\lvert Z \rvert^{1+\epsilon}} \Big(\frac{x_T}{\log(x_T)} \Big)^{-\alpha} c\ell_{D}\Big(\frac{x_T}{\log(x_T)}\Big) \E{\sum_{j=1}^\mfM \I{Q_j> T - \Gamma}}}{x_T^{-\alpha} \ell_{D}(x_T)} \\
    &\le \frac{\log(x_T)^{\alpha-\epsilon}}{x_T^{\epsilon}} \E{\lvert Z \rvert^{1+\epsilon}}\E{\sum_{j=1}^\mfM \I{Q_j> T - \Gamma}} c\ell_{D}\Big(\frac{x_T}{\log(x_T)}\Big) / \ell_{D}(x_T) \\
    &= \smallO(1),\, \text{ as } T \rightarrow \infty.
\end{align*}
Combining the above shows that $$\T_{111} = \smallO(1), \text{  as } T \rightarrow \infty. $$

We turn to term $\T_{112}.$ We use Lemma 3.5 in \cite{olvera12} to note that 
\begin{align*}
    \T_{112} &\le \frac{\E{\I{\frac{x_T}{\log(x_T)} < \sum_{j=1}^\mfM \I{Q_j> T - \Gamma} \le x/(\E{Z}+\delta)} \exp\Big(-\theta x_T + \big(\E{Z} + \frac{C}{\log((1-\delta)x_T)} \big) \theta \sum_{j=1}^\mfM \I{Q_j> T - \Gamma}  \Big)^+} }{\p{D > x_T}}
\end{align*}
where $C > 0$ is a constant that only depend on the expectation of $Z$, $\theta = \frac{\epsilon}{(1-\delta)x_T} \log((1-\delta)x_T)$ for some $\epsilon > 0$, and $(\cdot)^+$ is the positive part of $(\cdot)$. Using similar techniques as in \cite{olvera12}, it can then be shown that 
\begin{align*}
    \T_{112} &\le \frac{x_T^{-\epsilon/2} \p{\sum_{j=1}^\mfM \I{Q_j> T - \Gamma} > \frac{x_T}{\log(x_T)}} + \exp(- \frac{\epsilon \sqrt{\log(x_T)}}{\E{Z}}) \p{\sum_{j=1}^\mfM \I{Q_j> T - \Gamma} > x_T/(2 \E{Z})}}{\p{D > x_T}} \\
    &= \smallO(1), \text{ as } T \rightarrow \infty.
\end{align*}
Combining with the above shows that 
$$\T_{11} = \smallO(1),\, \text{ as } T \rightarrow \infty.$$

Finally, term $\T_2$ can be treated by applying, sequentially, Markov inequality of order 1, using the conditional independence of $\{Q_j\}$ and $\mfM$ given some $\sigma-$algebra $\cF_{CI}$ specified in Assumption~\ref{assumption5}, Hölder inequality of order $p < \alpha$, with $1/p + 1/q = 1$, and the fact that $\{Q_j\}$ is conditionally independent from $D$ given $\cF_{CI}$, which includes all the required information. Letting $c= (\E{Z}+\delta),$ we note  
\begin{align*}
    \T_2 &\le \frac{\E{\sum_{j=1}^\mfM \I{Q_j> T - \Gamma} \mid D > x_T}}{cx_T} \\
    &\le \frac{\E{ (\mfM/x_T) \p{Q  > T - \Gamma \mid \cF_{CI}} \mid D > x_T}}{c} \\
    &\le \frac{\E{ (\mfM/x_T)^p \mid D > x_T}^{1/p}  \E{\p{Q > T - \Gamma \mid \cF_{CI}}^q\mid D > x_T}^{1/q} }{c} \\
    &\le \frac{\E{ (\mfM/x_T)^p \mid D > x_T}^{1/p}  \p{Q  > T - \Gamma}^{1/q} }{c}.
\end{align*}
Now, note that $(\mfM, D)$ is jointly regularly varying with index $\alpha > 1$ by Assumptions~\ref{assumption1} and~\ref{assumption3}: this can solely follow from the assumption that $D$ is regularly varying with index $\alpha > 1$, and, for joint regular variation, if one coordinate is regularly varying, then the vector is (by applying the characterisation through the regular variation of linear combinations of the elements of the vector, found in Theorem 1.1 of \cite{bdm02}). It then follows, by an application of Corollary 2.1.10 in \cite{ks20}, noting that $p < \alpha$, that
$$\E{(\mfM/x_T)^p \mid D > x_T} = \cO(1), \text{ as } T \rightarrow \infty.$$
Finally, by Assumption~\ref{assumption5}, since it also holds unconditionally that $\p{Q>T - \Gamma} = \smallO(1)$, as $T \rightarrow \infty$, we conclude $$\T_2 = \smallO(1), \text{ as } T \rightarrow \infty.$$
The proof of the lemma is complete.
\end{proof}

We are now able to prove Proposition~\ref{prop:t2}. 

\begin{proof}[Proof of Proposition~\ref{prop:t2}]
    We start by the case $k = 0$, which is a building block for the general case. 
    
\textbf{Case of $k=0$:}

The expression to control is 
$$
\T_2 = v'(x_T)\p{d_{M_1}\Big(x_T^{-1} D_{(N)}  \I{\Gamma_{(N)}/T \le t},\,  x_T^{-1} \sum_{j=0}^{\mfK_{(N)}} Z_{(N)j}  \I{(\Gamma_{(N)} + W_{(N)j})/T \le t}\Big) > \varepsilon,\,  d_{M_1}(x_T^{-1}\tilde{R}_T, \D_0) > r}.
$$

Just as in the proof of Proposition~\ref{prop:t1}, using Lemma 6.1 in \cite{dtw22}, it is possible to bound the second event in the above probability, and the claim will follow if we can show the overall negligibility of the upper bound
$$
\T_2 \le  v'(x_T)\p{d_{M_1}\Big(x_T^{-1} D_{(N)}  \I{\Gamma_{(N)}/T \le t},\,  x_T^{-1} \sum_{j=0}^{\mfK_{(N)}} Z_{(N)j}  \I{(\Gamma_{(N)} + W_{(N)j})/T \le t}\Big) > \varepsilon,\, D_{(N)} > 2rx_T}.
$$

To control the $M_1$ distance in the first part of the probability above, we need to construct the parametric representations of the stochastic processes considered, as defined in Section~\ref{m1}. We let the parametric representation of 
$$
x_1(t) \defeq x_T^{-1} D_{(N)} \I{\Gamma_{(N)} \le tT}
$$ 
be denoted by $(r, u)$ and the one of 
$$
x_2(t) \defeq x_T^{-1} \sum_{j=0}^{\mfK_{(N)}} Z_{(N)j}  \I{(\Gamma_{(N)} + W_{(N)j})/T \le t}
$$ 
by $(r_T, u_T).$ 

To construct the parametric representations, we suppose that the cluster $D_{(N)}$ is unfinished by time $T$; the case where it is finished by this time is obtained easily from it, since by construction, the spatial difference will be $0$, as is clear from below. Recall $K_{(N)}$ denotes the total size of the largest cluster $D_{(N)}$; hence, it means that there exists some $j \in \{1, 2, \ldots, K_{(N)}\}$ such that 
$$
(\Gamma_{(N)} + Q_{(N)j})/T > 1,
$$ 
recalling that, in Assumption~\ref{assumption5} and Assumption~\ref{assumption6}, for any event $j$ of the cluster of interest, $Q_{(N)j}$ is a notation localising the event by the generational depth of event $j$.  

We define the following rescaled times:
\begin{enumerate}
    \item $s_{\Gamma} := \Gamma_{(N)}/T$, the (rescaled) start time of the cluster;
    
    \item $s_{\max} :=  \max_{1 \le j \le K_{(N)}} \left\{(\Gamma_{(N)} + Q_{(N)j} \, \I{Q_{(N)j} \le T})/T \right\}$, 
    the latest (rescaled) time of an event in the cluster occurring before $T$;
    
    \item $s_{\epsilon} := s_{\max} + \epsilon$, for some fixed $\epsilon > 0$.
\end{enumerate}
Hence, by construction, we have
$$
0 < s_{\Gamma} < s_{\max} < s_{\epsilon} < 1.
$$

We construct $(u, u_T)$ and $(r, r_T)$ as follows:
\begin{enumerate}
    \item[] \textbf{Spatial parts $u, u_T$, represented in Figure~\ref{fig:both}\subref{fig1}:}
    \begin{enumerate}
        \item For $s \in [0, s_{\max}]$, we let
        $$
        u(0) = u_T(0) = 0, \quad u(s_{\max}) = u_T(s_{\max}) = \sum_{j=0}^{\mfK_{(N)}} Z_{(N)j} \I{\Gamma_{(N)} + W_{(N)j} \le T},
        $$
        and interpolate linearly in between, for $s \in (0, s_{\max})$. At $s_{\max}$ begins the spatial discrepancy, as $u$ has yet to reach the total cluster height $x_T^{-1} D_{(N)}$, whereas $u_T$ is blind to everything occurring after time $T$:
        \begin{enumerate}
            \item We let
            $$
            u_T(s) = \sum_{j=0}^{\mfK_{(N)}} Z_{(N)j} \I{\Gamma_{(N)} + W_{(N)j} \le T}, \quad \text{for } s \in [s_{\max}, 1],
            $$
            stay constant;
            \item We let
            $$
            u(s_{\epsilon}) = \sum_{j=0}^{\mfK_{(N)}} Z_{(N)j} \I{\Gamma_{(N)} + W_{(N)j} \le T} + \sum_{j=1}^{\mfM_{(N)}} Z_{(N)j} \I{Q_{(N)j} > T - \Gamma_{(N)}},
            $$
            where the second term is the generic remainder term of Equation~\eqref{eq:remaindert} considered in Assumption~\ref{assumption5}. We interpolate linearly in between, for $s \in [s_{\max}, s_{\epsilon}]$; then,
            $$
            u(s) = u(s_{\epsilon}) = \sum_{j=0}^{\mfK_{(N)}} Z_{(N)j} \I{\Gamma_{(N)} + W_{(N)j} \le T} + \sum_{j=1}^{\mfM_{(N)}} Z_{(N)j} \I{Q_{(N)j} > T - \Gamma_{(N)}}, \quad \text{for } s \in (s_{\epsilon}, 1],
            $$
            and keep $u$ constant on that interval.
        \end{enumerate}
    \end{enumerate}

    \item[] \textbf{Temporal parts $r, r_T$, represented in Figure~\ref{fig:both}\subref{fig2}:}
    \begin{enumerate}
        \item For $s \in [0, s_{\Gamma}]$, we let
        $$
        r(0) = r_T(0) = 0, \quad r(s_{\Gamma}) = r_T(s_{\Gamma}) = \Gamma/T,
        $$
        and interpolate linearly between these points. At $s_{\Gamma}$ begins the temporal discrepancy: $r$ stays constant to represent the instantaneous jump of $x_1$, while $r_T$ increases at event times within $[0, T]$ (or rescaled $[0, 1]$):
        \begin{enumerate}
            \item We let
            $$
            r(s) = \Gamma/T, \quad \text{for } s \in [s_{\Gamma}, s_{\epsilon}];
            $$
            \item For $s \in [s_{\Gamma}, s_{\max}]$, we let $r_T$ interpolate between the jump times
            $$
            (\Gamma_{(N)} + Q_{(N)1})/T,\; (\Gamma_{(N)} + Q_{(N)2})/T,\; \ldots,\; (\Gamma_{(N)} + Q_{(N)j})/T,
            $$
            where the maximum is taken over all $j$ such that $\Gamma_{(N)} + Q_{(N)j} \le T$. Then, we let
            $$
            r_T(s) = \max_j (\Gamma_{(N)} + Q_{(N)j})/T, \quad \text{for } s \in [s_{\max}, s_{\epsilon}];
            $$
            \item Finally, we let
            $$
            r(s_{\epsilon} + \delta) = r_T(s_{\epsilon} + \delta) = \max_j (\Gamma_{(N)} + Q_{(N)j})/T
            $$
            for some $\delta > 0$, and we let both grow together linearly for $s \in [s_{\epsilon} + \delta, 1]$.
        \end{enumerate}
    \end{enumerate}
\end{enumerate}

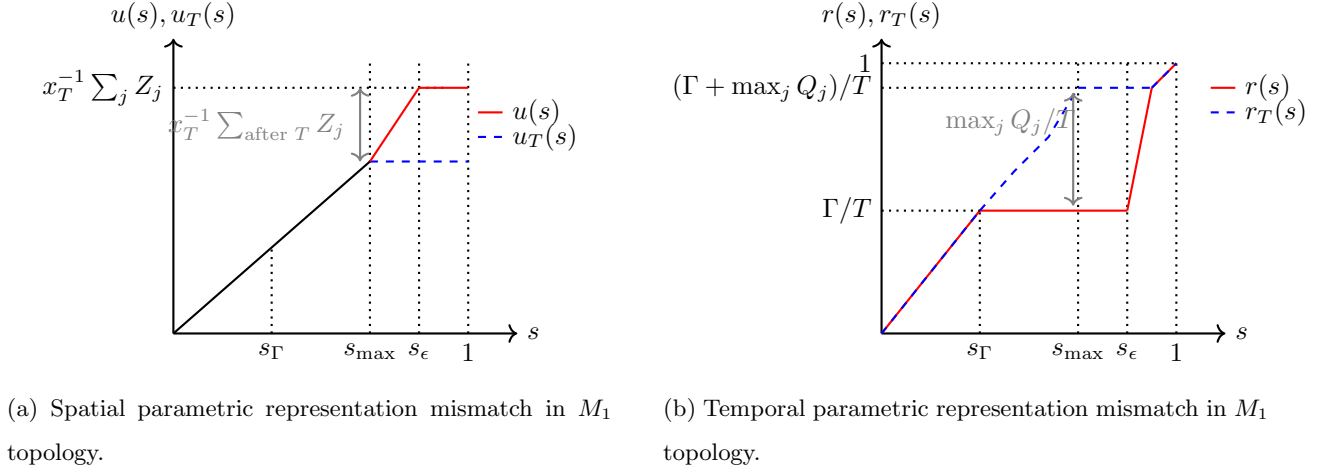
\begin{figure}[ht!]
\centering
\begin{subfigure}{0.48\textwidth}
    \centering
    \begin{tikzpicture}[scale=0.65, thick]  
    \draw[->] (0,0) -- (7,0) node[right] {$s$};
    \draw[->] (0,0) -- (0,6) node[above] {$u(s), u_T(s)$};
    \draw[dotted] (6,0) -- (6,5.5);
    \node[below] at (6,0) {$1$};

    \node[left] at (0,5) {$x_T^{-1}\sum_{j} Z_j$};
    \draw[dotted] (0,5) -- (5.5,5);

    \draw[dotted] (4,0) -- (4,5.5);
    \node[below] at (4,0) {$s_{\max}$};

    \draw[dotted] (5,0) -- (5,5.5);
    \node[below] at (5,0) {$s_{\epsilon}$};

    \draw[dotted] (2,0) -- (2,1.8);
    \node[below] at (2,0) {$s_{\Gamma}$};
    
    \draw[thick] (0,0) -- (4,3.5);
    \draw[blue,thick, dashed] (4,3.5) -- (6, 3.5);
    \draw[red,thick] (4,3.5) -- (5,5)  -- (6,5);
    \draw[gray,<->] (3.8,3.5) -- (3.8,5);
    \node[left,gray] at (3.7,4.25) {$x_T^{-1}\sum_{\text{after } T} Z_j$};
    \draw[red,thick] (6.2,4.5) -- (6.7,4.5);
    \node[right] at (6.7,4.5) {$u(s)$};
    \draw[blue,thick,dashed] (6.2,4) -- (6.7,4);
    \node[right] at (6.7,4) {$u_T(s)$};
    \end{tikzpicture}
    \caption{Spatial parametric representation mismatch in $M_1$ topology.}
    \label{fig1}
\end{subfigure}
\hfill
\begin{subfigure}{0.48\textwidth}
    \centering
    \begin{tikzpicture}[scale=0.65, thick]  
    \draw[->] (0,0) -- (7,0) node[right] {$s$};
    \draw[->] (0,0) -- (0,6) node[above] {$r(s), r_T(s)$};
    \draw[dotted] (6,0) -- (6,5.5);
    \node[below] at (6,0) {$1$};
    \draw[dotted] (2,0) -- (2,2.5);
    \node[below] at (2,0) {$s_{\Gamma}$};

    \draw[dotted] (4,0) -- (4,5.5);
    \node[below] at (4,0) {$s_{\max}$};

    \draw[dotted] (5,0) -- (5,5.5);
    \node[below] at (5,0) {$s_{\epsilon}$};
    
    \draw[red,thick] (0,0) -- (2,2.5);
    \draw[blue,thick,dashed] (0,0) -- (2,2.5);
    \draw[red,thick] (2,2.5) -- (5,2.5);
    \draw[blue,thick,dashed] (2,2.5) -- (2.7,3.3) -- (3.4,4) -- (4,5);
    \draw[gray,<->] (3.9,2.6) -- (3.9,4.9);
    \node[right,gray] at (1.1,4.3) {$\max_j Q_j/T$};
    \draw[red,thick] (5,2.5) -- (5.5,5);
    \draw[blue,thick,dashed] (4,5) -- (5.5,5);
    \draw[red,thick] (5.5,5) -- (6,5.5);
    \draw[blue,thick,dashed] (5.5,5) -- (6,5.5);
    \node[left] at (0,2.5) {$\Gamma/T$};
    \draw[dotted] (0,2.5) -- (2,2.5);
    \node[left] at (0,5.5) {$1$};
    \draw[dotted] (0,5.5) -- (6,5.5);
    \node[left] at (0,5) {$(\Gamma + \max_j  Q_j) / T$};
    \draw[dotted] (0,5) -- (4,5);
    \draw[red,thick] (6.7,5) -- (7.2,5);
    \node[right] at (7.2,5) {$r(s)$};
    \draw[blue,thick,dashed] (6.7,4.5) -- (7.2,4.5);
    \node[right] at (7.2,4.5) {$r_T(s)$};
    \end{tikzpicture}
    \caption{Temporal parametric representation mismatch in $M_1$ topology.}
    \label{fig2}
\end{subfigure}
\caption{Parametric representation mismatches in $M_1$ topology.}
\label{fig:both}
\end{figure}

By the above construction, the sup-norm of the spatial discrepancy $\norm{u-u_T}_{\infty}$ is exactly the upper bound on the remainder term of Assumption~\ref{assumption5}, $\sum_{j=1}^{\mfM_{(N)}} Z_{(N)j} \I{Q_{(N)j} > T - \Gamma_{(N)}}$, while the sup-norm of the temporal discrepancy $\norm{r-r_T}_{\infty}$ is further bounded from above by the rescaled maximal displacement time in the entire cluster $\max_{1 \le j \le K_{(N)}} Q_{(N)j}/T$; hence, it means that, overall, we have to control the probabilities 
\begin{multline*}
	\T_2 \le v'(x_T) \p{x_T^{-1} \sum_{j=1}^{\mfM_{(N)}} Z_{(N)j} \I{Q_{(N)j} > T-\Gamma_{(N)}} > \epsilon/2, D_{(N)} > 2rx_T } \\ + v'(x_T) \p{\max_{1 \le j \le K_{(N)}} \big(Q_{(N)j}/T\big) > \epsilon/2, D_{(N)} > 2rx_T } \\ \eqdef \T_{21} + \T_{22}. 
\end{multline*}

We treat term $\T_{21}$ first.

\begin{rem}
Throughout this part of the proof, we need a single random index
$$ (N)=\arg\max_{1\le i\le N_0(T)} D_i .$$
Because several indices may attain the same maximal value, we make the
choice
$$ (N)=\min\bigl\{\,i\in\{1,\dots,N_0(T)\} :
                 D_i=\max_{1\le j\le N_0(T)}D_j\bigr\},$$
i.e. we keep the smallest maximising index.
This rule produces a unique random variable, and does not affect any of the bounds that follow.  
For every collection of events $\{A_i\}_{i \ge 1}$, we have
$$\sum_{i=1}^{N_0(T)}
    \mathbf 1_{\{A_i\}}\mathbf 1_{\{(N)=i\}}
  \;\le\;
  \sum_{i=1}^{N_0(T)} \mathbf 1_{\{A_i\}},$$
so all inequalities remain valid, even if
$\p{(N)=i\mid N_0(T)}$ might not be exactly uniform
(ties slightly bias towards small indices). We proceed accordingly when $k \ge 1.$
\end{rem}

Note that, using exchangeability of the sequence $\{D_i\}_{1 \le i \le N_0(T)}$, we can consider only 
\begin{align*}
    \T_{21} &\le \E{\p{x_T^{-1} \sum_{j=1}^{\mfM_{(N)}} Z_{(N)j} \I{Q_{(N)j} > T- \Gamma_{(N)}} > \epsilon/2, D_{(N)} > 2rx_T \mid  N_0(T)}} \\
    &\le \E{\sum_{i=1}^{N_0(T)} \p{x_T^{-1} \sum_{j=1}^{\mfM_{i}} Z_{ij} \I{Q_{ij} > T- \Gamma_{i}} > \epsilon/2, D_{i} > 2rx_T, (N) = i \mid  N_0(T)}} \\
    &\le \E{N_0(T) \p{x_T^{-1} \sum_{j=1}^{\mfM_{1}} Z_{1j} \I{Q_{1j} > T- \Gamma_{1}} > \epsilon/2, D_{1} > 2rx_T \mid  N_0(T)}} \\
    &\le \E{N_0(T) \p{D_1 > 2rx_T \mid N_0(T)} \p{x_T^{-1} \sum_{j=1}^{\mfM_{1}} Z_{1j} \I{Q_{1j} > T- \Gamma_{1}} > \epsilon/2 \mid D_{1} > 2rx_T \mid  N_0(T)}}
\end{align*}




By virtue of Assumption~\ref{assumption3} to Assumption~\ref{assumption6}, applying Lemma~\ref{lem:sbj} to the second conditional probability above, it follows that 
$$\p{x_T^{-1} \sum_{j=1}^{\mfM_1} Z_{1j} \I{Q_{1j} > T- \Gamma_1} > \epsilon/2 \mid D_1 > 2rx_T, N_0(T) } = \smallO(1),\, \text{ as } T \rightarrow \infty.$$ Now, $$v'(x_T)\E{N_0(T)} \p{D_1 > 2rx_T} = (\E{N_0(T)}/T) v(x_T)\p{D_1 > 2rx_T} = \cO(1), \text{ as } T \rightarrow \infty, $$ by virtue of the regular variation assumption on $D_1$, and the uniform integrability of $N_0(T)/T.$  

Recollecting with the above, this yields $$\T_{21} = \smallO(1), \text{ as } T \rightarrow \infty.$$

Similarly, to treat term $\T_{22}$, consider (recalling that, in this case, we consider potentially all events $K_{(N)}$ in the cluster) 
\begin{multline*}
    v'(x_T) \p{\max_{1 \le j \le K_{(N)}} \big(Q_{(N)j}/T\big) > \epsilon/2,D_{(N)} > 2rx_T } \\ = v'(x_T) \E{\p{\max_{1 \le j \le K_{(N)}} \big(Q_{(N)j}/T\big) > \epsilon/2, D_{(N)} > 2rx_T \mid N_0(T)}}.
\end{multline*}

We bound the latter quantity, using exchangeability, by
\begin{align*}
    v'(x_T) & \E{\p{\max_{1 \le j \le K_{(N)}} \big(Q_{(N)j}/T\big) > \epsilon/2, D_{(N)} > 2rx_T \mid N_0(T)}} \\
    &\quad = v'(x_T) \E{\sum_{i=1}^{N_0(T)}\p{\max_{1 \le j \le K_{i}} \big(Q_{ij}/T\big) > \epsilon/2, D_{i} > 2rx_T, (N) = i \mid N_0(T)}} \\
    &\quad \le v'(x_T) \E{N_0(T)\p{\max_{1 \le j \le K_{1}} \big(Q_{1j}/T\big) > \epsilon/2, D_{1} > 2rx_T \mid N_0(T)}} \\
    &\quad \le v'(x_T) \E{N_0(T)\p{D_1 > 2rx_T \mid N_0(T)} \p{\max_{1 \le j \le K_{1}} \big(Q_{1j}/T\big) > \epsilon/2 \mid D_{1} > 2rx_T,  N_0(T)}}.
\end{align*}

Using a union bound for the second conditional probability, it holds that
\begin{align*}
    \p{\max_{1 \le j \le K_{1}}  \big(Q_{1j}/T\big) > \epsilon/2 \mid D_{1} > 2rx_T, N_0(T)} 
	&\le \E{\sum_{j=1}^{K_{1}} \p{Q_{1j}  > \epsilon T/2} \mid D_{1} > 2rx_T}\, \\
    &\le \E{x_T^{-1} \sum_{j=1}^{K_{1}} x_T\p{Q_{1j}  > \epsilon T/2} \mid D_{1} > 2rx_T}\,.
\end{align*}
By Assumption~\ref{assumption6}, for all $\delta > 0$, there exists $T_0 > T$ such that, for each $i \ge 1, j \ge 1$, $x_T\p{Q_{ij} > \epsilon T/2} < \delta$; hence, for $T > 0$ large enough, 
\begin{align*}
    \E{x_T^{-1} \sum_{j=1}^{K_{1}} x_T\p{Q_{1j} > \epsilon T/2} \mid D_{1} > 2rx_T} &\le \delta\, \E{x_T^{-1} K_1 \mid D_{1} > 2rx_T}.
\end{align*}
Now, because $K_{1}$ and $D_{1}$ are (potentially) both regularly varying, with minimal index given by $\alpha > 1$, the vector $(K_{1}, D_{1})$ is jointly regularly varying with index $\alpha > 1$ (by Theorem 1.1 in \cite{bdm02}); by Corollary 2.1.10 in \cite{ks20}, it then follows that 
$$\E{x_T^{-1} K_{1} \mid D_{1} > 2rx_T} = \cO(1), \text{ as } T \rightarrow \infty.$$ 

Now, similarly as before, by the regular variation of $D_1$ and uniform integrability of $N_0(T)/T$, $$v'(x_T) \E{N_0(T)} \p{D_1 > 2rx_T} = (\E{N_0(T)}/T) v(x_T) \p{D_1 > 2rx_T} = \cO(1), \text{ as } T \rightarrow \infty.$$

Recollecting with the above shows that 
\begin{align*}
     v'(x_T) & \p{\max_{1 \le j \le K_{(N)}} \big(Q_{(N)j}/T\big) > \epsilon/2,D_{(N)} > 2rx_T} \\ &\le \delta\, \E{x_T^{-1}K_{1} \mid D_{1} > 2rx_T} (\E{N_0(T)}/T)v(x_T)\p{D_{1} > 2rx_T} = \cO(\delta),
\end{align*}

which then yields, as $\delta \rightarrow 0$, that
$$\T_{22} = \smallO(1), \text{ as } T \rightarrow \infty.$$
This concludes the proof in the case $k=0.$

\textbf{Case of $k \ge 1$:}

We now have to show that {\footnotesize\begin{equation*} v'(x_T)^{k+1} \p{d_{M_1}\Big(x_T^{-1} \sum_{i=0}^{k} D_{(N-i)}  \I{\Gamma_{(N-i)} \le tT},\,  x_T^{-1} \sum_{i=0}^{k} \sum_{j=0}^{\mfK_{(N-i)}} Z_{(N-i)j} \I{\Gamma_{(N-i)} + W_{(N-i)j} \le tT}\Big) > \epsilon,\,  d_{M_1}(x_T^{-1}\tilde{R}_T, \D_k) > r} = \smallO(1).\end{equation*}}

We only sketch the proof in this case, since it is exactly the same as in the case $k=0$ except that we have to consider the $(k+1)$ largest clusters. Hence, first note that, as in Lemma 6.1 in \cite{dtw22}, $$\{d_{M_1}(x_T^{-1} \tilde{R}_T, \D_k) > r\} \subseteq \{D_{(N-k) }  >2rx_T\}.$$ On the spatial parts, the spirit remain the same: it is only necessary to control those events triggered by points in $[0,T]$ that occur after $T$, independently for each cluster; using Assumption~\ref{assumption5}, it means that the corresponding expression to $\T_{21}$ in the case $k=0$ is given, in the general case, by:
{\footnotesize
\begin{multline*}
   \T_{21} =  v'(x_T)^{k+1} \p{x_T^{-1} \sum_{i=0}^{k} \sum_{j=1}^{\mfM_{(N-i)}} Z_{(N-i)j} \I{Q_{(N-i)j}  > T - \Gamma_{(N-i)}} > \epsilon/2, D_{(N-k)} > 2rx_T} \\ = v'(x_T)^{k+1} \E{\p{x_T^{-1} \sum_{i=0}^{k} \sum_{j=1}^{\mfM_{(N-i)}} Z_{(N-i)j} \I{Q_{(N-i)j} > T - \Gamma_{(N-i)}} > \epsilon/2, D_{(N-k)} > 2rx_T \mid  N_0(T)}}. 
\end{multline*}
}

On the inner conditional probability, we use an union bound: 
\begin{multline*}
    \p{x_T^{-1} \sum_{i=0}^{k} \sum_{j=1}^{\mfM_{(N-i)}} Z_{(N-i)j} \I{Q_{(N-i)j} > T - \Gamma_{(N-i)}} > \epsilon/2, D_{(N-k)} > 2rx_T \mid  N_0(T)} \\ \le \sum_{i=0}^{k} \p{x_T^{-1} \sum_{j=1}^{\mfM_{(N-i)}} Z_{(N-i)j} \I{Q_{(N-i)j} > T - \Gamma_{(N-i)}} > \epsilon/(2k+2), D_{(N-k)} > 2rx_T \mid  N_0(T)}.
\end{multline*} 

Note that the full conditioning event, knowing that the clusters are exchangeable, can be written as 
\begin{align*}
    \{D_{(N-k)} > 2rx_T\} &= \{D_{(N-k)} > 2rx_T, D_{(N-k+1) }>2rx_T, \ldots, D_{(N)} > 2rx_T\} \\
    &= \{\exists 1 \le i_1 < \ldots < i_{k+1} \le N_0(T): D_{i_l} > 2rx_T\}.
\end{align*}

The number of strictly ordered $(k+1)-$tuples $(i_1 < \ldots < i_{k+1})$ that can be selected among $N_0(T)$ is $\binom{N_0(T)}{k+1}$. We proceed similarly as in the case of $k=0$, recognising that the clusters are independent, by bounding from above the quantity:
\footnotesize{
\begin{multline*}
    \sum_{i=0}^{k} \p{x_T^{-1} \sum_{j=1}^{\mfM_{(N-i)}} Z_{(N-i)j} \I{Q_{(N-i)j} > T - \Gamma_{(N-i)}} > \epsilon/(2k+2), D_{(N-k)} > 2rx_T \mid  N_0(T)} \\
    =  \sum_{i_1 < \ldots < i_{k+1}} \mathbb{P}\Bigg(x_T^{-1} \sum_{j=1}^{\mfM_{(N-i)}} Z_{(N-i)j} \I{Q_{(N-i)j} > T - \Gamma_{(N-i)}} > \epsilon/(2k+2), \\ D_{i_1} > 2rx_T, \ldots, D_{i_{k+1}} > 2rx_T, (N) = i_{k+1}, \ldots, (N-k) = i_1 \mid  N_0(T)\Bigg) \\
    \le \sum_{i_1 < \ldots < i_{k+1}} \mathbb{P}\Bigg(x_T^{-1} \sum_{j=1}^{\mfM_{i_l}} Z_{i_lj} \I{Q_{i_lj} > T - \Gamma_{i_l}} > \epsilon/(2k+2), D_{i_1} > 2rx_T, \ldots, D_{i_{k+1}} > 2rx_T \mid  N_0(T)\Bigg) \\
    \le \p{D_1 > 2rx_T \mid N_0(T)}^{k} \binom{N_0(T)}{k+1} \p{x_T^{-1} \sum_{j=1}^{\mfM_{1}} Z_{1j} \I{Q_{1j} > T - \Gamma_{1}} > \epsilon/(2k+2), D_{1} > 2rx_T \mid N_0(T)} \\
    \le \p{D_1 > 2rx_T \mid N_0(T)}^{k+1} \binom{N_0(T)}{k+1} \p{x_T^{-1} \sum_{j=1}^{\mfM_{1}} Z_{1j} \I{Q_{1j} > T - \Gamma_{1}} > \epsilon/(2k+2) \mid D_{1} > 2rx_T, N_0(T)}.
\end{multline*}
}

Now, the second conditional probability is controlled using Lemma~\ref{lem:sbj}, similarly as in the case $k=0$, which yields that $$ \p{x_T^{-1} \sum_{j=1}^{\mfM_{i_l}} Z_{i_lj} \I{Q_{i_lj} > T - \Gamma_{i_l}} > \epsilon/(2k+2) \mid D_{i_l} > 2rx_T, N_0(T)} = \smallO(1), \text{ as } T \rightarrow \infty.$$ Recollecting the combinatorial factor with the first conditional probability and the speed $v'(x_T)^{k+1}$, using the uniform integrability of $(N_0(T)/T)^{k+1}$ and the regular variation property of $D_1$ yields $$v'(x_T)^{k+1} \E{\binom{N_0(T)}{k+1}} \p{D_1 > 2rx_T}^{k+1} = \cO(1), \text{ as } T \rightarrow \infty. $$ All in all, this once again proves that $$\T_{21} = \smallO(1), \text{ as } T \rightarrow \infty.$$


The second expression to control, $\T_{22}$, is of the form

\begin{align*}
   \T_{22} &\le v'(x_T) \E{\p{\max_{0 \le i \le k}\,\max_{1 \le j \le K_{(N-i)}} \big(Q_{(N-i)j}/T\big) > \epsilon/2, D_{(N-k)} > 2rx_T \mid N_0(T) }}.
\end{align*}

Using a union bound, and strictly similar combinatorial and exchangeability arguments used to treat term $\T_{21}$ yields 

{\footnotesize
\begin{multline*}
    \p{\max_{0 \le i \le k}\,\max_{1 \le j \le K_{(N-i)}} \big(Q_{(N-i)j}/T\big) > \epsilon/2 \mid D_{(N-k)} > 2rx_T, N_0(T) } \\ \le (\E{N_0(T)}/T)^{k+1} \big(v(x_T)\p{D_1 > 2rx_T})\big)^{k+1}  \p{\max_{1 \le j \le K_{1}} \big( Q_{1j}/T \big)> \epsilon/(2k+2) \mid D_{1} > 2rx_T, N_0(T) }
\end{multline*}
}

which is negligible by the same arguments as in the case $k=0$, and overall this yields
$$\T_{22} = \smallO(1), \text{ as } T \rightarrow \infty.$$

The sketch of the proof for $k\ge 1$ is complete, which concludes the proof of the proposition. 
\end{proof}

\appendix

\section{Proof of Corollary~\ref{cor:mb}}\label{app:pcmb}

\begin{proof}

To prove the corollary, we verify Assumptions~\ref{assumption3} to~\ref{assumption5} and Assumption~\ref{assumption7} from Section~\ref{section:res}. The result then follows from Theorem~\ref{thm:clustrv}. 

\begin{enumerate}
    \item[] \textbf{Assumption 1.} The assumption that $(X_{i0}, K_{A_{i0}}) \in \text{RV}((\R_+ \setminus \{0\}) \times (\R_+ \setminus \{0\}), v(\cdot), \mu') $ implies, by Proposition 6 in \cite{bcw23}, because $\{X_{ij}\}_{j \ge 0}$ is an i.i.d. sequence, that $$\p{D_i > x} \sim \p{X_{i0} + \E{X_{i0}}K_{A_{i0}} > x} + \E{K_{A_{i0}}} \p{X_{i1} > x}, \text{ as } x \rightarrow \infty.$$ Define $g(\cdot, \cdot): \R^2_+ \rightarrow \R_+,$ by $g(x_0, k) = x_0 + \E{X} k,$ which is a positively 1-homogeneous map. By e.g. Subsection 3.2.5.2 in \cite{mw24}, this implies that $$v(x_T) \p{x_T^{-1}(X_{i0}+\E{X}K_{A_{i0}}) \in \cdot} \rightarrow \mu' \circ g^{-1}(\cdot) \eqdef \mu_g(\cdot), \text{ as } T \rightarrow \infty,$$ which shows in essence that, for each $i \ge 1$, $$v(x_T) \p{x_T^{-1} D_i \in \cdot} \rightarrow \mu(\cdot), \text{ as } T \rightarrow \infty,$$ with $\mu(\cdot) = \mu_g(\cdot) + \E{K_{A_{i0}}} \mu_{X_{i0}}(\cdot),$ where $\mu_{X_{i0}}(\cdot)$ is the projection on the first coordinate of $\mu'(\cdot)$. This essentially shows that $D_i \in \text{RV}(\R_+ \setminus \{0\}, v(\cdot), \mu).$
    \item[] \textbf{Assumption 3.} In the notations of Section~\ref{section:res}, for each $i \ge 1$, $\mfK_i \equiv K_{A_{i0}}$, the total cluster size of the $i$th cluster, and $\{W_{ij}\}_{j \ge 0}$ is the sequence of first-generation offspring waiting times, which are conditionally independent given $\cF_{FG,i} = \sigma(A_{i0})$, and conditionally independent from $K_{A_{i0}}$ given $\cF_{FG,i} = \sigma(A_{i0}).$  By the joint regular variation assumption for each $i \ge 1$ of $(X_{i0}, K_{A_{i0}})$, and because projections on coordinates are 1-homogeneous maps, $v(x_T)\p{K_{A_{i0}} > x_T} \rightarrow \mu'(\{k > 1\})$ as $T \rightarrow \infty$ for $\mu'(\cdot)$ in Assumption 1 of \textbf{(A. MB)}. Hence, $$\frac{\p{K_{A_{i0}} > x_T}}{\p{D_i > x_T}} \rightarrow \frac{\mu'(\{k > 1\})}{\mu((1, \infty))} \eqdef c \in [0,\infty)$$ as $x \rightarrow \infty$ depending on the form of $\mu'(\cdot)$ and the subspace charged by it. 
    \item[] \textbf{Assumption 4.} In notations of Section~\ref{section:check}, for each $i \ge 1, j \ge 0$, $Z_{ij} \equiv X_{ij},$ and since the transformed marks  $X_{ij} \defeq f(A_{ij})$ are i.i.d. ($f(\cdot): \A \rightarrow \R_+$ ), it follows from the joint regular variation assumption of $(X_{i0}, K_{A_{i0}})$ that, for each $i \ge 1, j \ge 1$, because projections on coordinates are 1-homogeneous, $v(x_T)\p{X_{i0} > x_T} \rightarrow \mu'(\{x_0 > 1\})$ for $\mu'(\cdot)$ in Assumption 1 of \textbf{(A. MB)}. Hence, $$\frac{\p{X_{i0} > x_T}}{\p{D_i > x_T}} \sim \frac{\p{X_{ij}> x_T}}{\p{D_i > x_T}} \rightarrow \frac{\mu'(\{k > 1\})}{\mu((1, \infty))} \eqdef c \in [0,\infty)$$ as $T \rightarrow \infty$, depending on the form of $\mu'(\cdot)$ and the subspace charged by it. 
    \item[] \textbf{Assumption~5.} 
    For each cluster $i \ge 1$, the remainder term is given by
    
    $$D_i^{>T} \defeq \sum_{j=1}^{K_{A_{i0}}} X_{ij} \I{W_{ij} > T - \Gamma_i} \text{ a.s.},$$
    where, in the notations of Section~\ref{section:res},  $\{Q_{ij}\}_{j \ge 1} \equiv \{W_{ij}\}_{j \ge 1}$ are conditionally i.i.d. given $\cF_{CI, i} \equiv \sigma(A_{i0})$, and conditionally independent from $\mfM_i \equiv K_{A_{i0}}$ given $\cF_{CI, i} \equiv \sigma(A_{i0})$; $\mfM_i \equiv K_{A_{i0}}$ is such that $\p{K_{A_{i0}} > x} \sim c_1 \p{D_i > x}, $ as $x \rightarrow \infty$ for $c_1 \in [0,\infty)$ (which is precisely what we checked in Assumption 3 above), and $\{Z_{ij}\}_{j \ge 1} \equiv \{X_{ij}\}_{j \ge 1}$ is an i.i.d. sequence of nonnegative random variables, such that $\p{Z_{ij} > x} \sim c_2 \p{D_i > x}, $ as $x \rightarrow \infty$ for $c_2 \in [0,\infty)$, precisely by what we checked in Assumption 4 above. We verify that, for each $i \ge 1$ and $j \ge 1$, $$\p{Q_{ij} > T - \Gamma_i \mid \cF_{CI,i}} = \p{W_{ij} > T - \Gamma_i \mid \cF_{CI,i}} = \smallO(1) \text{ a.s.},\, \text{ as } T \rightarrow \infty.$$



        Since $\Gamma_i \sim \text{Unif}[0,T]$, letting $T - \Gamma_i \eqdist U_{T} \sim \text{Unif}[0,T]$, we have to show, for a fixed $i \ge 1$ and for each $j \ge 1$ that $$\p{W_{ij} > U_{T} \mid \cF_{CI, i}} = \smallO(1), \text{ as } T \rightarrow \infty.$$ By Markov inequality, for each fixed $i \ge 1$ and $j \ge 1$,
        \begin{align*}
            \p{W_{ij} > U_{T} \mid \cF_{CI,i}} &= \frac{1}{T} \int_{0}^{T} \p{W_{ij} > u \mid \cF_{CI,i}} \diff u  \le \frac{1}{T} \E{W \mid \cF_{CI,i}} = \smallO(1) \text{ a.s.} , \text{ as } T \rightarrow \infty, 
        \end{align*}
        where negligibility follows because $\E{W} < \infty$ implies that $\E{W \mid \cF} < \infty$ a.s.. 

    \item[] \textbf{Assumption 7.} We finally verify that Assumption~\ref{assumption7} holds for this submodel. Recall that, for each fixed $i \ge 1$, $K_{A_{i0}}$ and $\{W_{ij}\}_{j \ge 0}$ are conditionally independent given the $\sigma-$algebra $\cF_{FG,i} =  \sigma(A_{i0})$. First, note that, for $j=0$, the indicator $\I{\Gamma_i + W_{i0} \le tT} = 1$, for the immigrant event always occur in the sum over $N_0(tT)$. We therefore consider $$ \E{\sum_{i=1}^{N_0(tT)} \sum_{j=0}^{K_{A_{i0}}} X_{ij} \I{\Gamma_i + W_{ij} \le tT}} =  \E{\sum_{i=1}^{N_0(tT)} \Big( X_{i0} +  \sum_{j=1}^{K_{A_{i0}}} X_{ij} \I{\Gamma_i + W_{ij} \le tT}\Big)}.$$ Using Campbell's Theorem (see e.g. Theorem 1.2.1 in \cite{bremaud20}), and recognising in the sum that the sequence $\{X_{ij}\}_{j \ge 1}$ is i.i.d., it holds that
        \begin{align*}
             \E{\sum_{i=1}^{N_0(tT)} \sum_{j=0}^{K_{A_{i0}}} X_{ij} \I{\Gamma_i + W_{ij} \le tT}} &= \lambda \int_0^{tT} \E{X_{i0} + \sum_{j=1}^{K_{A_{i0}}} X_{ij} \I{W_{ij} \le tT- u}} \diff u \\
             &= \lambda tT\E{X_{i0}} + \lambda \E{X_{i1}} \int_0^{tT} \E{ \E{K_{A_{i0}} \mid \cF_{FG,i}} \p{W \le tT-u \mid \cF_{FG,i}}} \diff u \\
             &= \lambda tT\E{X_{i0}} + \lambda \E{X_{i1}} \E{\E{K_{A_{i0}} \mid \cF_{FG,i}} \int_0^{tT}  \p{W \le tT-u \mid \cF_{FG,i}} \diff u} \\
             &= \lambda tT\E{X} + \lambda \E{X_{i1}} \E{\E{K_{A_{i0}} \mid \cF_{FG,i}} \int_0^{1} tT  \p{W \le tTy \mid \cF_{FG,i}} \diff y} \\
             &= \lambda tT\E{X_{i0}} + \lambda tT \E{X_{i1}} \E{\E{K_{A_{i0}} \mid \cF_{FG,i}} \p{W \le tTU \mid \cF_{FG,i}}}
        \end{align*}
        where $U \sim \text{Unif}[0,1]$, and the third equality holds by Fubini's theorem. Rewriting the quantity of Assumption~\ref{assumption7} of Section~\ref{section:res}, upon noting that, generically, because the clusters $i$ are i.i.d.,  $\E{D} \equiv \E{X} + \E{X} \E{K_{A_0}}$ therein, which holds in the mixed Binomial case, and using a triangular inequality, 
        \begin{align*}
           &\sup_{0 \le t \le 1} \Bigg\lvert x_T^{-1}\Big(N_0(tT) \big( \E{X} +  \E{X} \E{K_{A_0}} \big) - \lambda t T \big(\E{X} + \E{X} \E{\E{K_{A_0} \mid \cF_{FG}} \p{W \le tTU \mid \cF_{FG}}}\Big) \Bigg\rvert \\
           &\quad \le \sup_{0 \le t \le 1} \Bigg\lvert x_T^{-1} \Big(N_0(tT) \E{X} - \lambda tT \E{X}  \Big) \Bigg\rvert \\
           &\qquad + \sup_{0 \le t \le 1} \Bigg\lvert x_T^{-1} \Big(N_0(tT) \E{X} \E{K_{A_0}} - \lambda tT \E{X} \E{K_{A_0}}  \Big) \Bigg\rvert \\
           &\qquad + \sup_{0 \le t \le 1} \Bigg\lvert  x_T^{-1} \lambda tT \E{X} \Big(\E{K_{A_0}} - \E{\E{K_{A_0} \mid \cF_{FG,i}} \p{W \le tTU  \mid \cF_{FG}}} \Big)\Bigg\rvert \\
           &\eqdef \T_{1} + \T_{2} + \T_{3}.
        \end{align*}
        The third term $\T_{3}$ can be neglected asymptotically: one can rewrite it as 
        \begin{align*}
            \T_{3} &= \sup_{0 \le t \le 1} \Big\lvert x_T^{-1} \lambda t T \E{K_{A_0}(1-\p{W \le tTU \mid \cF_{FG}}}\Big\rvert \\
            &= \sup_{0 \le t \le 1} \Big\lvert x_T^{-1} \lambda t T \E{K_{A_0}\p{W > tTU \mid \cF_{FG}}}\Big\rvert. 
        \end{align*}
        
        This amounts to control $$\frac{\sup_{0 \le t \le 1} \E{tT \p{W > tTU \mid \cF_{FG}}}}{x_T}.$$
        For the numerator, note that one can actually consider $$\sup_{0 \le t \le 1} \E{tT \p{W > tUT \mid \cF_{FG}}} = \sup_{0 \le t \le 1} \E{tT \p{U < W/(tT) \mid \cF_{FG}}}$$ and, because $U \sim \text{Unif}(0,1)$, $tT U \eqdist U_{tT} \sim \text{Unif}(0,tT).$ This implies, by Markov, that 
        $$\p{U_{tT} < W \mid \cF_{FG} } = \frac{1}{tT} \int_{0}^{\infty} \p{W > u \mid \cF_{FG,i}} \diff u \le \frac{1}{tT} \E{W \mid \cF_{FG}} $$ which yields 
        $$\sup_{0 \le t \le 1} \E{tT \p{U_{tT} < W \mid \cF_{FG} }} \le \E{W \mid \cF_{FG}}.$$
        Since $\E{W} < \infty$, this implies $\E{W \mid \cF_{FG}} < \infty$ a.s., and it follows that
        $$\limsup_{T \rightarrow \infty} \frac{\sup_{0 \le t \le 1}\E{tT \p{W > tUT \mid \cF_{FG}}}}{x_T} = \smallO(1), \text{ as } T \rightarrow \infty. $$
        
        Because $\T_{3}$ is asymptotically negligible (in front of $x_T$) in the probability of interest, we regroup it with $\T_2$ and, by a triangular inequality, we in fact need to control the following upper bound, for each $k \ge 0$ and for all $\epsilon, r > 0$, 
        \begin{align*}
            &v'(x_T)^{k+1}\p{\T_1 + \T_2 + \T_3 > \epsilon, D_{(N-k)} > 2rx_T} \\
            &\qquad\le v'(x_T)^{k+1} \p{\T_1 > \epsilon/2, D_{(N-k)} > 2rx_T} + v'(x_T)^{k+1} \p{\T_2 + \T_3 > \epsilon/2, D_{(N-k)} > 2rx_T} \\
            &\qquad \le v'(x_T)^{k+1} \p{ \sup_{0 \le t \le 1} x_T^{-1} \Big\lvert N_0(tT) \E{X} - \lambda tT \E{X}  \Big\rvert> \epsilon/2, D_{(N-k)} > 2rx_T} \\
            &\qquad \quad + v'(x_T)^{k+1}\p{\sup_{0 \le t \le 1} \Big\lvert \big(N_0(tT) \E{X}\E{K_{A_0}} - \lambda tT \E{K_{A_0}} \E{X} \big) \Big\rvert > \epsilon x_T (1+\smallO(1))/2, \, D_{(N-k)} > 2rx_T}
        \end{align*}
        Both terms in the above upper bound are treated in the same fashion, because their first event only differs by a constant. Without loss of generality, we show that the second expression is negligible. 
        Because $N_0(tT)$ is an independent Poisson random variable, independent from $D_{(N-k)}$, $$v'(x_T)^{k+1}\p{D_{(N-k)} > 2rx_T}\p{\sup_{0 \le t \le 1} \Big\lvert \big(N_0(tT) \E{X}\E{K_{A_0}} - \lambda tT \E{K_{A_0}} \E{X} \big) \Big\rvert > \epsilon x_T(1+\smallO(1))}$$ and the first product $v'(x_T)^{k+1} \p{D_{(N-k)} > 2rx_T} = \cO(1)$ as $T  \rightarrow \infty$, due to the regular variation of $Z$ and using the properties of order statistics, see similar arguments in the proof of Proposition~\ref{proof:t2}. Hence, we are left to show, for all $\epsilon >0,$ $$\p{\sup_{0 \le t \le 1} \lvert N_0(tT) - \lambda tT \rvert > \frac{\epsilon x_T(1+\smallO(1))}{2\E{X}\E{K_{A_0}}}} = \smallO(1), \text{ as } T \rightarrow \infty.$$
        
        Let $a_T \defeq \frac{\epsilon x_T(1+\smallO(1))}{2\E{X}\E{K_{A_0}}}.$ Recognise by standard Poisson theory that $B(t) \defeq N_0(tT) - \lambda tT$ is an $\cF_{tT}^{N_0}$-martingale, where $\cF_{tT}^{N_0}$ is the natural filtration of the process $N_0(\cdot)$; it follows from Doob's inequality, and using the fact that $N_0(tT)$ is a Poisson random variable, with Jensen inequality, that $$\p{\sup_{0 \le t \le 1} |B(t)| > a_T } \le \frac{\E{|B(1)|}}{a_T} =  \frac{\E{|N_0(T)-\lambda T|}}{a_T} \le \frac{\sqrt{\lambda T}}{a_T} = \smallO(1), \text{ as } T \rightarrow \infty,$$ where final negligibility follows from the choice of $x_T$.
\end{enumerate}

Having verified all the assumptions of Section~\ref{section:res}, the proof of the corollary is complete. 
\end{proof}

\section{Proof of Corollary~\ref{cor:hawkes}}\label{app:pchawkes}

\begin{proof}
We then verify each of the assumptions from Section~\ref{section:res}. The corollary then follows by Theorem~\ref{thm:clustrv}. 
\begin{enumerate}
    \item[] \textbf{Assumption 1.} The assumption that $(X_{ij}, \kappa_{A_{ij}}) \in \text{RV}((\R_+ \setminus \{0\}) \times (\R_+ \setminus \{0\}), v(\cdot), \mu') $ implies, by Proposition 6 in \cite{bcw23}, because $\{X_{ij}\}_{j \ge 0}$ is an i.i.d. sequence, that $$\p{D_i > x} \sim \frac{1}{1-\E{\kappa_A}} \p{X_{ij} + \frac{\E{X}}{1-\E{\kappa_{A_{ij}}}} \kappa_{A_{ij}} > x },\, \text{ as } x \rightarrow \infty.$$ Define $g(\cdot, \cdot): \R^2_+ \rightarrow \R_+,$ by $g(x, k) = x + \big(\E{X}/(1-\E{\kappa_A})\big) k,$ which is a positively 1-homogeneous map. By e.g. Subsection 3.2.5.2 in \cite{mw24}, this implies that $$v(x_T)\p{x_T^{-1}\Big(X_{ij} + \frac{\E{X}}{1-\E{\kappa_{A_{ij}}}} \kappa_{A_{ij}}\Big) \in \cdot } \rightarrow \mu' \circ g^{-1}(\cdot) \eqdef \mu_g(\cdot), \text{ as } T \rightarrow \infty$$ which shows in essence that, for each $i \ge 1$, $$v(x_T) \p{x_T^{-1} D_i \in \cdot} \rightarrow \mu(\cdot), \text{ as } T \rightarrow \infty,$$ with $\mu(\cdot) = \mu_g(\cdot).$ This essentially shows that $D_i \in \text{RV}(\R_+ \setminus \{0\}, v(\cdot), \mu).$
    \item[] \textbf{Assumption 3.} For Equation~\eqref{eq:cs2}, and in the notation of Section~\ref{section:res}, for each cluster $i \ge 1$, $\mfK_i \equiv L_{A_{i0}}$, where $L_{A_{i0}}$ denotes the number of first-generation offspring associated with the immigrant mark $A_{i0}$. Note that $L_{A_{i0}} \le K_i$ a.s., where $K_i$ is the total number of events in the $i$th cluster. 

For each $i \ge 1$, the vector $(L_{A_{i0}}, \{W_{ij}\}_{j \ge 0})$ is generally dependent, but conditionally independent given $\cF_{FG,i} = \sigma(A_{i0})$. Moreover, the sequence of first-generation waiting times $\{W_{ij}\}_{j \ge 0}$ is conditionally i.i.d. given $\cF_{A_{i0}} = \sigma(A_{i0})$. Recall we set $W_{i0} \defeq 0 $ by convention for each $i \ge 1$. 

For an event $l$, the waiting time to its $j$th first-generation offspring $W_{lj}$, conditionally on the mark $A_l$ of the triggering, event admits a density with respect to Lebesgue measure given by (see e.g. \cite{mr05})
$$
\frac{h(\cdot, A_l)}{\int_0^{\infty} h(s, A_l) \, \diff s},
$$
where $h(t, A_l)$ is the fertility function defined in Subsection~\ref{section:intropp}.

By the joint regular variation assumption for each $i \ge 1, j \ge 1$ of $(X_{ij}, \kappa_{A_{ij}})$, and because projections on coordinates are 1-homogeneous maps, $v(x_T)\p{\kappa_{A_{ij}} > x} \rightarrow \mu'(\{k > 1\})$ as $T \rightarrow \infty,$ for $\mu'(\cdot)$ in Assumption 1 of \textbf{(A. H)}. By the elements in the proof of Proposition~8 of \cite{bcw23}, it further holds that $v(x_T)\p{L_{A_{ij}} > x_T} \rightarrow \mu'(\{k > 1\})$ as $T \rightarrow \infty.$ Hence, $$\frac{\p{L_{A_{ij}} > x_T}}{\p{D_i > x_T}} \rightarrow \frac{\mu'(\{k > 1\})}{\mu((1, \infty))} \eqdef c \in [0,\infty),$$ as $T \rightarrow \infty,$ depending on the form of $\mu'(\cdot)$ and the subspace charged by it. 

    \item[] \textbf{Assumption 4.} For Equation~\eqref{eq:cs2}, and in the notations of Section~\ref{section:res}, for each $i \ge 1$ and for $j = 0$, we recognise $Z_{i0} \equiv D_{i0}^{\le tT} \defeq X_{i0}$, and for each $i \ge 1, j \ge 1$, $Z_{ij} \equiv D_{ij}^{\le tT}.$ Since the transformed marks for each $i \ge 1, j \ge 0$ $X_{ij} \defeq f(A_{ij})$ are i.i.d., the sequence $\{D_{ij}^{\le tT}\}_{j \ge 0}$ is a sequence of independent elements, but is not necessarily identically distributed, because of the truncation at time $tT$; however, this suffices because, for each $i \ge 1, j \ge 1$, $D_{ij}^{\le tT} \le D_{ij} \text{ a.s.}$, the problem being circumvented in the centering term considered in Equation~\eqref{eq:cs2}. Now, for each $i \ge 1, j \ge 1,$ because the sub-sums $D_{ij}$ are self-similar, they are in particular regularly varying with the same properties as $D_i$ in Assumption 1 above. 
    \item[] \textbf{Assumption 5.} 
     In the Hawkes process, each point acts as a potential ancestor to a new, self-similar stream of events. For each cluster $i \ge 1$, the remainder term in Equation~\ref{eq:remaindert} admits a general form given by 
        $$D^{>T}_i \defeq \sum_{l=1}^{N_H(T)}\sum_{j=1}^{L_{A_{il}}} D_{ilj} \I{ \tau_{il} + W_{ilj} > T - \Gamma_i} \text{ a.s.,}$$
where $N_H(T) \defeq N_H([0,T])$ is the total number of points in the $i$th cluster up to time $T$, including all generations, $L_{A_{il}}$ is the number of first-generation offsprings of an event occurring at time $\tau_{il} \in [0,T]$ with mark $A_{il}$, with first-generation offspring waiting times given by the conditionally i.i.d. $\{W_{ilj}\}_{j\ge 1}$ given the $\sigma-$algebra generated by $A_{il}$. Noting that, for each cluster $i \ge 1$ and fixed $l \ge 1$ it is possible to write the waiting time to the $l$th event as  $\tau_{il} = \sum_{k=1}^{\cG(l)} W_{ilk}$ a.s., for an $\N-$valued random variable $\cG(l)$ giving the generation depth of the $l$th event, it is possible to define, for each $j \ge 1$, $$Q_{ilj} \defeq \tau_{il} + W_{ilj} = \sum_{k=1}^{\cG(l)} W_{ilk} + W_{ilj} \text{ a.s.,} $$ the waiting time to the $j$th first-generation offspring of the $l$th event. The sequence $\{Q_{ilj}\}_{j \ge 1}$ is conditionally i.i.d. given $\cF_{CI, i} \defeq \sigma((\tau_{il}, A_{il}): \tau_{il} \le T)$, and further conditionally independent from $L_{A_{il}}$ given $\cF_{CI, i} \defeq \sigma((\tau_{il}, A_{il}): \tau_{il} \le T)$. 

The remainder term can now be re-indexed into a single sum, fitting the form of Equation~\eqref{eq:remaindert} as needed in the proofs of the results of Section~\ref{section:res}: 

$$D^{>T}_i = \sum_{m=1}^{\mfM_i} D_{iI(m)J(m)} \I{Q_{iI(m)J(m)} > T - \Gamma_i}$$

where we define $P_0 = 0$ and $P_y := \sum_{k=1}^{y} L_{A_{ik}}$ so that:
\begin{enumerate}
    \item $I(m)$ is the unique index $y$ such that $P_{y-1} < m \le P_y$,
    \item $J(m) := m - P_{I(m)-1}$
\end{enumerate}
with $\mfM_i \defeq \sum_{l=1}^{N_H(T)} L_{A_{il}}.$ Inheriting properties from the above, the vector $(\mfM_i, \{Q_{iI(m)J(m)}\}_{m \ge 1})$ is generally dependent, but conditionally independent given the $\sigma-$algebra $\cF_{CI,i}$. For each $i \ge 1$, the sequence $\{D_{iI(m)J(m)}\}_{m \ge 1}$ inherits its i.i.d. character from $\{D_{ilj}\}_{l \ge 1, j \ge 1},$ and its elements are regularly varying with common index $\alpha > 1$ by self-similarity, from Assumption 1. For each $i \ge 1$, the sequence $\{Q_{iI(m)J(m)}\}_{m \ge 1}$ is conditionally i.i.d. given $\cF_{CI, i}. $ 

We now need to verify that 
\begin{enumerate}
    \item[] 1. for each $i \ge 1$, $\mfM_i$ is such that $\p{\mfM_i > x} \sim c\p{D_i > x},$ as $x \rightarrow \infty$, for $c \in [0,\infty)$;
    \item[] 2. for each $i \ge 1$, $m \ge 1$ the sequence of waiting times satisfy $\p{Q_{i I(m) J(m)} > T- \Gamma_i \mid \cF_{CI, i}} = \smallO(1) \text{ a.s.}, \text{ as } T \rightarrow \infty.$
\end{enumerate}

To prove 1. we use the arguments in the proof of Theorem 1 in \cite{af18}. First, let $S_n \defeq \sum_{j=1}^{n} \xi_j \defeq \sum_{j=1}^{n} (L_{A_{ij}}-1)$ and note $S_n$ is a random walk, with $\E{\xi_j} < 0$ for each $j \ge 1 $ (since $\E{L_{A_{ij}}} = \E{\kappa_{A_{ij}}} < 1$ for each $i \ge 1, j \ge 1$). Next, define $$\tau = \min\{n \ge 1: S_n <0 \} = \min\{n \ge 1: S_n = -1 \}.$$ 

As emphasised in \cite{af18}, $\tau \eqdist 1 + \sum_{j=1}^{L_A} \tau_j$ is also the total progeny size of a Galton-Watson tree rooted at 0 and giving birth to a generic $L_{A}$ number of offsprings of first generation. Note that, a.s., $L_{A_{i1}} \le \mfM_i \le \tau$. Using the arguments and notations in Theorem 1 of \cite{af18}, and using the main theorem in \cite{fz03}, it holds for each $i \ge 1$, 
\begin{align*}
   \p{L_{A_{i1}} > x_T} \le \p{\mfM_i  > x_T}  &\le \p{\max_{1 \le k \le \tau} \sum_{j=1}^k (1+r \xi_j) > x_T - r}  \le \E{\tau} \p{L_A > x_T / r}
\end{align*}
which yields in particular, since $L_{A_{i1}}$ and the generic $L_A$ are i.i.d., that $\mfM_i$ shares its properties, notably those checked in Assumption 3 above. 

We now verify 2. for each $i \ge 1$ and $m \ge 1$. For simplicity, fix an index $I(m) = l$ and $J(m) = j$.  
    Note that, because $\Gamma_i \sim \text{Unif}[0,T]$, $T- \Gamma_i \eqdef U_T \sim \text{Unif}[0,T]$. Recalling the definition of $Q_{ilj} = \sum_{k=1}^{\cG(l)} W_{ilk} + W_{ilk}$, and knowing that we work conditionally to $\cF_{CI, i}$ which includes $\tau_{il} = \sum_{k=1}^{\cG(l)} W_{ilk},$ we really need to show \begin{align*}\p{Q_{ilj} > U_T \mid \cF_{CI,i}} &= \p{W_{ilj} > U_T - \sum_{k=1}^{\cG(l)} W_{ilk} \mid \cF_{CI, i} } \\ &= \p{W_{ilj} > U_T - \tau_{il} \mid \cF_{CI,i}} \\
    &= \smallO(1) \text{ a.s., } \, \text{ as } T \rightarrow \infty.\end{align*}
    Recognise that, conditionally on $\cF_{CI,i}$, and because $\tau_{il} \in [0,T]$, we consider $U_T - \tau_{il} \sim \text{Unif}[0, T- \tau_{il}]$; by Markov inequality, \begin{align*}
        \p{W_{ilj} > U_T - \tau_{il} \mid \cF_{CI,i}} &= \frac{1}{T- \tau_{il}} \int_{0}^{T-\tau_{il}} \p{W_{ilj} > u \mid \cF_{CI, i}} \diff u  \\
        &\le  \frac{1}{T- \tau_{il}} \E{W_{ilj} \mid \cF_{CI, i}} \int_{0}^{T-\tau_{il}} \frac{1}{u} \diff u.  \end{align*}
    Because $W$ has a proper density near $u=0$, it classically follows that, for some $\epsilon >0$,  $$\p{W_{ilj} > U_T - \tau_{il} \mid \cF_{CI,i}} \le \frac{1}{T-\tau_{il}} \big(\epsilon + \E{W_{ilj} \mid \cF_{CI, i}} \log((T-\tau_{il})/\epsilon)\big) = \smallO(1) \text{ a.s.,} \, \text{ as } T \rightarrow \infty $$
    where the final step follows from the assumption that $\E{W} < \infty$ which implies $\E{W_{ilj} \mid \cF_{CI, i}} < \infty \text{ a.s..}$
    
\item[] \textbf{Assumption 6}. 
We need to show that, for each cluster $i \ge 1$, and each event $j \in \{1, \ldots, K_i\}$ (where we recall $K_i$ is the total cluster size of the $i$th cluster), it holds $$x_T \p{Q_{ij} > \epsilon T} = \smallO(1), \text{ as } T, x_T \rightarrow \infty.$$ 

Recalling that we can write for any cluster $ i \ge 1$ the waiting time to its $j$th event as $$Q_{ij} = \sum_{l=1}^{\cG(j)} W_{il}\, \text{ a.s. }$$ for some $\N-$valued generational depth random variable $\cG(j)$, letting $M > 0$ be large enough, we note~that
\begin{align*}
x_T \p{Q_{ij} > \epsilon T}
&= x_T \sum_{g=1}^{\infty} \p{Q_{ij} > \epsilon T \mid \cG(j) = g} \p{\cG(j) = g} \\
&= x_T \sum_{g=1}^{M} \p{\sum_{l=1}^{g} W_{il} > \epsilon T \mid \cG(j) = g}\p{\cG(j) = g} \\
&\quad + x_T \sum_{g=M+1}^{\infty} \p{\sum_{l=1}^{g} W_{il} > \epsilon T \mid  \cG(j) = g}\p{\cG(j) = g}.
\end{align*}

Using a union bound,  we have
\begin{align*}
x_T \sum_{g=1}^{M} & \p{\sum_{l=1}^{g} W_{il} > \epsilon T \mid \cG(j) = g}\p{\cG(j) = g} \\ &\le x_T \sum_{g=1}^M \sum_{l=1}^g \p{W_{il} > \frac{\epsilon T}{2g}}\p{\cG(j)=g} + x_T \sum_{g=M+1}^{\infty} \sum_{l=1}^g \p{W_{il} > \frac{\epsilon T}{2g}}\p{\cG(j)=g}.
\end{align*}

Because  $\{W_{il}\}_{l \ge 1}$ is an identically distributed sequence, the first sum is bounded from above by
$$x_T \sum_{g=1}^M \sum_{l=1}^g \p{W_{il} > \frac{\epsilon T}{2g}}\p{\cG(j)=g} \le x_T M \p{W>\frac{\epsilon T}{2M}},$$
which vanishes by our assumption that, for all $\epsilon > 0$, 
$$x_T\p{W>\epsilon T}=\smallO(1), \quad\text{ as }T \rightarrow \infty.$$

To control the second sum, let $\delta >0$ be small. Split further at $T^{\delta},$ recalling that for large $g$, in the case of subcritical Galton-Watson trees, $\p{\cG(j) = g} \sim C(1-\E{\kappa_A})\E{\kappa_A}^g$ for some $C > 0$, and that $\{W_{il}\}_{l \ge 1}$ are identically distributed, which yields for the second term an upper bound of the form, 
\begin{align*}
    x_T \sum_{g=M+1}^{\infty} \p{\sum_{l=1}^{g} W_{il} > \epsilon T \mid  \cG(j) = g}\p{\cG(j) = g} &\le  C(1-\E{\kappa_A})\sum_{g=M+1}^{\infty} x_T g \p{W>\frac{\epsilon T}{2g}} \E{\kappa_A}^g  
    \\ &= C(1-\E{\kappa_A})\sum_{g=M+1}^{\lfloor T^\delta\rfloor}x_T \p{W>\frac{\epsilon T}{2g}}g\E{\kappa_A}^g \\
    &\quad + C(1-\E{\kappa_A}) \sum_{g=\lfloor T^\delta\rfloor+1}^{\infty}x_T\p{W>\frac{\epsilon T}{2g}}g\E{\kappa_A}^g.
\end{align*}

The first part is bounded above by
\begin{align*}
C(1-\E{\kappa_A})\sum_{g=M+1}^{\lfloor T^\delta\rfloor}x_T \p{W>\frac{\epsilon T}{2g}}g\E{\kappa_A}^g  &\le x_T\p{W>\frac{\epsilon T^{1-\delta}}{2}}C(1-\E{\kappa_A})\sum_{g=1}^{\infty}g\E{\kappa_A}^g.
\end{align*}

Since $\E{\kappa_A}<1$, this series converges absolutely. Thus, it suffices that for sufficiently small $\delta >0$ and all $\epsilon > 0$, 

$$
x_T\p{W>\epsilon T^{1-\delta}}=\smallO(1),\, \text{ as } T \rightarrow \infty,$$ 
which is guaranteed by our Assumption 6' in \textbf{(A. H)}. 

For the second part, using properties of subcritical Galton-Watson processes, and bounding the probability by 1, is bounded from above by
\begin{align*}
    C(1-\E{\kappa_A}) \sum_{g=\lfloor T^\delta\rfloor+1}^{\infty}x_T\p{W>\frac{\epsilon T}{2g}}g\E{\kappa_A}^g &\le x_T\,C(1-\E{\kappa_A})\sum_{g=\lfloor T^\delta\rfloor+1}^{\infty}g\E{\kappa_A}^g \sim x_T \frac{T^\delta \E{\kappa_A}^{T^\delta}}{(1-\E{\kappa_A})},
\end{align*}

which vanishes exponentially fast since $\E{\kappa_A}<1$. Thus, this second term is negligible, as exponential decay dominates any polynomial (or sub-exponential) growth of $x_T$ as dictated in Assumption~\ref{assumption2}, ensuring it is $\smallO(1)$ as $T  \to \infty.$

Combining both parts yields the desired result, for each fixed $i \ge 1$ and each $j \ge 1$, all $\epsilon >0$
$$x_T\p{Q_{ij}>\epsilon T}=\smallO(1),\quad\text{ as } T\to\infty.$$

\item[] \textbf{Assumption 7.} We sketch the verification of Assumption~\ref{assumption7} for the Hawkes process, since it is very similar to the mixed Binomial case undertaken in the proof of Corollary~\ref{cor:mb}. Recall that, in the centering term of Equation~\eqref{eq:cs2}, for each $i \ge 1$, $L_{A_{i0}}$ and $\{W_{ij}\}_{j \ge 1}$ are conditionally independent given the $\sigma-$algebra $\cF_{FG,i} = \sigma(A_{i0}).$ 
    We consider the centering in Equation~\eqref{eq:cs2}: $$ \E{\sum_{i=1}^{N_0(tT)} \Big(D_{i0} +  \sum_{j=1}^{L_{A_{i0}}} D_{ij} \I{\Gamma_i + W_{ij} \le tT} \Big)} =  \E{\sum_{i=1}^{N_0(tT)} \Big( X_{i0} +  \sum_{j=1}^{L_{A_{i0}}} D_{ij} \I{\Gamma_i + W_{ij} \le tT}\Big)}.$$ Using Campbell's Theorem (see e.g. Theorem 1.2.1 in \cite{bremaud20}), and recognising in the sum that the sequence $\{D_{ij}\}_{j \ge 1}$ is i.i.d., it holds that
        \begin{align*}
             \E{\sum_{i=1}^{N_0(tT)} \Big(X_{i0} + \sum_{j=0}^{l_{A_{i0}}} D_{ij} \I{\Gamma_i + W_{ij} \le tT}\Big)} &= \lambda \int_0^{tT} \E{X_{i0} + \sum_{j=1}^{L_{A_{i0}}} D_{ij} \I{W_{ij} \le tT- u}} \diff u \\
             &= \lambda tT\E{X_{i0}} + \lambda \E{D} \int_0^{tT} \E{ \E{L_{A_{i0}} \mid \cF_{FG,i}} \p{W \le tT-u \mid \cF_{FG,i}}} \diff u \\
             &= \lambda tT\E{X_{i0}} + \lambda \E{D} \E{\E{L_{A_{i0}} \mid \cF_{FG,i}} \int_0^{tT}  \p{W \le tT-u \mid \cF_{FG,i}} \diff u} \\
             &= \lambda tT\E{X_{i0}} + \lambda \E{D} \E{\E{L_{A_{i0}} \mid \cF_{FG,i}} \int_0^{1} tT  \p{W \le tTy \mid \cF_{FG,i}} \diff y} \\
             &= \lambda tT\E{X_{i0}} + \lambda tT \E{D} \E{\E{L_{A_{i0}} \mid \cF_{FG,i}} \p{W \le tTU \mid \cF_{FG,i}}}
        \end{align*}
        where $U \sim \text{Unif}[0,1]$, and the third equality holds by Fubini's theorem. 
        Rewriting the quantity of Assumption~\ref{assumption7} of Section~\ref{section:res}, upon noting that, generically (since clusters $i$ are i.i.d.) $\E{D} \equiv \E{X_0} + \E{D} \E{L_{A_0}}$ therein, which holds in the Hawkes case, and using a triangular inequality, 
        \begin{align*}
           &\sup_{0 \le t \le 1} \Bigg\lvert x_T^{-1}\Big(N_0(tT) \big( \E{X} +  \E{D} \E{L_{A_0}} \big) - \lambda t T \big(\E{X} + \E{D} \E{\E{L_{A_0} \mid \cF_{FG}} \p{W \le tTU \mid \cF_{FG}}}\Big) \Bigg\rvert \\
           &\quad \le \sup_{0 \le t \le 1} \Bigg\lvert x_T^{-1} \Big(N_0(tT) \E{X} - \lambda tT \E{X}  \Big) \Bigg\rvert \\
           &\qquad + \sup_{0 \le t \le 1} \Bigg\lvert x_T^{-1} \Big(N_0(tT) \E{D} \E{K_{A_0}} - \lambda tT \E{D} \E{L_{A_0}}  \Big) \Bigg\rvert \\
           &\qquad + \sup_{0 \le t \le 1} \Bigg\lvert  x_T^{-1} \lambda tT \E{D} \Big(\E{L_{A_0}} - \E{\E{L_{A_0} \mid \cF_{FG,i}} \p{W \le tTU  \mid \cF_{FG}}} \Big)\Bigg\rvert
        \end{align*}
which is strictly similar as in the corresponding expression obtained in Assumption 7 in the mixed Binomial case. Hence, the negligibility, for each $k \ge 0$ and all $\epsilon, r > 0$ of $$v'(x_T)^{k+1}\p{\T_1 + \T_2 + \T_3 > \epsilon, D_{(N-k)} > 2rx_T} = \smallO(1), \text{ as } T \rightarrow \infty$$ by similar reasoning as therein. The full proof is omitted for brevity. 
\end{enumerate}
Having verified all the assumptions of Section~\ref{section:res}, the proof of the corollary is complete. 
\end{proof}

\acks 
The authors would like to express their sincere gratitude to Prof. Dr. Valérie Chavez-Demoulin for insightful discussions and valuable input throughout the development of this work, and Juraj Bodik for careful reading of an earlier version of the paper. They also extend their thanks to the Wolfgang Pauli Institute in Vienna, and its director, Prof. Dr. Norbert Mauser, for providing a welcoming and stimulating environment in which significant progress on this paper was made.

\fund 
There are no funding bodies to thank relating to this creation of this article.

\competing 
There were no competing interests to declare which arose during the preparation or publication process of this article.


\bibliography{mybib}

\end{document}